\documentclass[final,onefignum,onetabnum]{siamart250211}

\usepackage[utf8]{inputenc}
\usepackage[T1]{fontenc}

\usepackage{amsfonts}
\usepackage{amssymb}
\usepackage{mathrsfs}
\usepackage[version=4]{mhchem}
\usepackage{stmaryrd}
\SetSymbolFont{stmry}{bold}{U}{stmry}{m}{n}
\usepackage[final]{microtype}
\usepackage{xurl}
\usepackage{algorithm}
\usepackage{algpseudocode}
\usepackage{enumitem}
\usepackage{graphicx}
\usepackage{subcaption}
\usepackage{tikz}
\usetikzlibrary{fit,calc,patterns,angles,quotes,arrows.meta,positioning,decorations.pathreplacing}
\usepackage{tikz-cd}
\DeclareMathOperator{\grad}{grad}
\newcommand{\Hess}{\operatorname{Hess}}
\newcommand{\tr}{\operatorname{tr}}

\newcommand{\diag}{\operatorname{diag}}

\DeclareMathOperator{\Exp}{Exp}

\usepackage{xcolor}
\newif\ifarxivversion
\arxivversiontrue

\newsiamremark{remark}{Remark}
\newsiamthm{assumption}{Assumption}
\setlist[enumerate]{leftmargin=*}

\headers{Nystr\"om Approximation on Manifolds}{
}
\title{Nystr\"om Approximation on Manifolds}
\author{
Hantao Nie\thanks{School of Mathematical Sciences, Peking University, Beijing, China. \texttt{nht@pku.edu.cn}} 
\and Bin Gao\thanks{State Key Laboratory of Mathematical Sciences, Academy of Mathematics and Systems Science, Chinese Academy of Sciences, Beijing, China. \texttt{gaobin@lsec.cc.ac.cn}}
\and Andi Han\thanks{School of Mathematics and Statistics, University of Sydney, Sydney, Australia. \texttt{andi.han@sydney.edu.au}}
\and Pratik Jawanpuria\thanks{Centre for Machine Intelligence and Data Science, IIT Bombay, Mumbai, India. \texttt{pratik.jawanpuria@iitb.ac.in}}
\and Bamdev Mishra\thanks{Microsoft, India. \texttt{bamdevm@microsoft.com}}
\and Zaiwen Wen\thanks{Beijing International Center for Mathematical Research, Peking University, Beijing, China. \texttt{wenzw@pku.edu.cn}}
}
\date{\today}
\begin{document}

\maketitle
\raggedbottom

\begin{abstract}
    Computations on a manifold often involve constructing an operator on the tangent space and computing its inverse, which can be time-consuming in many applications. In order to reduce the computational costs and preserve the benign properties of tangent operators, we develop the Riemannian Nystr\"om approximation on manifolds, a low-rank approximation of tangent operators through subspace projections onto the tangent space. 
    The developed approximation is intrinsically constructed and inherits desirable properties from the classical Nystr\"om approximation, e.g., positive semidefiniteness and approximation errors. Instead of the Gaussian sketching, we introduce the Haar--Grassmann sketching condition with a coordinate-free representation, which remains compatible under isometric vector transport across tangent spaces. Moreover, we propose a randomized Newton-type method for optimization on manifolds in which the linear system is constructed via the Riemannian Nystr\"om approximation. Numerical experiments on the SPD and Grassmann manifolds, together with principal geodesic analysis on real data, illustrate that the proposed approximation reduces the computational cost of operators while maintaining comparable accuracy.
  \end{abstract}

\section{Introduction}

Many large-scale problems in machine learning~\cite{bronstein2021geometric} and signal processing~\cite{zhu2018image} are naturally posed with manifold constraints, making geometry-aware algorithms increasingly important.
While the manifold structure enables faithful modeling of constraints and geometry, the associated linear algebra on the tangent space often becomes a computational bottleneck in high-dimensional scenarios, especially within iterative methods; e.g.,~\cite{AbsilBakerGallivan2007}. Specifically, given a Riemannian manifold $(\mathcal{M},g)$, we consider a tangent-space operator
\[
\mathcal{H}_x : \mathrm{T}_x\mathcal M \to \mathrm{T}_x\mathcal M, \quad x \in \mathcal{M},
\]
which is self-adjoint and positive semidefinite (PSD) with respect to the metric $g$.

One representative example involving tangent-space operators is solving linear systems on manifolds, where $\mathcal{H}_x$ encodes second-order information from objective functions, e.g., the Riemannian Hessian operator~\cite{absil2008optimization, boumal2023introduction}. 
Another example adopting tangent-space operators is principal geodesic analysis (PGA)~\cite{fletcher2004principal}, which extracts the dominant spectral structure from covariance-type objects on curved geometries, and $\mathcal{H}_x$ is formulated from a large covariance tensor. For instance, data in medical imaging and computer vision are frequently modeled as points on symmetric positive definite (SPD) manifolds~\cite{Pennec2006}, rendering $\mathcal{H}_x$ an operator on the tangent space of SPD manifolds. These applications call for computing a tangent-space operator and its pseudoinverse, but the explicit formulation of operators is often prohibitive or intractable, underscoring the need for efficient approximation of the operators.

A variety of memory- and computation-efficient techniques for problems involving tangent-space operators have been proposed in the literature. Specifically, spectral truncation on manifolds~\cite{ValletLevy2008_ManifoldHarmonics} approximates an operator by restricting it to a space spanned by the leading eigenvectors, reducing operator inversion or calculus to computations on the truncated coefficients.
For linear systems induced by manifold elliptic operators (e.g., discretized Laplace--Beltrami equations), multigrid and multilevel solvers~\cite{BonitoPasciak2012} achieve fast convergence by combining smoothing with coarse-grid correction across a hierarchy of discretizations, yielding a scalable approximate inverse. 
In optimization, a common alternative to Newton-type methods on manifolds is a Riemannian quasi-Newton method (e.g., Riemannian BFGS and limited-memory variants~\cite{HuangAbsilGallivan2016_RBFGS}), which maintains a secant approximation to the Hessian operator (or its inverse) during iterations. 
More recently, randomized subspace strategies~\cite{gutman2023coordinate,han2024riemannian} were proposed for high-dimensional manifold computations, restricting each iteration to a low-dimensional tangent subspace. In summary, these techniques avoid forming, storing, and inverting the full operator, and instead exploit problem-specific low-dimensional structures, which prompts a natural question: can one construct an efficient approximation for a tangent-space operator on manifolds with both provable approximation errors and fast computations?

Our goal is therefore to construct an operator $\widehat{\mathcal{H}}_x: \mathrm{T}_x\mathcal M \to \mathrm{T}_x\mathcal M$ on the tangent space of a manifold $\mathcal{M}$ that approximates $\mathcal{H}_x$ with a controlled approximation error:
\begin{equation*}
\|\mathcal{H}_x-\widehat{\mathcal{H}}_x\|_{\mathrm{op}} \leq \phi(\lambda(\mathcal{H}_x)),
\end{equation*}
where $\lambda(\mathcal{H}_x)$ denotes the spectrum of $\mathcal{H}_x$ and $\phi(\lambda(\mathcal{H}_x))$ is a function of $\lambda(\mathcal{H}_x)$. When the construction involves randomization, the bound is interpreted in expectation.
Additionally, $\widehat{\mathcal{H}}_x$ should preserve basic operator properties of $\mathcal{H}_x$, such as self-adjointness and positive semidefiniteness.
It is also desirable to control the perturbation of inverse operators, e.g., 
\begin{equation*}
\|\mathcal{H}_x^{\dagger}-\widehat{\mathcal{H}}_x^{\dagger}\|_{\mathrm{op}}\ \quad \mathrm{or}\ \quad 
\big\|(\mathcal{H}_x + \nu\,\mathrm{Id}_x)^{-1}-(\widehat{\mathcal{H}}_x + \nu\,\mathrm{Id}_x)^{-1}\big\|_{\mathrm{op}}
\end{equation*}
with bounded errors. 

The Nystr\"om approximation~\cite{GittensMahoney2016, martinsson2020randomized} is a powerful technique in large-scale matrix computations.
Given a symmetric positive definite matrix $\mathbf{H}\in\mathbb R^{d\times d}$, the Nystr\"om approximation is defined by 
\begin{equation}
    \label{eq:Euclidean-nystrom}
\widehat{\mathbf{H}} = (\mathbf{H} \mathbf{P}) (\mathbf{P}^\top \mathbf{H} \mathbf{P})^{\dagger} (\mathbf{H} \mathbf{P})^\top,
\end{equation}
where $\ell \ll d$ is the sketch size, $\mathbf{P} \in \mathbb{R}^{d \times \ell}$ is a sketching matrix, e.g., a selection matrix $[e_{i_1},\dots,e_{i_\ell}]\in\mathbb R^{d\times \ell}$ that extracts \(\ell\) coordinate directions, or a random Gaussian matrix, and $\dagger$ denotes the pseudoinverse. A rich theory underpins the success of Nystr\"om approximations. Specifically, the Nystr\"om approximation can be viewed as a low-rank approximation for a matrix while preserving positive semidefiniteness, and is closely connected to eigendecompositions~\cite{martinsson2020randomized}.
Worst-case and probabilistic approximation errors are established under randomized column sampling and projection, e.g., approximation errors in spectral and Frobenius norms~\cite{ZhangTsangKwok2008, Gittens2011,GittensMahoney2016}. 
Beyond matrix approximation, Nystr\"om approximations are widely used in numerical algorithms for acceleration purposes, such as kernel methods~\cite{WilliamsSeeger2001,DrineasMahoney2005,bucci2025numerical}, preconditioned conjugate gradient methods~\cite{Frangella2023} and interior-point proximal methods of multipliers~\cite{chu2026randomized}.
In kernel ridge regression and other kernel learning problems, properly constructed Nystr\"om approximations can match the predictive performance of the full kernel method, while requiring only a limited number of sampled features~\cite{Bach2013,AlaouiMahoney2015}.
More broadly, Nystr\"om theory has been extended to infinite-dimensional Hilbert spaces for non-negative self-adjoint operators~\cite{PerssonBoulleKressner2024}.

On Riemannian manifolds, extending Nystr\"om approximations from matrices to tangent-space operators requires additional care, where the key is to construct a sketching operator from a tangent space to its low-dimensional subspace. The standard Gaussian sketching is usually formulated via Gaussian matrices expressed in a coordinate basis, but there is no canonical coordinate system on manifolds. Hence, 
it is desirable to define both the sketching operators and the Nystr\"om approximation in a coordinate-free manner.
Moreover, in iterative methods, sketching is expected to be efficiently transported between successive tangent spaces via vector transport, which motivates a sketching condition formulated invariant under transport.
To this end, we intend to develop a Riemannian Nystr\"om approximation on manifolds for self-adjoint PSD tangent-space operators, along with a coordinate-free and transport-compatible sketching.

\subsection{Main contributions}
In this paper, we develop a Riemannian Nystr\"om approximation for self-adjoint PSD tangent-space operators on the $d$-dimensional Riemannian manifold $(\mathcal{M}, g)$.
Our main contributions are as follows.

We propose a Riemannian Nystr\"om approximation for self-adjoint PSD operators on the tangent space.
Specifically, let $x \in \mathcal{M}$ and let $\mathcal{H}_x : \mathrm{T}_x\mathcal{M} \to \mathrm{T}_x\mathcal{M}$ be a $g_x$-self-adjoint PSD operator. The Riemannian Nystr\"om approximation of the operator $\mathcal{H}_x$ is defined by
\[
\mathcal{H}_{x, B, \Xi}[u]=\left(\mathcal{H}_x \mathcal{P}_{x, B, \Xi}\left(\mathcal{P}_{x, B, \Xi}^{*} \mathcal{H}_x \mathcal{P}_{x, B, \Xi}\right)^{\dagger} \mathcal{P}_{x, B, \Xi}^{*} \mathcal{H}_x\right)[u], \quad  \text{for all } u \in \mathrm{T}_x \mathcal{M},
\]
where the sketching operator $\mathcal{P}_{x, B, \Xi}$ and its adjoint $\mathcal{P}_{x, B, \Xi}^*$ map onto the $\ell$-dimensional subspaces $\Xi$ and $B$ of $\mathrm{T}_x\mathcal{M}$, respectively. The construction of the sketching operator and the approximation does not rely on the explicit representation of the basis or coordinates.
We establish basic properties analogous to the Euclidean case, including positive semidefiniteness and properties on the range.

To enable randomized error analysis, we introduce a Haar--Grassmann sketching condition, which generalizes the standard Gaussian sketching in the Euclidean setting while remaining compatible with the intrinsic geometry of the manifold.
Under this condition, we establish the approximation error for the Riemannian Nystr\"om approximation as follows:
  \[
  \mathbb E\left[ \big\|\mathcal{H}_x-\mathcal{H}_{x, B, \Xi}\big\|_{\mathrm{op}} \right]
  \le
  \min_{2\le p\le \ell-2}
\Bigg\{\left(1+\tfrac{C_1(\ell-p)}{p-1}\right)\lambda_{\ell-p+1}(\mathcal{H}_x) +\tfrac{C_2\ell}{p^{2}-1}\sum_{j>\ell-p}\lambda_j(\mathcal{H}_x)
  \Bigg\}.
  \]
where $\lambda_{1}(\mathcal{H}_x), \lambda_{2}(\mathcal{H}_x), \ldots, \lambda_{d}(\mathcal{H}_x)$ are eigenvalues of $\mathcal{H}_x$ and $C_1, C_2$ are constants; see~\cref{thm:manifold-22-haar}.
Moreover, the Haar--Grassmann condition is proved to be transport-compatible in the sense that, under an isometric vector transport, the transported sketching still satisfies the Haar--Grassmann condition.

In a coordinate description, we provide efficient computations for the Riemannian Nystr\"om approximation and its pseudoinverse. As an application, we consider optimization problems on a manifold, and propose a randomized Nystr\"om Riemannian Newton-type method, in which the linear system is solved via the Riemannian Nystr\"om approximation.
Numerical experiments on the SPD and Grassmann manifolds demonstrate that the proposed approximation reduces the computational cost of tangent-space operators while maintaining comparable accuracy. In addition, experiments on principal geodesic analysis with real data illustrate reduced memory usage while preserving competitive statistical performance.

\subsection{Organization}
The paper is organized as follows.
Section~\ref{sec:sketching} develops the Riemannian Nystr\"om approximation for tangent-space operators on Riemannian manifolds and establishes its basic properties and approximation errors.
Section~\ref{sec:computations} presents coordinate representations and computational formulas for the Riemannian Nystr\"om approximation.
Section~\ref{sec:riemannian-newton} introduces the randomized Nystr\"om Newton-type method on manifolds as one application of the Riemannian Nystr\"om approximation.
Section~\ref{sec:experiments} reports numerical experiments demonstrating the approximation properties and the effectiveness of the proposed optimization method.
Section~\ref{sec:conclusion} concludes the paper and discusses future work.

\subsection{Notation}
  Let $(\mathcal M,g)$ be a Riemannian manifold and $x\in\mathcal M$.
  The metric \(g_x\) induces an inner product \(\langle\cdot,\cdot\rangle_x\) on \(\mathrm{T}_x\mathcal M\). For \(v\in \mathrm{T}_x\mathcal M\), the vector norm is defined by
$
\|v\|_x:=\sqrt{\langle v,v\rangle_x}.
$
For a linear map or operator $\mathcal{L}:\mathrm{T}_x\mathcal M\to \mathrm{T}_x\mathcal M$, its range is the subspace $\mathrm{range}(\mathcal{L}) := \{\mathcal{L}[u]: u\in \mathrm{T}_x\mathcal M\}$ and its rank is the dimension of the range.
The $g_x$-adjoint of $\mathcal{L}$ is the unique linear map $\mathcal{L}^{*}$ satisfying
$
\langle \mathcal{L}[u],v\rangle_x=\langle u,\mathcal{L}^{*}[v]\rangle_x
$
for all $u,v\in \mathrm{T}_x\mathcal M$.
  We say that $\mathcal{L}$ is {$g_x$-self-adjoint} (or simply self-adjoint if $g_x$ is clear from the context) if $\mathcal{L}=\mathcal{L}^{*}$, or equivalently, in any $g_x$-orthonormal frame $\{b_{x, j}\}_{j=1}^d$, the matrix representation of $\mathcal{L}$ is symmetric.
  Moreover, $\mathcal{L}\succeq 0$ if and only if $\langle u,\mathcal{L}[u]\rangle_x\ge 0$ for all $u\in \mathrm{T}_x\mathcal M$.
  For two operators $\mathcal{L}_1$ and $\mathcal{L}_2$, we denote the Loewner order $\mathcal{L}_1 \succeq \mathcal{L}_2$ if $\mathcal{L}_1-\mathcal{L}_2 \succeq 0$.
  The Moore--Penrose pseudoinverse of $\mathcal{L}$ is denoted by $\mathcal{L}^{\dagger}$.
The operator norm of \(\mathcal{L}\) induced by \(g_x\) is
$
\|\mathcal{L}\|_{\mathrm{op}}:=\sup_{v\neq 0}\frac{\|\mathcal{L}[v]\|_x}{\|v\|_x}=\sup_{\|v\|_x=1}\|\mathcal{L}[v]\|_x.
$
We also use the Hilbert--Schmidt norm as
$
\|\mathcal{L}\|_{\mathrm{HS}}^{2}:=\operatorname{tr}(\mathcal{L}^{*}\mathcal{L}).
$
For any linear subspace \(V\subset \mathrm T_x\mathcal M\), define its \(g_x\)-orthogonal complement by
$
V^{\perp}:=\{w\in \mathrm T_x\mathcal M:\langle w,v\rangle_x=0,\ \forall v\in V\}.
$
The \(g_x\)-orthogonal projection onto \(V\) is the unique linear map
$
\Pi_V:\mathrm T_x\mathcal M\to V
$
such that, for every \(u\in \mathrm T_x\mathcal M\),
$
\Pi_V[u]\in V,
u-\Pi_V[u]\in V^{\perp}.
$
For a smooth function $f:\mathcal M\to\mathbb R$, the Riemannian gradient $\grad f(x)\in \mathrm{T}_x\mathcal M$ is defined by $\langle \grad f(x),u\rangle_x=\mathrm D f(x)[u]$ for all $u\in \mathrm{T}_x\mathcal M$. The Riemannian Hessian $\Hess f(x):\mathrm{T}_x\mathcal M\to \mathrm{T}_x\mathcal M$ is the $g_x$-self-adjoint operator defined by $\Hess f(x)[u]=\nabla_u(\grad f)$, where $\nabla$ is the Levi-Civita connection. 
For two random variables or vectors $X$ and $Y$, we write $X\stackrel{d}{=}Y$ if they are equal in distribution. For two inner-product spaces $V$ and $W$ of the same dimension, we denote by $\mathrm{Iso}(V,W)$ the set of all linear isometries from $V$ to $W$. The Grassmann manifold $\mathrm{Gr}(\ell,V)$ is the set of all $\ell$-dimensional subspaces of a vector space $V$.

\section{Riemannian Nystr\"om approximation and its properties}
\label{sec:sketching}
This section introduces the coordinate-free Riemannian Nystr\"om approximation and then discusses its basic properties, approximation errors under randomized sketching, and transport compatibility under isometric vector transport.

\subsection{Riemannian Nystr\"om approximation}
\label{subsec:Nystrom-manifold}
Let $(\mathcal{M}, g)$ be a $d$-dimensional Riemannian manifold and $x\in\mathcal M$.
Let $B, \Xi\subset \mathrm T_x\mathcal M$ be two $\ell$-dimensional subspaces with $\ell \le d$.
 We specify a full-rank linear map
\begin{equation*}
  \label{eq:F}
\mathcal{F}:B\to \Xi.
\end{equation*}
Let $\Pi_{B}:\mathrm T_x\mathcal M\to B$ and $\Pi_{\Xi}:\mathrm T_x\mathcal M\to\Xi$ denote the $g_x$-orthogonal projections.
We define the sketching operator by
\begin{equation}
  \label{eq:Px-coordfree}
\mathcal P_{x,B,\Xi}\colon \mathrm{T}_x\mathcal{M} \to \mathrm{T}_x\mathcal{M}, \qquad 
v \mapsto \mathcal{F}\,\Pi_{B}[v].
\end{equation}
Note that $\mathrm{range}(\mathcal{P}_{x,B,\Xi})=\Xi$, and hence $\mathcal{P}_{x,B,\Xi}$ is a map from $\mathrm{T}_x\mathcal{M}$ to an $\ell$-dimensional subspace $\Xi$.
Let $\mathcal F^*:\Xi\to B$ be the $g_x$-adjoint of $\mathcal{F}$. It follows from $\mathrm{range}(\mathcal{F}\Pi_B)=\Xi$ and $\mathrm{range}(\mathcal{F}^*\Pi_\Xi)=B$ that 
\[
\langle \mathcal F\Pi_B[v],u\rangle_x
=\langle \mathcal F\Pi_B[v],\Pi_\Xi[u]\rangle_x
=\langle \Pi_B[v],\mathcal F^*\Pi_\Xi[u]\rangle_x
=\langle v,\mathcal F^*\Pi_\Xi[u]\rangle_x
\]
for any \(u,v\in \mathrm T_x\mathcal M\). Hence, the adjoint of $\mathcal{P}_{x,B,\Xi}$ is given by
\begin{equation}
  \label{eq:Px*-coordfree}
 \mathcal P^{*}_{x,B,\Xi} \colon \mathrm{T}_x\mathcal{M} \to \mathrm{T}_x\mathcal{M}, \qquad 
 u \mapsto \mathcal{F}^{*}\,\Pi_{\Xi}[u].
\end{equation}

In the construction of sketching operators, $\Pi_B$ compresses the full tangent space to an $\ell$-dimensional subspace. $\Xi$ is the image space of sketching, namely the $\ell$-dimensional subspace into which the sketching operator $\mathcal{P}_{x, B, \Xi}$ maps. When $B = \Xi$ and $\mathcal{F} = \mathrm{Id}_B$, the sketching operator is simply a subspace projection. The introduction of
$\mathcal{F}$ allows constructions beyond simple subspace projection. In general,
$\mathcal{F}$ serves as the transfer map between the spaces $B$ and $\Xi$. This extra degree of freedom makes it possible to reweight or rotate the compressed information and encode randomness, thereby covering a broader class of sketching operators and potentially improving approximation quality; see examples in Gaussian sketching~\eqref{eq:gaussian-sketching-condition} and Haar--Grassmann sketching in~\cref{def:haar-grassmann-sketch}.

When $\mathcal M=\mathbb R^d$ with the Euclidean metric, if one chooses a full-row-rank sketching matrix $\mathbf{P} \in\mathbb R^{d\times \ell}$ in~\eqref{eq:Euclidean-nystrom}, then the subspace
$B=\operatorname{range} (\mathbf{P}^\top) = \mathbb R^{\ell}$. $\Xi\subset \mathbb R^d$ is the subspace generated by the directions after applying the sketching map, namely
$\Xi=\operatorname{range}(\mathbf{P})$.
The map $\mathcal F$ is the intrinsic counterpart of the Euclidean sketching matrix restricted to the subspace $B$. Once a basis is fixed, $\mathcal{P}_{x, B, \Xi} = \mathcal{F} \Pi_B$ is represented by the sketching matrix $\mathbf{P}$. 

With the sketching operator and its adjoint defined above, the Riemannian Nystr\"om approximation is defined as follows.
\begin{definition}[Riemannian Nystr\"om approximation]
\label{def:Riemannian-nystrom}
  Let $x\in\mathcal M$ and $\mathcal{H}_x:\mathrm T_x\mathcal M\to\mathrm T_x\mathcal M$ be a $g_x$-self-adjoint PSD operator.
  The Riemannian Nystr\"om approximation of $\mathcal{H}_x$ is the operator $\mathcal{H}_{x,B,\Xi}:\mathrm T_x\mathcal M\to\mathrm T_x\mathcal M$ defined by
\begin{equation}
  \label{eq:Nystrom-Hessian}
\mathcal{H}_{x, B, \Xi}[u]=\left(\mathcal{H}_x \mathcal{P}_{x, B, \Xi}\left(\mathcal{P}_{x, B, \Xi}^{*} \mathcal{H}_x \mathcal{P}_{x, B, \Xi}\right)^{\dagger} \mathcal{P}_{x, B, \Xi}^{*} \mathcal{H}_x\right)[u], \quad \text{for all } u \in \mathrm{T}_x \mathcal{M}.
\end{equation}
  The dimension $\ell=\dim(B)$ is called the sketch size.
\end{definition}

An illustration of the sketching operator and the Riemannian Nystr\"om approximation is shown in~\cref{fig:sketching-and-nystrom}. 
Notice that the sketching operator and its adjoint map the tangent space $\mathrm{T}_x\mathcal{M}$ to low-dimensional subspaces, i.e., $\mathrm{range}(\mathcal{P}_{x,B,\Xi})=\Xi$ and $\mathrm{range}(\mathcal{P}_{x,B,\Xi}^{*})=B$. In Riemannian Nystr\"om approximation, the linear system (the pseudoinverse term) is compressed into the low-dimensional subspace $B$ via $\mathcal{P}_{x, B, \Xi}^* \mathcal{H}_x$. After the computation of the pseudoinverse, $\mathcal{H}_x \mathcal{P}_{x, B, \Xi}$ lifts the information back to the original tangent space. Hence, $\mathcal{H}_{x, B, \Xi}$ can be viewed as a low-rank approximation of $\mathcal{H}_x$, which can save storage and computational cost in practice.

The Moore--Penrose pseudoinverse of $\mathcal{H}_{x, B, \Xi}$ can be computed as
\begin{equation}
  \label{eq:Nystrom-Hessian-inverse}
\mathcal{H}_{x, B, \Xi}^{\dagger}[u]=\left(\mathcal{P}_{x, B, \Xi}\left(\mathcal{P}_{x, B, \Xi}^{*} \mathcal{H}_x \mathcal{P}_{x, B, \Xi}\right)^{\dagger} \mathcal{P}_{x, B, \Xi}^{*}\right)[u], \quad \text{for all } u \in \mathrm{T}_x \mathcal{M}.
\end{equation}
It can be verified that $\mathcal{H}_{x, B, \Xi}^{\dagger} \mathcal{H}_{x, B, \Xi} \mathcal{H}_{x, B, \Xi}^{\dagger} = \mathcal{H}_{x, B, \Xi}^{\dagger}$, $\mathcal{H}_{x, B, \Xi} \mathcal{H}_{x, B, \Xi}^{\dagger} \mathcal{H}_{x, B, \Xi} = \mathcal{H}_{x, B, \Xi}$, $(\mathcal{H}_{x, B, \Xi} \mathcal{H}_{x, B, \Xi}^{\dagger})^* = \mathcal{H}_{x, B, \Xi} \mathcal{H}_{x, B, \Xi}^{\dagger}$ and $(\mathcal{H}_{x, B, \Xi}^{\dagger}\mathcal{H}_{x, B, \Xi} )^* = \mathcal{H}_{x, B, \Xi}^{\dagger}\mathcal{H}_{x, B, \Xi}$.

\definecolor{ForwardBlue}{RGB}{0,114,178}
\definecolor{AdjointRed}{RGB}{160,70,130}
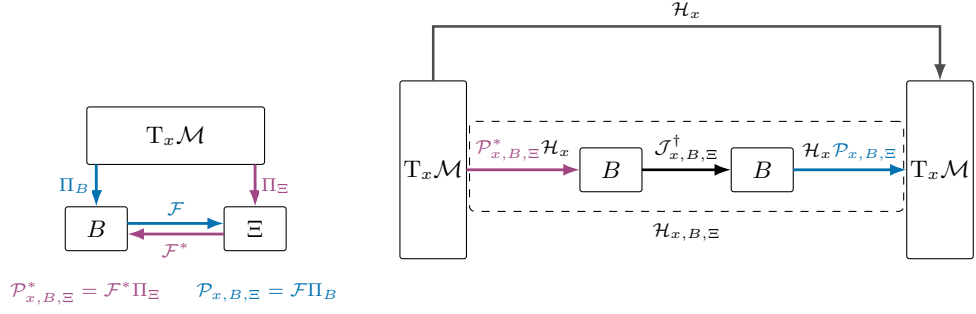
\begin{figure}[htbp]
\centering

\begin{subfigure}[t]{0.35\textwidth}
\centering
\begin{tikzpicture}[
  font=\small,
  >={Latex[length=2.2mm,width=1.6mm]},
  txobj/.style={
    draw,
    rounded corners=1.0pt,
    minimum width=2.35cm,
    minimum height=0.76cm,
    inner sep=1.3pt,
    outer sep=0pt,
    fill=white,
    align=center
  },
  subobj/.style={
    draw,
    rounded corners=1.0pt,
    minimum width=0.82cm,
    minimum height=0.58cm,
    inner sep=1.2pt,
    outer sep=0pt,
    fill=white,
    align=center
  },
  pblue/.style={
    ->,
    line width=1.05pt,
    draw=ForwardBlue
  },
  pred/.style={
    ->,
    line width=1.05pt,
    draw=AdjointRed
  },
  lab/.style={
    font=\scriptsize,
    fill=white,
    inner sep=0.8pt
  }
]

\node[txobj] (Tx) at (0,1.95) {$\mathrm{T}_x\mathcal{M}$};
\node[subobj] (B)  at (-1.05,0.72) {$B$};
\node[subobj] (Xi) at ( 1.05,0.72) {$\Xi$};

\coordinate (PB0)  at ($(Tx.south)+(-1.05,0)$);
\coordinate (PXi0) at ($(Tx.south)+( 1.05,0)$);

\draw[pblue]
  (PB0) -- node[lab, left=1pt, text=ForwardBlue] {$\Pi_B$} (B.north);

\draw[pred]
  (PXi0) -- node[lab, right=1pt, text=AdjointRed] {$\Pi_\Xi$} (Xi.north);

\draw[pblue]
  ([yshift=2.0pt]B.east) --
  node[lab, above=2pt, text=ForwardBlue] {$\mathcal F$}
  ([yshift=2.0pt]Xi.west);

\draw[pred]
  ([yshift=-2.0pt]Xi.west) --
  node[lab, below=2pt, text=AdjointRed] {$\mathcal F^{*}$}
  ([yshift=-2.0pt]B.east);

\end{tikzpicture}

\vspace{2pt}
{\scriptsize\color{AdjointRed}$\mathcal P_{x,B,\Xi}^{*}=\mathcal F^{*}\Pi_\Xi$}
\quad 
{\scriptsize\color{ForwardBlue}$\mathcal P_{x,B,\Xi}=\mathcal F\Pi_B$}

\label{fig:sketching-operator-colored}
\end{subfigure}
\hspace{0.1\textwidth}
\begin{subfigure}[t]{0.53\textwidth}
\centering
\begin{tikzpicture}[
  trim left=0pt,
  trim right=28pt,
  font=\small,
  >={Latex[length=2.2mm,width=1.6mm]},
  wideobj/.style={
    draw,
    rounded corners=1.0pt,
    minimum width=0.76cm,
    minimum height=2.35cm,
    inner sep=1.3pt,
    outer sep=0pt,
    fill=white,
    align=center
  },
  subobj/.style={
    draw,
    rounded corners=1.0pt,
    minimum width=0.82cm,
    minimum height=0.58cm,
    inner sep=1.2pt,
    outer sep=0pt,
    fill=white,
    align=center
  },
  pblue/.style={
    ->,
    line width=1.05pt,
    draw=ForwardBlue
  },
  pred/.style={
    ->,
    line width=1.05pt,
    draw=AdjointRed
  },
  pblack/.style={
    ->,
    line width=1.0pt,
    draw=black
  },
  pgray/.style={
    ->,
    line width=0.95pt,
    draw=black!70
  },
  lab/.style={
    font=\scriptsize,
    fill=white,
    inner sep=0.8pt
  },
  longlab/.style={
    font=\scriptsize,
    fill=white,
    inner sep=0.6pt,
    align=center
  }
]

\node[wideobj] (TxL) at (-3.35,0) {$\mathrm{T}_x\mathcal{M}$};
\node[subobj]  (B1)  at (-1,0) {$B$};
\node[subobj]  (B2)  at ( 1,0) {$B$};
\node[wideobj] (TxR) at ( 3.35,0) {$\mathrm{T}_x\mathcal{M}$};

\draw[pgray]
  (TxL.north) -- ++(0,0.72)
  -- node[lab, above=2pt] {$\mathcal H_x$} ($(TxR.north)+(0,0.72)$)
  -- (TxR.north);

\draw[pred]
  (TxL) -- node[lab, above=2pt]
  {$\textcolor{AdjointRed}{\mathcal P_{x,B,\Xi}^{*}}\mathcal H_x$} (B1);

\draw[pblack]
  (B1) -- node[longlab, above=2pt] {$\mathcal J_{x,B,\Xi}^{\dagger}$} (B2);

\draw[pblue]
  (B2) -- node[lab, above=2pt]
  {$\mathcal H_x\textcolor{ForwardBlue}{\mathcal P_{x,B,\Xi}}$} (TxR);

\coordinate (boxNW) at (-2.87,0.64);
\coordinate (boxSE) at ( 2.87,-0.55);
\node[draw,dashed,rounded corners=2pt,fit=(boxNW)(boxSE),inner sep=0pt] (corebox) {};

\node[font=\scriptsize, fill=white, inner sep=1pt] at (0,-0.82)
{$\mathcal H_{x,B,\Xi}$};
\path[use as bounding box] (-3.8,-1.05) rectangle (3.45,1.45);

\end{tikzpicture}

\label{fig:nystrom-flow-colored}
\end{subfigure}

\caption{Left: construction of the sketching operator and its adjoint. Right: flow view of the Riemannian Nystr\"om approximation, where $\mathcal{J}_{x,B,\Xi} = \mathcal{P}_{x,B,\Xi}^* \mathcal{H}_x \mathcal{P}_{x,B,\Xi}$. The blue arrows correspond to the forward sketching map $\mathcal P_{x,B,\Xi}$, while the red arrows correspond to its adjoint $\mathcal P_{x,B,\Xi}^{*}$.}
\label{fig:sketching-and-nystrom}
\end{figure}

\subsection{Basic properties}
The following results establish basic properties of the Riemannian Nystr\"om approximation, which are inherited from the Euclidean Nystr\"om approximation~\cite{martinsson2020randomized}. In this section, we present these properties for the fixed sketching operator, and in the next section we analyze approximation guarantees under randomized sketching.

The following proposition verifies the semidefiniteness and low-rankness of the Riemannian Nystr\"om approximation, which further implies that $\mathcal{H}_{x, B, \Xi}$ is a low-rank approximation of $\mathcal{H}_x$.
\begin{proposition}
  \label{lem:riem-21}
  For $x\in\mathcal M$, the following statements hold.
\begin{enumerate}
\item[\textup{(i)}] It holds that $0 \preceq \mathcal{H}_{x, B, \Xi} \preceq \mathcal{H}_x$ and $\operatorname{rank}(\mathcal{H}_{x, B, \Xi})\le \ell$.

\item[\textup{(ii)}] The range of $\mathcal{H}_{x, B, \Xi}$ depends only on $\mathcal{H}_x$ and $\Xi$. Furthermore,
  \[
  \mathrm{range}(\mathcal{H}_{x, B, \Xi})=\mathrm{range}\big(\mathcal{H}_x\mathcal{P}_{x, B, \Xi}\big)\subseteq \mathrm{range}(\mathcal{H}_x).
  \]
\end{enumerate}
  \end{proposition}
  \begin{proof}
\noindent(i) Let $\mathcal{Z}:=\mathcal{H}_x^{\frac{1}{2}}\mathcal{P}_{x, B, \Xi}:\mathrm{T}_x\mathcal M\to \mathrm{T}_x\mathcal M$. Note that $\mathcal{Z}^{*}=\mathcal{P}_{x, B, \Xi}^{*}\mathcal{H}_x^{\frac{1}{2}}$ and $\mathcal{H}_x=\mathcal{H}_x^{\frac{1}{2}}\mathcal{H}_x^{\frac{1}{2}}$. Using the Moore--Penrose pseudoinverse, we have
  \[
  \mathcal{H}_{x, B, \Xi}
  =\mathcal{H}_x^{\frac{1}{2}}\,\mathcal{Z}\,(\mathcal{Z}^{*}\mathcal{Z})^{\dagger}\mathcal{Z}^{*}\,\mathcal{H}_x^{\frac{1}{2}}
  =\mathcal{H}_x^{\frac{1}{2}}\,\Pi_Z\,\mathcal{H}_x^{\frac{1}{2}},
  \]
  where $\Pi_Z:=\mathcal{Z}(\mathcal{Z}^{*}\mathcal{Z})^{\dagger}\mathcal{Z}^{*}$ is the $g_x$-orthogonal projection onto $Z :=\mathrm{range}(\mathcal{Z})$. Indeed, $\Pi_Z$ is $g_x$-self-adjoint and idempotent. 
  Since $\mathcal{H}_x^{\frac{1}{2}}\succeq 0$ and $\Pi_Z$ is a $g_x$-orthogonal projection, $\mathcal{H}_{x, B, \Xi}=\mathcal{H}_x^{\frac{1}{2}}\Pi_Z \mathcal{H}_x^{\frac{1}{2}}\succeq 0$. Moreover, $\operatorname{rank}(\mathcal{H}_{x, B, \Xi})=\operatorname{rank}(\Pi_Z\mathcal{H}_x^{\frac{1}{2}})\le \operatorname{rank}(\Pi_Z)=\operatorname{rank}(\mathcal{Z}) = \operatorname{rank}(\mathcal{H}_x^{\frac{1}{2}} \mathcal{P}_{x, B, \Xi})\le \operatorname{rank}(\mathcal{P}_{x, B, \Xi})\le \ell$.
  For every \(u\in \mathrm{T}_x\mathcal M\),
  \[
  \begin{aligned}
  \langle u,(\mathcal{H}_x-\mathcal{H}_{x, B, \Xi})[u]\rangle_x
  &=\langle (\mathrm{Id}_x-\Pi_Z)\mathcal{H}_x^{\frac{1}{2}}[u],(\mathrm{Id}_x-\Pi_Z)\mathcal{H}_x^{\frac{1}{2}}[u]\rangle_x \\
  &=\|(\mathrm{Id}_x-\Pi_Z)\mathcal{H}_x^{\frac{1}{2}}[u]\|_x^2 
  \ge 0.
  \end{aligned}
  \]

\noindent(ii) The expression $\mathcal{H}_{x, B, \Xi}=\mathcal{H}_x^{\frac{1}{2}}\Pi_Z \mathcal{H}_x^{\frac{1}{2}}$ shows that $\mathcal{H}_{x, B, \Xi}$ depends on $\Pi_Z$, which depends only on the subspace $\mathrm{range}(\mathcal{Z})=\mathrm{range}(\mathcal{H}_x^{\frac{1}{2}}\mathcal{P}_{x, B, \Xi}) = \mathrm{range}(\mathcal{H}_x\mathcal{P}_{x, B, \Xi}) $. Since $\mathrm{range}(\mathcal{Z})$ in turn depends only on $\mathrm{range}(\mathcal{P}_{x, B, \Xi})$,  $\mathrm{range}(\mathcal{H}_{x, B, \Xi})$ depends only on $\mathcal{H}_x$ and $\mathrm{range}(\mathcal{P}_{x, B, \Xi}) = \Xi$. Furthermore, this implies
  \begin{align*}
  \mathrm{range}(\mathcal{H}_{x, B, \Xi})
  =\mathcal{H}_x^{\frac{1}{2}}\big(\mathrm{range}(\Pi_Z)\big)
  =\mathcal{H}_x^{\frac{1}{2}}\big(\mathrm{range}(\mathcal{Z})\big)
  =\mathrm{range}(\mathcal{H}_x\mathcal{P}_{x, B, \Xi})
  \subseteq \mathrm{range}(\mathcal{H}_x).
  \end{align*}
  This yields the desired result.
  \end{proof}

  For tangent-space operators associated with the same subspace and sketching operators, the Riemannian Nystr\"om approximation enjoys the monotonicity property established in the following proposition.

  \begin{proposition}
    \label{prop: monotone}
    Let $x \in \mathcal{M}$ and let $\mathcal{H}_x,\mathcal{H}_x^\prime : \mathrm{T}_x\mathcal{M} \to \mathrm{T}_x\mathcal{M}$ be two $g_x$-self-adjoint PSD operators such that
    $
    \mathcal{H}_x\preceq \mathcal{H}_x^\prime.
    $
    Fix the sketching operator $\mathcal{P}_{x, B, \Xi}$ and denote by $\mathcal{H}_{x, B, \Xi}, \mathcal{H}_{x, B, \Xi}^\prime$ the Riemannian Nystr\"om approximation of $\mathcal{H}_x$ and $\mathcal{H}_x^\prime$, respectively. Then it holds that
    \[
    \mathcal{H}_x - \mathcal{H}_{x, B, \Xi} \preceq \mathcal{H}_x^\prime - \mathcal{H}_{x, B, \Xi}^\prime.
    \]
    \end{proposition}

    \begin{proof}
      Denote
      $\mathcal P := \mathcal{P}_{x,B,\Xi}$ for abbreviation.
      Consider the block operator
      \[
      \mathcal K(\mathcal{H})
      :=
      \begin{pmatrix}
      \mathcal{H} & \mathcal{H}\mathcal P\\
      \mathcal P^{*} \mathcal{H} & \mathcal P^{*} \mathcal{H}\mathcal P
      \end{pmatrix}.
      \]
      This operator is $g_x$-self-adjoint and positive semidefinite, and 
    the generalized Schur complement of the lower-right block satisfies that
      \[
      \mathcal K(\mathcal{H}_x)\,/\,(\mathcal P^{*}\mathcal{H}_x\mathcal P) := \mathcal{H}_x - \mathcal{H}_x \mathcal{P} (\mathcal{P}^* \mathcal{H}_x \mathcal{P})^\dagger \mathcal{P}^* \mathcal{H}_x = \mathcal{H}_x-\mathcal{H}_{x,B,\Xi}.
      \]
      Moreover, if $\mathcal{H}_x\preceq \mathcal{H}_x'$, then $\mathcal K(\mathcal{H}_x)\preceq \mathcal K(\mathcal{H}_x')$ since
      $
      \mathcal K(\mathcal{H}_x')-\mathcal K(\mathcal{H}_x)=\mathcal K(\mathcal{H}_x'-\mathcal{H}_x)\succeq 0.
      $
      The generalized Schur complement is monotone in the Loewner order; see, e.g.,~\cite{Ando1979GeneralizedSchur,AndersonTrapp1975ShortedII}. Therefore,
      $
      \mathcal K(\mathcal{H}_x)\preceq \mathcal K(\mathcal{H}_x')
      $ implies that
      $
      \mathcal K(\mathcal{H}_x)\,/\,(\mathcal P^{*}\mathcal{H}_x\mathcal P)
      \preceq
      \mathcal K(\mathcal{H}_x')\,/\,(\mathcal P^{*}\mathcal{H}_x'\mathcal P).
      $
      Substituting the Schur complements above yields
      $\mathcal{H}_x-\mathcal{H}_{x,B,\Xi}\preceq \mathcal{H}_x'-\mathcal{H}_{x,B,\Xi}'$.
      \end{proof}

The next result shows the optimality of the Riemannian Nystr\"om approximation within a given subspace under the Loewner order.
  \begin{proposition}
    \label{prop:nystrom-loewner-opt}
    Denote ${U}:=\operatorname{range}\!\big(\mathcal{H}_x\mathcal{P}_{x, B, \Xi}\big)$. $\mathcal{H}_{x, B, \Xi}$ is the unique maximum element of the operator set
    \[
    \mathcal{C}_{U} \;:=\; \left\{\mathcal{G}:\mathcal{G} \text{ is $g_x$-self-adjoint and PSD}, \operatorname{range}(\mathcal{G})\subseteq U,\ \mathcal{G}\preceq \mathcal{H}_x\,\right\}
    \]
    in the sense of Loewner order.
    \end{proposition}
    
\begin{proof}
  Following the proof of \cref{lem:riem-21},
  let $\mathcal{Z}:=\mathcal{H}_x^{\frac12}\mathcal P_{x,B,\Xi}$, $Z := \mathrm{range}(\mathcal{Z})$ and $\Pi_Z:=\mathcal{Z}(\mathcal{Z}^{*}\mathcal{Z})^{\dagger}\mathcal{Z}^{*}$, then $\Pi_Z$ is the $g_x$-orthogonal projection onto $Z$.
  We have
$\mathcal{H}_{x,B,\Xi}=\mathcal{H}_x^{\frac12}\,\Pi_Z\,\mathcal{H}_x^{\frac12}
$  , $\operatorname{range}(\mathcal{H}_{x,B,\Xi})=\operatorname{range}(\mathcal{H}_x\mathcal P_{x,B,\Xi})=U,
  $
  and therefore $\mathcal{H}_{x,B,\Xi}\in\mathcal C_U$.

  Let $\mathcal{G}\in\mathcal C_U$ be an arbitrary element.
  Let
  $
  \mathcal{A}:=\mathcal{H}_x^{\dagger/2}\,\mathcal{G}\,\mathcal{H}_x^{\dagger/2}.
  $
  Then $\mathcal{A}$ is $g_x$-self-adjoint and $\mathcal{A}\succeq 0$. Moreover,
  $0 \preceq \mathcal{G} \preceq \mathcal{H}_x$ implies
  $
  \mathcal{A}
  \preceq
  \mathcal{H}_x^{\dagger/2}\,\mathcal{H}_x\,\mathcal{H}_x^{\dagger/2}
  =
  \Pi_{\operatorname{range}(\mathcal{H}_x)}
  \preceq
  \mathrm{Id}_x.
  $
  The range constraint $\operatorname{range}(\mathcal{G})\subseteq U=\operatorname{range}(\mathcal{H}_x^{\frac{1}{2}}\Pi_Z)$ implies
  $
  \operatorname{range}(\mathcal{H}_x^{\frac{1}{2}}\mathcal{A})\subseteq \operatorname{range}(\mathcal{H}_x^{\frac{1}{2}}\Pi_Z)
  $
  Since $\mathcal{H}_x^{\frac{1}{2}}$ is injective on $\operatorname{range}(\mathcal{H}_x)$, the inclusion above implies
  $
  \operatorname{range}(\mathcal{A})\subseteq \operatorname{range}(\Pi_Z)=Z.
  $
  Hence $\Pi_Z\mathcal{A}=\mathcal{A}$. Since $\mathcal{A}$ is self-adjoint, it also follows that
  $
  \mathcal{A}=(\Pi_Z\mathcal{A})^*=\mathcal{A}\Pi_Z.
  $
  Therefore,
  $
  \mathcal{A}=\Pi_Z\,\mathcal{A}\,\Pi_Z.
  $
  Using $0\preceq \mathcal{A}\preceq \mathrm{Id}_x$, we have
  $
  0\preceq \mathcal{A}=\Pi_Z\mathcal{A}\Pi_Z\preceq \Pi_Z\,\mathrm{Id}_x\,\Pi_Z=\Pi_Z.
  $
  Multiplying by $\mathcal{H}_x^{\frac{1}{2}}$ on both sides yields
  \[
  \mathcal{G}=\mathcal{H}_x^{\frac{1}{2}}\mathcal{A}\mathcal{H}_x^{\frac{1}{2}}
  \preceq
  \mathcal{H}_x^{\frac{1}{2}}\Pi_Z\mathcal{H}_x^{\frac{1}{2}}
  =\mathcal{H}_{x,B,\Xi}.
  \]
  Since $\mathcal{G}\in\mathcal C_U$ is arbitrary, this shows that $\mathcal{H}_{x,B,\Xi}$ is the maximal element of $\mathcal C_U$ in the Loewner order.
  Uniqueness follows immediately: if $\widetilde{\mathcal{G}}\in\mathcal C_U$ is also maximum, then
  $\widetilde{\mathcal{G}}\preceq \mathcal{H}_{x,B,\Xi}$ and $\mathcal{H}_{x,B,\Xi}\preceq \widetilde{\mathcal{G}}$, hence $\widetilde{\mathcal{G}}=\mathcal{H}_{x,B,\Xi}$.
  \end{proof}

\subsection{Approximation errors}
This section discusses the randomized approximation errors.
Conditioned on $B$, the randomness of Riemannian Nystr\"om approximation derives from the map $\mathcal{F}$. In the construction of sketching operators, we can treat $\mathcal{F}$ as a random map connecting two $\ell$-dimensional subspaces of $\mathrm{T}_x \mathcal{M}$ and then set $\Xi$ to its image space. Alternatively, we can view $\Xi$ as a random element in the Grassmann manifold $\mathrm{Gr}(\ell, \mathrm{T}_x \mathcal{M})$~\cite{bendokat2024grassmann}, i.e., the set of all $\ell$-dimensional subspaces of $\mathrm{T}_x \mathcal{M}$, and afterward let $\mathcal{F}: B \to \Xi$ be a random map.
In this section, the approximation error is first analyzed under the Gaussian sketching condition analogous to the Euclidean case~\cite{martinsson2020randomized}, and then extended to the results under a weaker, intrinsic and geometric Haar--Grassmann sketching condition.

\paragraph{Gaussian sketching} In the manifold setting, a sketching operator $\mathcal{P}_{x, B, \Xi}$ satisfies the Gaussian sketching condition if, conditioned on $B$, the collection of real random variables
\[
\big\{\,\langle \mathcal F[u],w\rangle_x : u \in B, \ w\in\mathrm T_x\mathcal M\,\big\}
\]
is jointly Gaussian and satisfies
\begin{equation}
  \label{eq:gaussian-sketching-condition}
\begin{aligned}
& \mathbb E\left[\,\langle \mathcal F[u],w\rangle_x \mid B\right]=0, \quad \text{ for all } u \in B,\ \text{ for all } w\in\mathrm T_x\mathcal M, \\
& \mathbb E\,\left[\langle \mathcal F[u],w\rangle_x\,\langle \mathcal F[v],z\rangle_x \mid B\right]
=\langle u,v\rangle_x\,\langle w,z\rangle_x,
\quad \text{ for all } u,v \in B, \text{ for all } w,z\in\mathrm T_x\mathcal M.
\end{aligned}
\end{equation}
Equivalently, conditioned on $B$, for every $u\in B$, $\mathcal F[u] \sim \mathcal N\big(0, \|u\|_x^{2}\,\mathrm{Id}_x\big)$ on $\mathrm T_x\mathcal M$, and for any $u,v\in B$ with $\langle u,v\rangle_x=0$, the random vectors $\mathcal F[u]$ and $\mathcal F[v]$ are independent.

Gaussian sketching is equivalent to the following coordinate description.
For any $g_x$-orthonormal frame $\{b_{x, j}\}_{j=1}^d$ of $\mathrm T_x\mathcal M$ and any $g_x$-orthonormal basis $\{b^\prime_{x, i}\}_{i=1}^{\ell}$ of $B$ ($\{b^\prime_{x, i}\}_{i=1}^{\ell}$ can be a subset of $\{b_{x, j}\}_{j=1}^d$ for convenience), define
\[
\omega_{ji}:=\langle \mathcal F[b^\prime_{x, i}],b_{x, j}\rangle_x,
\qquad j=1,\dots,d,\quad i=1,\dots,\ell.
\]
Then Gaussian sketching holds if and only if
$
\omega_{ji}\stackrel{\mathrm{i.i.d.}}{=}\mathcal N(0,1),
j=1,\dots,d, i=1,\dots,\ell.
$

\begin{remark}
\label{rem:gaussian-coordinates}
The coordinate formulation of Gaussian sketching on a manifold closely corresponds to the Gaussian sketching defined in a Hilbert space~\cite{PerssonBoulleKressner2024}. Specifically,
let
  $
J:\mathbb R^{\ell}\to B$ be the linear map defined by $J e_i=b^\prime_{x,i}$ for $i=1,\dots,\ell$, where $\{e_i\}_{i=1}^{\ell}$ denotes the canonical basis of $\mathbb R^{\ell}$. Since $\{b^\prime_{x,i}\}_{i=1}^{\ell}$ is $g_x$-orthonormal, $J$ is an isometric isomorphism. Moreover, because $\mathcal P_{x,B,\Xi}=\mathcal F\Pi_B$ and $J a\in B$,  for every $a\in\mathbb R^{\ell}$, we have
$
(\mathcal P_{x,B,\Xi}\circ J)[a]
=
\mathcal F\Pi_B[J a]
=
\mathcal F[J a].
$
Therefore the operator
\[
\Omega:=\mathcal P_{x,B,\Xi}\circ J=\mathcal F\circ J:
\mathbb R^{\ell}\to \mathrm T_x\mathcal M
\]
is the precise counterpart of the matrix (finite-dimensional case) or quasimatrix (infinite-dimensional case)
proposed in~\cite{PerssonBoulleKressner2024}.
In particular,
$
\Omega e_i
=
\mathcal F[b^\prime_{x,i}], i=1,\dots,\ell,
$
and its matrix representation in the frame $\{b_{x,j}\}_{j=1}^d$ is the coefficient matrix
\[
[\Omega]_{ji}
=
\langle b_{x,j},\Omega e_i\rangle_x
=
\langle b_{x,j},\mathcal F[b^\prime_{x,i}]\rangle_x
=
\omega_{ji},
\qquad j=1,\dots,d,\quad i=1,\dots,\ell.
\]
Hence, when a Hilbert space is specialized to $\mathrm T_x\mathcal M$ equipped with $\langle\cdot,\cdot\rangle_x$, the Gaussian sketching in~\cite{PerssonBoulleKressner2024} coincides with the Gaussian sketching condition introduced above on a manifold.
\end{remark}

\paragraph{Haar–Grassmann sketching} Although the Gaussian sketching~\cref{eq:gaussian-sketching-condition} is coordinate-free in form, \cref{rem:gaussian-coordinates} shows that it is still equivalent to a specific Gaussian coefficient model in orthonormal coordinates.
It is desirable to introduce an intrinsic sketching condition stated directly in terms of geometry.
Motivated by a viewpoint that appears in the literature on randomized subspace methods (e.g.,~\cite{halko2011finding}), 
we impose a polar factorization~\cite{higham2008functions} in which the isometric component is Haar-uniform~\cite{chikuse2003special} conditionally on $B, \Xi$ and independent of the radial factor.

By the polar decomposition theorem~\cite{higham2008functions}, $\mathcal F$ admits the polar decomposition
\[
\mathcal F=\mathcal U\mathcal R,
\]
where $\mathcal U:B\to \Xi$ is a linear isometry and $\mathcal R:B\to B$ is $g_x$-self-adjoint and positive definite. This leads to the following Haar--Grassmann sketching condition. A key geometric advantage of this condition is that it is formulated entirely in terms of subspaces and isometries, without reference to a particular coordinate system, and is therefore naturally compatible with intrinsic manifold operations such as vector transport.

\begin{definition}
[Haar--Grassmann sketching condition]
  \label{def:haar-grassmann-sketch}
  Assume the randomness is conditioned on $B$. 
  We say that the sketching operator $\mathcal{P}_{x, B, \Xi}$ satisfies the Haar--Grassmann sketching condition if the following statements hold.
\begin{enumerate}[label=\textup{(\roman*)}, ref=\textup{(\roman*)}]
\item \label{itm:haar-grassmann-sketch-haar}
The random subspace $\Xi$ is Haar-uniform on $\mathrm{Gr}(\ell,\mathrm T_x\mathcal M)$.

\item \label{itm:haar-grassmann-sketch-polar}
Almost surely, $\operatorname{rank}(\mathcal F)=\ell$.
Conditionally on $\Xi$, the isometry $\mathcal U$ is Haar-uniform on the set of all linear isometries from $B$ to $\Xi$, and $\mathcal U$ is independent of $\mathcal R$.

\item \label{itm:haar-grassmann-sketch-moment}
The radial factor $\mathcal R$ satisfies
\[
\mathbb E\big[\mathcal R^2 \mid B\big]=d\,\mathrm{Id}_B,
\qquad
\mathbb E\big[\|\mathcal R^{-1}\|_{\mathrm{op}}^2 \mid B\big]<\infty.
\]
\end{enumerate}
\end{definition}

Without loss of generality, we assume that the randomness is conditioned on $B$ throughout the paper and drop the conditioning notation for brevity.
The Gaussian sketching condition implies the Haar--Grassmann sketching condition, as in the following proposition.

\begin{proposition}
  \label{prop:gaussian-implies-haar-grassmann}
  Let $d > \ell + 1$. Then Gaussian sketching implies the Haar--Grassmann sketching condition.
\end{proposition}

\begin{proof}
Fix any $g_x$-orthonormal basis $\{b'_{x,i}\}_{i=1}^{\ell}$ of $B$ and any $g_x$-orthonormal frame $\{b_{x,j}\}_{j=1}^{d}$ of $\mathrm T_x\mathcal M$.
By \cref{rem:gaussian-coordinates}, the matrix representation $\mathbf{\Omega}\in\mathbb R^{d\times \ell}$ of the operator $\mathcal F$ in these coordinates has i.i.d. standard normal entries. We verify the three requirements of \cref{def:haar-grassmann-sketch}.

\noindent(i) For every orthogonal matrix $\mathbf{Q}\in\mathbb R^{d\times d}$, rotational invariance of the standard Gaussian law gives
$
\mathbf{Q}\mathbf{\Omega}\stackrel{d}{=}\mathbf{\Omega}.
$
Let $T_{\mathbf Q}:\mathrm T_x\mathcal M\to\mathrm T_x\mathcal M$ be the $g_x$-isometry represented by $\mathbf Q$ in the frame $\{b_{x,j}\}_{j=1}^d$.
Then the range of $T_{\mathbf Q}\circ \mathcal F$ is $T_{\mathbf Q}(\Xi)$, while the matrix of $T_{\mathbf Q}\circ \mathcal F$ is $\mathbf Q\mathbf\Omega$.
Hence
$
T_{\mathbf Q}(\Xi)\stackrel{d}{=}\Xi.
$
Therefore the law of $\Xi=\operatorname{range}(\mathcal F)$ is invariant under all $g_x$-isometries of $\mathrm T_x\mathcal M$, i.e., $\Xi$ is Haar-uniform on $\mathrm{Gr}(\ell,\mathrm T_x\mathcal M)$.

\noindent(ii) Since $d > \ell$, the Gaussian matrix $\mathbf\Omega$ has full column rank almost surely. Let
$
\mathbf S:=(\mathbf\Omega^{\top}\mathbf\Omega)^{1/2}$,
$
\mathbf W:=\mathbf\Omega\mathbf S^{-1}.
$
Then almost surely $\mathbf S$ is symmetric positive definite, $\mathbf W^{\top}\mathbf W=\mathbf I_{\ell}$, and
$
\mathbf\Omega=\mathbf W\mathbf S
$
is the matrix polar decomposition of $\mathbf\Omega$.
Define operators $\mathcal U:B\to \Xi$ and $\mathcal R:B\to B$  whose matrix representations in the chosen bases are $\mathbf W$ and $\mathbf S$, respectively. Then the factorization $\mathcal F=\mathcal U\mathcal R$ holds.

It remains to verify the conditional Haar property and the independence of $\mathcal U$ and $\mathcal R$.
For every orthogonal matrix $\mathbf O\in\mathbb R^{\ell\times \ell}$, right invariance of the Gaussian law gives
$
\mathbf\Omega\mathbf O\stackrel{d}{=}\mathbf\Omega.
$
Moreover,
$
\operatorname{range}(\mathbf\Omega\mathbf O)=\operatorname{range}(\mathbf\Omega)$,
$
\mathbf S(\mathbf\Omega\mathbf O)=\mathbf O^{\top}\mathbf S(\mathbf\Omega)\mathbf O$,
$
\mathbf W(\mathbf\Omega\mathbf O)=\mathbf W(\mathbf\Omega)\mathbf O$.
Thus, conditionally on the subspace $\Xi=\operatorname{range}(\mathcal F)$, the law of $\mathcal U$ is invariant under right composition by every orthogonal operator on $B$.
Since the set of linear isometries $B\to \Xi$ forms a homogeneous space, denoted by $\mathrm{Iso}(B,\Xi)$, under this right action, this right-invariant probability law is exactly the Haar law on $\mathrm{Iso}(B,\Xi)$.
Finally, for a standard Gaussian matrix, the polar factors $\mathbf W$ and $\mathbf S$ are independent; equivalently, the operator-valued polar factors $\mathcal U$ and $\mathcal R$ are independent.

\noindent(iii) For any $u\in B$, Gaussian sketching gives
$
\mathbb E\big[\langle u,\mathcal F^{*}\mathcal F[u]\rangle_x\big]
=
\mathbb E\big[\|\mathcal F[u]\|_x^2\big]
=d\,\|u\|_x^2.
$
Hence
$
\mathbb E[\mathcal F^{*}\mathcal F]=d\,\mathrm{Id}_B.
$
Since $\mathcal F^{*}\mathcal F=\mathcal R^2$, it follows that
$
\mathbb E[\mathcal R^2]=d\,\mathrm{Id}_B.
$
Also, because the matrix of $\mathcal R$ is $\mathbf S=(\mathbf\Omega^{\top}\mathbf\Omega)^{1/2}$, we have
$
\|\mathcal R^{-1}\|_{\mathrm{op}}=
\|\mathbf S^{-1}\|_2
$.
Therefore the smallest-singular-value tail bound for Gaussian matrices, for example~\cite[Theorem~4.6.1]{Vershynin2018HDP}, yields
$
\mathbb E [ \|\mathcal R^{-1}\|_{\mathrm{op}}^2 ]<\infty
$
whenever $d>\ell+1$.
Thus all requirements of \cref{def:haar-grassmann-sketch} hold.
\end{proof}

We now establish approximation guarantees for the Riemannian Nystr\"om approximation under randomized sketching. Proposition~\ref{prop:manifold-22-routeB} is distribution-agnostic: its conclusion holds without any assumption on the distribution of the sketching operator. We then specialize this general result to two concrete randomized models, namely Gaussian sketching and Haar--Grassmann sketching. Let \(\lambda_1(\mathcal{H}_x)\ge\cdots\ge\lambda_d(\mathcal{H}_x)\ge 0\) denote the eigenvalues of \(\mathcal{H}_x\).

\begin{proposition}
  \label{prop:manifold-22-routeB}
  Let $\ell\ge 4$.
  For any $p\in\{2,\dots,\ell-2\}$,
  there exist two constants $m_{\mathrm{HS}}$ and $m_{\mathrm{op}}$ (depending on $p, \ell$ and $\mathcal{P}_{x, B, \Xi}$) such that
  \[
  \mathbb E\left[\,\big\|\mathcal{H}_x-\mathcal{H}_{x,B,\Xi}\big\|_{\mathrm{op}}\,\right]
  \le
  \Bigl(1+2\,m_{\mathrm{HS}}\Bigr)\,\lambda_{\ell-p+1}(\mathcal{H}_x)
  +2\,m_{\mathrm{op}}\sum_{j>\ell-p}\lambda_j(\mathcal{H}_x).
  \]

\end{proposition}

\begin{proof}
Write $\mathcal P:=\mathcal P_{x,B,\Xi}$ for abbreviation and let $\mathcal Z:=\mathcal H_x^{\frac{1}{2}}\mathcal P$.
  Let $\Pi_Z:=\mathcal Z(\mathcal Z^{*}\mathcal Z)^{\dagger}\mathcal Z^{*}$ be the $g_x$-orthogonal projection onto $\operatorname{range}(\mathcal Z)$.
  As shown in \cref{lem:riem-21}, it holds that
  \[
  \|\mathcal H_x-\mathcal H_{x,B,\Xi}\|_{\mathrm{op}}
  =\big\|\mathcal H_x^{\frac{1}{2}}(\mathrm{Id}_x-\Pi_Z)\mathcal H_x^{\frac{1}{2}}\big\|_{\mathrm{op}}
  =\big\|(\mathrm{Id}_x-\Pi_Z)\mathcal H_x^{\frac{1}{2}}\big\|_{\mathrm{op}}^{2}.
  \]
  Fix $p\in\{2,\dots,\ell-2\}$ and set $r:=\ell-p$.
  Let $\Pi_1$ denote the $g_x$-orthogonal projection onto the invariant subspace of $\mathcal H_x$ associated with the largest $r$ eigenvalues, and set $\Pi_2:=\mathrm{Id}_x-\Pi_1$.
  Define $\Sigma_2:=\mathcal H_x^{\frac{1}{2}}\Pi_2$ and the decomposed sketching maps $\mathcal P_1:=\Pi_1 \mathcal P_{x,B,\Xi}$ and $\mathcal P_2:=\Pi_2 \mathcal P_{x,B,\Xi}$.
It follows from the spectral theorem that
  \[
  \|\Sigma_2\|_{\mathrm{op}}^{2}=\lambda_{r+1}(\mathcal H_x)=\lambda_{\ell-p+1}(\mathcal H_x),
  \qquad
  \|\Sigma_2\|_{\mathrm{HS}}^{2}=\sum_{j>r}\lambda_j(\mathcal H_x)=\sum_{j>\ell-p}\lambda_j(\mathcal H_x).
  \]
  A standard Nystr\"om estimate (see \cite[Proposition~2.2]{Frangella2023} and its proof) yields
  $
  \big\|(\mathrm{Id}_x-\Pi_Z)\mathcal H_x^{\frac{1}{2}}\big\|_{\mathrm{op}}^{2}
  \le
  \|\Sigma_2\|_{\mathrm{op}}^{2}
  +\big\|\Sigma_2\,\mathcal P_2\,\mathcal P_1^{\dagger}\big\|_{\mathrm{op}}^{2}
  $.
  Taking expectations and using 
  $
  \|\Sigma_2\,\mathcal P_2\,\mathcal P_1^{\dagger}\|_{\mathrm{op}}
  \le \|\Sigma_2\|_{\mathrm{op}}\,\|\mathcal P_1^{\dagger}\|_{\mathrm{HS}},
  \|\Sigma_2\,\mathcal P_2\,\mathcal P_1^{\dagger}\|_{\mathrm{op}}
  \le \|\Sigma_2\|_{\mathrm{HS}}\,\|\mathcal P_1^{\dagger}\|_{\mathrm{op}},
$
  we obtain
  \[
  \mathbb E [ \|\mathcal H_x-\mathcal H_{x,B,\Xi}\|_{\mathrm{op}} ]
  \le
  \lambda_{\ell-p+1}(\mathcal H_x)
  +2\,\|\Sigma_2\|_{\mathrm{op}}^{2}\,\mathbb E [ \|\mathcal P_1^{\dagger}\|_{\mathrm{HS}}^{2} ]
  +2\,\|\Sigma_2\|_{\mathrm{HS}}^{2}\,\mathbb E [ \|\mathcal P_1^{\dagger}\|_{\mathrm{op}}^{2} ].
  \]
  Finally, define
  $m_{\mathrm{HS}}:=\mathbb E [ \|\mathcal P_1^{\dagger}\|_{\mathrm{HS}}^{2} ]$
  and
  $m_{\mathrm{op}}:=\mathbb E [ \|\mathcal P_1^{\dagger}\|_{\mathrm{op}}^{2} ]$.
  Substituting these definitions together with the identities for $\|\Sigma_2\|_{\mathrm{op}}^{2}$ and $\|\Sigma_2\|_{\mathrm{HS}}^{2}$ into the preceding estimate yields the stated bound.
  \end{proof}

\Cref{prop:manifold-22-routeB} shows that the only quantities in the approximation analysis that depend on the distribution of the sketching are the two pseudoinverse moments $m_{\mathrm{HS}}$ and $m_{\mathrm{op}}$.
Under Gaussian sketching, these moments admit explicit bounds (Appendix~B.0.1 in~\cite{Frangella2023}), and substituting them into \cref{prop:manifold-22-routeB} yields the explicit Gaussian error bound stated in \cref{cor:manifold-22-gaussian}. The result matches the Euclidean Nystr\"om approximation error under Gaussian sketching.

\begin{corollary}
  \label{cor:manifold-22-gaussian}
  Let $\ell\ge 4$.
  Suppose $\mathcal{P}_{x, B, \Xi}$ satisfies the Gaussian sketching condition.
  Then it holds that
  \[
  \mathbb E\left[ \big\|\mathcal{H}_x-\mathcal{H}_{x, B, \Xi}\big\|_{\mathrm{op}} \right]
  \le
  \min_{2\le p\le \ell-2}
  \left\{
  \Big(1+\tfrac{2(\ell-p)}{p-1}\Big)\lambda_{\ell-p+1}(\mathcal{H}_x)
  +\tfrac{2\mathrm e^2\,\ell}{p^{2}-1}\sum_{j>\ell-p}\lambda_j(\mathcal{H}_x)
  \right\}.
  \]
\end{corollary}

\begin{proof}
Applying \cref{prop:manifold-22-routeB} and leveraging the Gaussian pseudoinverse moment bounds (Appendix~B.0.1 in~\cite{Frangella2023})
$
\mathbb E\big[\|\mathcal P_1^{\dagger}\|_{\mathrm{HS}}^{2}\big]=\frac{\ell-p}{p-1}
$,
$
\mathbb E\big[\|\mathcal P_1^{\dagger}\|_{\mathrm{op}}^{2}\big]\le \mathrm e^{2}\,\frac{\ell}{p^{2}-1}
$
yields the claimed bound.
\end{proof}

Under the Haar--Grassmann sketching condition, the polar factorization $\mathcal F=\mathcal U\mathcal R$ separates the randomness in the range geometry from that in the radial scaling. When $\mathcal R$ is deterministic, or more generally when suitable inverse-moment bounds for $\mathcal R$ are available, the resulting approximation bounds have the same order as in the Gaussian case. The next technical lemma shows that the Hilbert--Schmidt moment of $\mathcal A$ can be bounded solely in terms of $p$ and $\ell$. We then use this estimate to control the pseudoinverse moments $m_{\mathrm{HS}}$ and $m_{\mathrm{op}}$ appearing in \cref{prop:manifold-22-routeB}; see \cref{prop:haar-grassmann-pinv-moments}.

\begin{lemma}
  \label{lem:haar-grassmann-A-moments}
  Suppose $\mathcal{P}_{x, B, \Xi}$ satisfies the Haar--Grassmann sketching condition, and let $\Pi_1$ be defined as in the proof of \cref{prop:manifold-22-routeB}.
  Define
  $
  \mathcal A:=\Pi_1\mathcal U:B\to \operatorname{range}(\Pi_1).
  $
  Then $\mathcal A$ has full row rank almost surely. Moreover, it holds that
  \begin{equation}
    \label{eq:A-hs-moment-new}
    \mathbb E\big[\|\mathcal A^{\dagger}\|_{\mathrm{HS}}^{2}\big]
    \le
    d\,\frac{\ell-p}{p-1},
  \end{equation}
  and there exists a universal constant $C_0>0$ such that
  \begin{equation}
    \label{eq:A-op-moment-new}
    \mathbb E\big[\|\mathcal A^{\dagger}\|_{\mathrm{op}}^{2}\big]
    \le
    C_0\,d\,\frac{\ell}{p^{2}-1}.
  \end{equation}
\end{lemma}

\begin{proof}
  Set $r:=\ell-p$.
  Choose arbitrary $g_x$-orthonormal isomorphisms
  $
  J_B:B\to \mathbb R^\ell,
  $
  $
  J_1:\operatorname{range}(\Pi_1)\to \mathbb R^r,
  $
  and let
  $
  \mathbf{A}:=J_1\mathcal A J_B^{-1}\in \mathbb R^{r\times \ell}
  $
  be the matrix representation of $\mathcal A$ in these coordinates.
  Since $J_B$ and $J_1$ are orthogonal, $\mathbf{A}$ and $\mathcal A$ have the same singular values. In particular,
  $
  \|\mathcal A^{\dagger}\|_{\mathrm{op}}=\|\mathbf{A}^{\dagger}\|_{\mathrm{op}},
  $
  $
  \|\mathcal A^{\dagger}\|_{\mathrm{HS}}=\|\mathbf{A}^{\dagger}\|_{\mathrm{HS}}.
  $
  Thus it suffices to estimate the corresponding moments for $\mathbf{A}$.
  Equivalently, let
  $
  \mathbf{X}\in\mathbb R^{r\times \ell}
  $
  and
  $
  \mathbf{Y}\in\mathbb R^{r\times (d-\ell)}
  $
  be independent standard Gaussian matrices, and define
  $
  \mathbf{W}_1:=\mathbf{X}\mathbf{X}^\top,
  \mathbf{W}_2:=\mathbf{Y}\mathbf{Y}^\top.
  $
  Then
  $
  \mathbf{A}\mathbf{A}^\top
  \stackrel{d}{=}
  (\mathbf{W}_1+\mathbf{W}_2)^{-1/2}\,\mathbf{W}_1\,(\mathbf{W}_1+\mathbf{W}_2)^{-1/2}.
  $
  In particular,
  $
  (\mathbf{A}\mathbf{A}^\top)^{-1}
  \stackrel{d}{=}
  \mathbf{I}_r+\mathbf{W}_1^{-1/2}\mathbf{W}_2\mathbf{W}_1^{-1/2}.
  $

  We first compute the Hilbert--Schmidt moment.
  Since $\mathbf{W}_1\sim \mathrm{Wishart}_r(\ell,\mathbf{I}_r)$ and $\mathbf{W}_2\sim \mathrm{Wishart}_r(d-\ell,\mathbf{I}_r)$ are independent, and $\ell-r=p\ge 2$, we have
  \[
  \mathbb E[\mathbf{W}_1^{-1}]
  =
  \frac{1}{\ell-r-1}\,\mathbf{I}_r
  =
  \frac{1}{p-1}\,\mathbf{I}_r,
  \qquad
  \mathbb E[\mathbf{W}_2]
  =
  (d-\ell)\,\mathbf{I}_r.
  \]
  Therefore
  $
  \mathbb E \big[\|\mathbf{A}^{\dagger}\|_{\mathrm{HS}}^{2}\big]
  =
  r+\operatorname{tr}\!\big(\mathbb E[\mathbf{W}_1^{-1}]\,\mathbb E[\mathbf{W}_2]\big)
  =
  r+\frac{r(d-\ell)}{p-1}
  \le
  d\,\frac{\ell-p}{p-1}.
  $
  This proves \eqref{eq:A-hs-moment-new}.

  For the operator norm, the same representation gives
  \[
  \|\mathbf{A}^{\dagger}\|_{\mathrm{op}}^{2}
  =
  \lambda_{\max}\!\big((\mathbf{A}\mathbf{A}^\top)^{-1}\big)
  \le
  1+\|\mathbf{W}_1^{-1/2}\mathbf{W}_2\mathbf{W}_1^{-1/2}\|_{\mathrm{op}}
  \le
  1+\|\mathbf{W}_1^{-1}\|_{\mathrm{op}}\|\mathbf{W}_2\|_{\mathrm{op}}.
  \]
  Since $\mathbf{W}_1=\mathbf{X}\mathbf{X}^\top$ and $\mathbf{W}_2=\mathbf{Y}\mathbf{Y}^\top$, this becomes
  $
  \|\mathbf{A}^{\dagger}\|_{\mathrm{op}}^{2}
  \le
  1+\|\mathbf{X}^{\dagger}\|_{\mathrm{op}}^{2}\,\|\mathbf{Y}\|_{\mathrm{op}}^{2}.
  $
  Taking expectations and using the independence of $\mathbf{X}$ and $\mathbf{Y}$ yields
  $
  \mathbb E\big[\|\mathbf{A}^{\dagger}\|_{\mathrm{op}}^{2}\big]
  \le
  1+\mathbb E\big[\|\mathbf{X}^{\dagger}\|_{\mathrm{op}}^{2}\big]\,
  \mathbb E\big[\|\mathbf{Y}\|_{\mathrm{op}}^{2}\big].
  $
  Now $\mathbf{X}$ is an $r\times \ell$ standard Gaussian matrix with oversampling gap $\ell-r=p$, so the Gaussian pseudoinverse estimate from Appendix~B.0.1 in~\cite{Frangella2023} gives
  $
  \mathbb E\big[\|\mathbf{X}^{\dagger}\|_{\mathrm{op}}^{2}\big]
  \le
  \mathrm e^{2}\,\frac{\ell}{p^{2}-1}.
  $
  Also, the standard Gaussian spectral norm estimate together with Gaussian concentration yields
  $
  \mathbb E\big[\|\mathbf{Y}\|_{\mathrm{op}}^{2}\big]
  \le
  \bigl(\sqrt r+\sqrt{d-\ell}\bigr)^{2}+1
  \le
  2r+2(d-\ell)+1
  \le
  3d.
  $
  Consequently,
  $
  \mathbb E\big[\|\mathbf{A}^{\dagger}\|_{\mathrm{op}}^{2}\big]
  \le
  1+3\mathrm e^{2}\,d\,\frac{\ell}{p^{2}-1}.
  $
  Since $d\ge \ell\ge 4$ and $p\ge 2$, we have
  $
  d\,\frac{\ell}{p^{2}-1}\ge \frac{16}{3}\ge 1.
  $
  Hence the additive constant can be absorbed into the same scale, and thus
  \[
  \mathbb E\big[\|\mathbf{A}^{\dagger}\|_{\mathrm{op}}^{2}\big]
  \le
  (3\mathrm e^{2}+1)\,d\,\frac{\ell}{p^{2}-1}.
  \]
  Since $\mathbf{A}$ and $\mathcal A$ have the same singular values, the same estimate holds for $\mathcal A$.
\end{proof}

\begin{proposition}
  \label{prop:haar-grassmann-pinv-moments}
  Let $\ell\ge 4$.
  Fix any $p\in\{2,\dots,\ell-2\}$.
  Suppose $\mathcal{P}_{x, B, \Xi}$ satisfies the Haar--Grassmann sketching condition.
  Let $\mathcal P,\Pi_1,\mathcal P_1$ be defined as in the proof of \cref{prop:manifold-22-routeB}.
  Define the inverse-moment quantities
  \begin{equation}
    \label{eq:haar-hs-moment}
    \rho_{\mathrm{HS}}
    :=
    \frac{d}{\ell}\,\mathbb E\big[\operatorname{tr}(\mathcal R^{-2})\big],
    \qquad
    \rho_{\mathrm{op}}
    :=
    d\,\mathbb E\big[\|\mathcal R^{-1}\|_{\mathrm{op}}^{2}\big].
  \end{equation}
  Then $\rho_{\mathrm{HS}},\rho_{\mathrm{op}}<\infty$, and there exists a universal constant $C_0>0$, independent of $x,d,\ell,p$, such that
  \[
  \mathbb E\big[\|\mathcal P_1^{\dagger}\|_{\mathrm{HS}}^{2}\big]
  \le
  \rho_{\mathrm{HS}}\,\frac{\ell-p}{p-1},
  \qquad
  \mathbb E\big[\|\mathcal P_1^{\dagger}\|_{\mathrm{op}}^{2}\big]
  \le
  C_0\,\rho_{\mathrm{op}}\,\frac{\ell}{p^{2}-1}.
  \]
  In particular, in the isometric special case $\mathcal R\equiv \sqrt d\,\mathrm{Id}_B$, one has $\rho_{\mathrm{HS}}=\rho_{\mathrm{op}}=1$.
\end{proposition}

\begin{proof}
  We follow the notation in the proof of \cref{prop:manifold-22-routeB}.
  Since $\Pi_B$ restricts the domain to $B$, the operator $\mathcal P_1=\Pi_1\mathcal P$ has the same nonzero singular values as its restriction to $B$, namely
  $
  \mathcal P_1\vert_B=\Pi_1\mathcal F:B\to \operatorname{range}(\Pi_1).
  $
  Hence
  $
  \|\mathcal P_1^{\dagger}\|_{\mathrm{HS}}
  =
  \|(\mathcal P_1\vert_B)^{\dagger}\|_{\mathrm{HS}}
  $, $
  \|\mathcal P_1^{\dagger}\|_{\mathrm{op}}
  =
  \|(\mathcal P_1\vert_B)^{\dagger}\|_{\mathrm{op}}.
  $
 Define
  $
  \mathcal A
  $ as in \cref{lem:haar-grassmann-A-moments}.
  Then
  $
  \mathcal P_1\vert_B
  =
  \Pi_1\mathcal F
  =
  \Pi_1\mathcal U\mathcal R
  =
  \mathcal A\mathcal R.
  $
  By \cref{lem:haar-grassmann-A-moments}, $\mathcal A$ has full row rank almost surely.
  Since $\mathcal R:B\to B$ is invertible almost surely, it follows that $\mathcal A\mathcal R$ has full row rank almost surely.
  Therefore
  $
  (\mathcal A\mathcal R)^{\dagger}
  =
  \mathcal R^{-1}\mathcal A^{\dagger}
 $
 almost surely.

  For the operator norm, it holds that
  \[
  \|\mathcal P_1^{\dagger}\|_{\mathrm{op}}^{2}
  =
  \|(\mathcal A\mathcal R)^{\dagger}\|_{\mathrm{op}}^{2}
  \le
  \|\mathcal R^{-1}\|_{\mathrm{op}}^{2}\,\|\mathcal A^{\dagger}\|_{\mathrm{op}}^{2}.
  \]
  By~\ref{itm:haar-grassmann-sketch-polar} of \cref{def:haar-grassmann-sketch}, $\mathcal U$ is independent of $\mathcal R$.
  Since $\mathcal A=\Pi_1\mathcal U$ is a measurable function of $\mathcal U$ alone, $\mathcal A$ is independent of $\mathcal R$.
  Taking expectations therefore gives
  \begin{equation}
    \label{eq:haar-op-separate-new}
    \mathbb E\big[\|\mathcal P_1^{\dagger}\|_{\mathrm{op}}^{2}\big]
    \le
    \mathbb E\big[\|\mathcal R^{-1}\|_{\mathrm{op}}^{2}\big]\,
    \mathbb E\big[\|\mathcal A^{\dagger}\|_{\mathrm{op}}^{2}\big]
    =
    \frac{\rho_{\mathrm{op}}}{d}\,\mathbb E\big[\|\mathcal A^{\dagger}\|_{\mathrm{op}}^{2}\big].
  \end{equation}

  For the Hilbert--Schmidt norm, using again $(\mathcal A\mathcal R)^{\dagger}=\mathcal R^{-1}\mathcal A^{\dagger}$ and cyclicity of the trace on $B$, we obtain
  $
  \|\mathcal P_1^{\dagger}\|_{\mathrm{HS}}^{2}
  =
  \|(\mathcal A\mathcal R)^{\dagger}\|_{\mathrm{HS}}^{2} 
  =
  \operatorname{tr}\Big(((\mathcal A\mathcal R)^{\dagger})^{*}(\mathcal A\mathcal R)^{\dagger}\Big) 
  =
  \operatorname{tr}\big(\mathcal A^{\dagger *}\mathcal R^{-2}\mathcal A^{\dagger}\big) 
  =
  \operatorname{tr}\big(\mathcal A^{\dagger}\mathcal A^{\dagger *}\mathcal R^{-2}\big).
  $
  Conditionally on $\Xi$, the isometry $\mathcal U$ is Haar-uniform on the set of linear isometries from $B$ to $\Xi$.
  Hence for every orthogonal operator $\mathcal O:B\to B$,
  $
  \mathcal A\mathcal O
  =
  \Pi_1(\mathcal U\mathcal O)
  \stackrel{d}{=}
  \Pi_1\mathcal U
  =
  \mathcal A.
  $
  Since $(\mathcal A\mathcal O)^{\dagger}=\mathcal O^{*}\mathcal A^{\dagger}$, it follows that
  $
  \mathcal A^{\dagger}\mathcal A^{\dagger *}
  \stackrel{d}{=}
  \mathcal O^{*}\big(\mathcal A^{\dagger}\mathcal A^{\dagger *}\big)\mathcal O.
  $
  Therefore $\mathbb E[\mathcal A^{\dagger}\mathcal A^{\dagger *}]$ commutes with every orthogonal operator on $B$, and hence
  $
  \mathbb E\big[\mathcal A^{\dagger}\mathcal A^{\dagger *}\big]
  =
  c\,\mathrm{Id}_B
  $
  for some scalar $c$.
  Taking traces gives
  $
  c
  =
  \frac{1}{\ell}\,\operatorname{tr}\Big(\mathbb E\big[\mathcal A^{\dagger}\mathcal A^{\dagger *}\big]\Big)
  =
  \frac{1}{\ell}\,\mathbb E\big[\|\mathcal A^{\dagger}\|_{\mathrm{HS}}^{2}\big].
  $
  Thus
  $
  \mathbb E\big[\mathcal A^{\dagger}\mathcal A^{\dagger *}\big]
  =
  \frac{1}{\ell}\,\mathbb E\big[\|\mathcal A^{\dagger}\|_{\mathrm{HS}}^{2}\big]\,
  \mathrm{Id}_B.
  $
  Using again the independence of $\mathcal A$ and $\mathcal R$, it holds that
  \begin{equation}
    \label{eq:haar-hs-separate-new}
    \mathbb E\big[\|\mathcal P_1^{\dagger}\|_{\mathrm{HS}}^{2}\big]
    =
    \mathbb E\Big[\operatorname{tr}\big(\mathcal A^{\dagger}\mathcal A^{\dagger *}\mathcal R^{-2}\big)\Big] 
    =
    \operatorname{tr}\Big(
    \mathbb E\big[\mathcal A^{\dagger}\mathcal A^{\dagger *}\big]\,
    \mathbb E\big[\mathcal R^{-2}\big]
    \Big)
    =
    \frac{\rho_{\mathrm{HS}}}{d}\,\mathbb E\big[\|\mathcal A^{\dagger}\|_{\mathrm{HS}}^{2}\big].
  \end{equation}
  Applying \cref{lem:haar-grassmann-A-moments} in \eqref{eq:haar-hs-separate-new} and \eqref{eq:haar-op-separate-new} yields
  \[
  \mathbb E\big[\|\mathcal P_1^{\dagger}\|_{\mathrm{HS}}^{2}\big]
  \le
  \rho_{\mathrm{HS}}\,\frac{\ell-p}{p-1},
  \qquad
  \mathbb E\big[\|\mathcal P_1^{\dagger}\|_{\mathrm{op}}^{2}\big]
  \le
  C_0\,\rho_{\mathrm{op}}\,\frac{\ell}{p^{2}-1}.
  \]
  Finally,
  $
  \operatorname{tr}(\mathcal R^{-2})
  \le
  \ell\,\|\mathcal R^{-1}\|_{\mathrm{op}}^{2}
  $
  implies
  $
  \rho_{\mathrm{HS}}
  \le
  \rho_{\mathrm{op}}
  <
  \infty
  $
  by~\ref{itm:haar-grassmann-sketch-moment} of \cref{def:haar-grassmann-sketch}.
  In the isometric special case $\mathcal R\equiv \sqrt d\,\mathrm{Id}_B$, one has
  $
  \rho_{\mathrm{HS}}
  =
  \frac{d}{\ell}\,\operatorname{tr}(d^{-1}\mathrm{Id}_B)
  =
  1$,
  $
  \rho_{\mathrm{op}}
  =
  d\,\|d^{-1/2}\mathrm{Id}_B\|_{\mathrm{op}}^{2}
  =
  1.
  $
\end{proof}

\begin{theorem}[Approximation error under Haar--Grassmann sketching]
  \label{thm:manifold-22-haar}
  Let $\ell\ge 4$.
  Suppose $\mathcal{P}_{x, B, \Xi}$ satisfies the Haar--Grassmann sketching condition and let $\rho_{\mathrm{HS}},\rho_{\mathrm{op}}$ and $C_0$ be as in \cref{prop:haar-grassmann-pinv-moments}.
  Then 
  \[
  \begin{aligned}
  \mathbb E\left[ \big\|\mathcal{H}_x-\mathcal{H}_{x, B, \Xi}\big\|_{\mathrm{op}} \right]
  \le
  \min_{2\le p\le \ell-2}
  &\Bigg\{\left(1+2\rho_{\mathrm{HS}}\,\tfrac{\ell-p}{p-1}\right)\lambda_{\ell-p+1}(\mathcal{H}_x)
  \\
  & \quad +2C_0\,\rho_{\mathrm{op}}\,\tfrac{\ell}{p^{2}-1}\sum_{j>\ell-p}\lambda_j(\mathcal{H}_x)
  \Bigg\}.
  \end{aligned}
  \]
\end{theorem}

\begin{proof}
Combining \cref{prop:manifold-22-routeB} with \cref{prop:haar-grassmann-pinv-moments} gives the desired result.
\end{proof}

For $p \in \{1,\dots,d\}$, $\lambda_p(\mathcal{H}_x)>0$, define the $p$-stable rank by
\[
\mathrm{sr}_p(\mathcal{H}_x):=\lambda_p(\mathcal{H}_x)^{-1}\sum_{j=p}^d \lambda_j(\mathcal{H}_x).
\]
The stable rank captures how much of the operator's ``energy'' is concentrated in the top eigenvalues.
The following corollary provides a simplified and more interpretable version of the error. By introducing the concept of stable rank, this result gives a cleaner bound that is easier to work with in practice.
\begin{corollary}
  \label{cor:manifold-23}
  Let $p\ge 2$ and $\ell=2p-1$. 
  suppose that $\mathcal{P}_{x, B, \Xi}$ satisfies the Haar--Grassmann sketching condition, and let $\rho_{\mathrm{HS}},\rho_{\mathrm{op}}$ and $C_0$ be as in \cref{prop:haar-grassmann-pinv-moments}.
  Then it holds that
  \[
  \mathbb E\big[ \|\mathcal{H}_x-\mathcal{H}_{x, B, \Xi}\|_{\mathrm{op}} \big]
  \le
  \left(1+2\rho_{\mathrm{HS}}+\frac{4C_0\rho_{\mathrm{op}}}{p}\,\mathrm{sr}_p(\mathcal{H}_x)\right)\lambda_p(\mathcal{H}_x).
  \]
\end{corollary}

\begin{proof}
Apply \cref{thm:manifold-22-haar} with $\ell=2p-1$ and choose the parameter in the minimum to be the same $p\in\{2,\dots,\ell-2\}$. Since $\ell-p+1=p$, we have
\[
\lambda_{\ell-p+1}(\mathcal{H}_x)=\lambda_p(\mathcal{H}_x),
\qquad
\sum_{j>\ell-p}\lambda_j(\mathcal{H}_x)=\sum_{j>p-1}\lambda_j(\mathcal{H}_x)=\sum_{j=p}^d\lambda_j(\mathcal{H}_x).
\]
The first coefficient simplifies to
$
1+2\rho_{\mathrm{HS}}\,\frac{\ell-p}{p-1}
=
1+2\rho_{\mathrm{HS}}\,\frac{p-1}{p-1}
=
1+2\rho_{\mathrm{HS}}.
$
For the second coefficient, using $\ell=2p-1$ and $p\ge 2$ we have
$
\frac{2\ell}{p^2-1}
=
\frac{2(2p-1)}{p^2-1}
\le
\frac{4}{p}.
$
Substituting these simplifications into \cref{thm:manifold-22-haar} yields
\[
\begin{aligned}
\mathbb E\big[ \|\mathcal{H}_x-\mathcal{H}_{x, B, \Xi}\|_{\mathrm{op}} \big]
&\le
\bigl(1+2\rho_{\mathrm{HS}}\bigr)\lambda_p(\mathcal{H}_x)
+\frac{4C_0\rho_{\mathrm{op}}}{p}\sum_{j=p}^d\lambda_j(\mathcal{H}_x)\\
&=
\left(1+2\rho_{\mathrm{HS}}+\frac{4C_0\rho_{\mathrm{op}}}{p}\,\mathrm{sr}_p(\mathcal{H}_x)\right)\lambda_p(\mathcal{H}_x).
\end{aligned}
\]
\end{proof}

In scientific computing, many algorithms (e.g., variants of Newton-type methods) rely on regularized inverses of operators. The following proposition quantifies how the Riemannian Nystr\"om approximation impacts ridge-regularized inverses. Specifically, it shows that the error in the ridge inverse is controlled predictably by the underlying approximation error.

\begin{proposition}
    \label{prop:manifold-31}
    For any $\nu>0$, it holds that
    \[
    \big\|( \mathcal{H}_{x, B, \Xi}+\nu \,\mathrm{Id}_x)^{-1}-(\mathcal{H}_x+\nu \,\mathrm{Id}_x)^{-1}\big\|_{\mathrm{op}}
    \le
    \frac{\|\mathcal{H}_x-\mathcal{H}_{x,B,\Xi}\|_{\mathrm{op}}}{\nu\,\bigl(\lambda_d(\mathcal{H}_x)+\nu\bigr)}.
    \]
\end{proposition}

\begin{proof}
Let $\Delta_x:=\mathcal{H}_x-\mathcal{H}_{x,B,\Xi}$. By \cref{lem:riem-21}, we have $\Delta_x\succeq 0$ and $\mathcal{H}_{x,B,\Xi}\preceq \mathcal{H}_x$.
For $\nu>0$, both $\mathcal{H}_x+\nu\,\mathrm{Id}_x$ and $\mathcal{H}_{x,B,\Xi}+\nu\,\mathrm{Id}_x$ are invertible.
Using the resolvent identity,
\[
(\mathcal{H}_{x,B,\Xi}+\nu\,\mathrm{Id}_x)^{-1}-(\mathcal{H}_x+\nu\,\mathrm{Id}_x)^{-1}
=(\mathcal{H}_{x,B,\Xi}+\nu\,\mathrm{Id}_x)^{-1}\,\Delta_x\,(\mathcal{H}_x+\nu\,\mathrm{Id}_x)^{-1}.
\]
Taking operator norms and using submultiplicativity yields
$
\big\|(\mathcal{H}_{x,B,\Xi}+\nu\,\mathrm{Id}_x)^{-1}-(\mathcal{H}_x+\nu\,\mathrm{Id}_x)^{-1}\big\|_{\mathrm{op}} 
  \le
  \|(\mathcal{H}_{x,B,\Xi}+\nu\,\mathrm{Id}_x)^{-1}\|_{\mathrm{op}}\,\|\Delta_x\|_{\mathrm{op}}\,\|(\mathcal{H}_x+\nu\,\mathrm{Id}_x)^{-1}\|_{\mathrm{op}}.
$
Since $\mathcal{H}_{x,B,\Xi}\succeq 0$, we have $\|(\mathcal{H}_{x,B,\Xi}+\nu\,\mathrm{Id}_x)^{-1}\|_{\mathrm{op}}\le \nu^{-1}$.
Moreover, $\|(\mathcal{H}_x+\nu\,\mathrm{Id}_x)^{-1}\|_{\mathrm{op}}=(\lambda_d(\mathcal{H}_x)+\nu)^{-1}$.
Substituting these bounds gives the claim.
\end{proof}

The final result combines the approximation error with the ridge inverse perturbation analysis to provide a complete picture of how Riemannian Nystr\"om approximations affect regularized inverses. 

\begin{proposition}
  \label{prop:manifold-42}
  Fix $p\ge 2$ and set $\ell=2p-1$.
  Suppose $\mathcal{P}_{x, B, \Xi}$ satisfies the Haar--Grassmann sketching condition, and let $\rho_{\mathrm{HS}},\rho_{\mathrm{op}}$ and $C_0$ be as in \cref{prop:haar-grassmann-pinv-moments}.
  Then for any $\nu>0$, it holds that
  \[
  \begin{aligned} 
  & \mathbb E\left[\big\|(\mathcal{H}_x+\nu \mathrm{Id}_x)^{-1}-(\mathcal{H}_{x, B, \Xi}+\nu \mathrm{Id}_x)^{-1}\big\|_{\mathrm{op}}\right]
  \\
  &
  \le
  \left(1+2\rho_{\mathrm{HS}}+\frac{4C_0\rho_{\mathrm{op}}}{p}\,\mathrm{sr}_p(\mathcal{H}_x)\right)
  \frac{\lambda_p(\mathcal{H}_x)}{\nu(\lambda_d(\mathcal{H}_x)+\nu)}.
  \end{aligned}
  \]
\end{proposition}

\begin{proof}
Taking expectations in \cref{prop:manifold-31} and using \cref{cor:manifold-23} to bound $\mathbb E [ \|\mathcal{H}_x-\mathcal{H}_{x,B,\Xi}\|_{\mathrm{op}} ]$ yields the stated inequality.
\end{proof}

\subsection{Transported sketching}
\label{subsec:transported-sketching}
In iterative methods on manifolds, the sketching is typically refreshed at each iterate.
When two points $x,x'\in\mathcal M$ are close and an isometric vector transport $\mathcal T_{x\to x'}$ is available, one may instead transport the sketching from $\mathrm T_x\mathcal M$ to $\mathrm T_{x'}\mathcal M$.
This reuses a common low-dimensional subspace structure across nearby iterates.
In this sense, the Haar--Grassmann sketching condition is a transport compatible geometric structure.

Let $\mathcal T_{x\to x'}:\mathrm T_x\mathcal M\to \mathrm T_{x'}\mathcal M$ be an isometric vector transport, namely
\[
\langle \mathcal T_{x\to x'}u,\mathcal T_{x\to x'}v\rangle_{x'}=\langle u,v\rangle_x, \quad \forall u,v\in\mathrm T_x\mathcal M.
\]
Given subspaces $B,\Xi\subset\mathrm T_x\mathcal M$ with $\dim(B)=\dim(\Xi)=\ell$, define their transported counterparts by
\[
B' := \mathcal T_{x\to x'}(B)\subset\mathrm T_{x'}\mathcal M,
\qquad
\Xi' := \mathcal T_{x\to x'}(\Xi)\subset\mathrm T_{x'}\mathcal M.
\]
Given a full-rank map $\mathcal F:B\to\Xi$, define the transported map $\mathcal F':B'\to\Xi'$ by conjugation,
\begin{equation}
\label{eq:transported-F}
\mathcal F'
:= \bigl(\mathcal T_{x\to x'}\vert_{\Xi}\bigr)\,\mathcal F\,\bigl(\mathcal T_{x\to x'}\vert_{B}\bigr)^{-1}.
\end{equation}
The transported sketching operator is then $\mathcal P_{x',B',\Xi'}:=\mathcal F'\Pi_{B'}$.
The following theorem shows that the transported sketching operator satisfies the Haar--Grassmann sketching condition at $x'$.

 \begin{theorem}
 \label{thm:transport-closure-iff-haar}
  If the Haar--Grassmann sketching condition holds at $x \in \mathcal{M}$ for $\mathcal P_{x,B,\Xi}$, then $\mathcal P_{x',B',\Xi'}$ satisfies the Haar--Grassmann sketching condition at $x' \in \mathcal{M}$.
 \end{theorem}

 \begin{proof}

 We verify the three requirements of \cref{def:haar-grassmann-sketch}.
 For~\ref{itm:haar-grassmann-sketch-haar}, let $\nu_x$ be the Haar measure on $\mathrm{Gr}(\ell,\mathrm T_x\mathcal M)$.
 Since $\mathcal T_{x\to x'}$ is an isometry, the pushforward, denoted by $(\mathcal T_{x\to x'})_{\#}\nu_x$, is isometry-invariant on $\mathrm{Gr}(\ell,\mathrm T_{x'}\mathcal M)$, hence equals the Haar measure $\nu_{x'}$ by uniqueness.
 Thus $\Xi\sim\nu_x$ implies $\Xi'=\mathcal T_{x\to x'}(\Xi)\sim\nu_{x'}$.

For~\ref{itm:haar-grassmann-sketch-polar},  define
$
\mathcal U'
:= \bigl(\mathcal T_{x\to x'}\vert_{\Xi}\bigr)\,\mathcal U\,\bigl(\mathcal T_{x\to x'}\vert_{B}\bigr)^{-1},
\mathcal R'
:= \bigl(\mathcal T_{x\to x'}\vert_{B}\bigr)\,\mathcal R\,\bigl(\mathcal T_{x\to x'}\vert_{B}\bigr)^{-1}.
$
Then $\mathcal F'=\mathcal U'\mathcal R'$ with $\mathcal U':B'\to\Xi'$ a linear isometry and $\mathcal R':B'\to B'$ $g_{x'}$-self-adjoint and positive definite.
To prove the conditional Haar property, fix $\Xi$ and consider the conjugation map
$
\Phi_{\Xi}:\mathrm{Iso}(B,\Xi)\to \mathrm{Iso}(B',\Xi'),
\Phi_{\Xi}(U):=(\mathcal T_{x\to x'}\vert_{\Xi})\,U\,(\mathcal T_{x\to x'}\vert_{B})^{-1}.
$
Since $\mathcal T_{x\to x'}\vert_{\Xi}$ and $\mathcal T_{x\to x'}\vert_{B}$ are linear isometries, $\Phi_{\Xi}$ is a bijection. Moreover, for every $O\in\mathrm{Iso}(B,B)$ one has
$\Phi_{\Xi}(U\circ O)=\Phi_{\Xi}(U)\circ O$.
Hence $\Phi_{\Xi}$ pushes forward the unique right-invariant probability measure on $\mathrm{Iso}(B,\Xi)$ to the unique right-invariant probability measure on $\mathrm{Iso}(B',\Xi')$.
Since, conditionally on $\Xi$, the law of $\mathcal U$ is Haar on $\mathrm{Iso}(B,\Xi)$, it follows that, conditionally on $\Xi$, the law of $\mathcal U'=\Phi_{\Xi}(\mathcal U)$ is Haar on $\mathrm{Iso}(B',\Xi')$.

Now let $\psi$ be any bounded measurable test function on the disjoint union of the spaces $\mathrm{Iso}(B',W)$ with $W\in\mathrm{Gr}(\ell,\mathrm T_{x'}\mathcal M)$. Then
$
\mathbb E\big[\psi(\mathcal U')\mid \Xi\big]
=\int_{\mathrm{Iso}(B',\Xi')}\psi(V)\,d\mu_{\Xi'}(V),
$
where $\mu_{\Xi'}$ denotes the Haar probability measure on $\mathrm{Iso}(B',\Xi')$.
The right-hand side depends on $\Xi$ only through $\Xi'$, and $\Xi'$ is a measurable function of $\Xi$.
Therefore, by the tower property,
$
\mathbb E\big[\psi(\mathcal U')\mid \Xi'\big]
=\int_{\mathrm{Iso}(B',\Xi')}\psi(V)\,d\mu_{\Xi'}(V),
$
which proves that, conditionally on $\Xi'$, the map $\mathcal U'$ is Haar-uniform on $\mathrm{Iso}(B',\Xi')$.

For~\ref{itm:haar-grassmann-sketch-moment}, note first that $\mathcal R'$ is a measurable function of $\mathcal R$ alone, while $(\Xi',\mathcal U')$ is a measurable function of $(\Xi,\mathcal U)$, hence of $\mathcal U$.
Since $\mathcal U$ is independent of $\mathcal R$ by~\ref{itm:haar-grassmann-sketch-polar}, it follows that $\mathcal U'$ is independent of $\mathcal R'$.
Moreover, conjugation by the isometry $\mathcal T_{x\to x'}\vert_B$ preserves adjoints, positivity, functional calculus, traces, and operator norms.
Consequently,
\[
\mathbb E\big[(\mathcal R')^{2}\big]
=\bigl(\mathcal T_{x\to x'}\vert_B\bigr)
\,\mathbb E\big[\mathcal R^{2}\big]\,
\bigl(\mathcal T_{x\to x'}\vert_B\bigr)^{-1}
=d\,\mathrm{Id}_{B'},
\]
and
$
\| (\mathcal R')^{-1}\|_{\mathrm{op}}=\|\mathcal R^{-1}\|_{\mathrm{op}},
\mathbb E\,\|(\mathcal R')^{-1}\|_{\mathrm{op}}^{2}<\infty.
$
Thus,~\ref{itm:haar-grassmann-sketch-moment} also holds at $x'$.
\end{proof}

\begin{remark}
In Riemannian optimization,
consider a retraction-based iterative method of the form
\[
  x^{\prime}=\mathrm{Retr}_x(\eta) \quad \text{with } \eta\in \mathrm T_x\mathcal M.
\]
A canonical choice of a transport map is the differentiated retraction
\[
  \widetilde{\mathcal T}_{x\to x^{\prime}}\ :=\ D\mathrm{Retr}_x(\eta):\mathrm T_x\mathcal M\to \mathrm T_{x^{\prime}}\mathcal M,
\]
which is convenient to evaluate in most implementations. When the two points are close, the transport map is approximately isometric.
A standard way to enforce an isometry is to take the isometric factor in the polar decomposition: define
\[
  \mathcal T_{x\to x^{\prime}}
  :=\bigl(\widetilde{\mathcal T}_{x\to x^{\prime}}\,\widetilde{\mathcal T}_{x\to x^{\prime}}^{*}\bigr)^{-1/2}\,\widetilde{\mathcal T}_{x\to x^{\prime}},
\]
where \(\widetilde{\mathcal T}_{x\to x^{\prime}}^{*}\) denotes the adjoint with respect to \(\langle\cdot,\cdot\rangle_x\) and \(\langle\cdot,\cdot\rangle_{x^{\prime}}\).
Then \(\mathcal T_{x\to x^{\prime}}\) is an isometry.
In the special case of the exponential map, one may take \(\mathcal T_{x\to x^{\prime}}\) to be parallel transport along the geodesic from \(x\) to \(x^{\prime}\), which is an exact isometry by construction.

Hence, in a practical implementation of an iterative method, one can adopt a lazy refresh strategy. At iteration \(k\), transport the previously constructed \((B_k,\Xi_k,\mathcal F_k)\) from \(\mathrm T_{x_k}\mathcal M\) to \(\mathrm T_{x_{k+1}}\mathcal M\) via the (approximately) isometric map \(\mathcal T_{x_k\to x_{k+1}}\), and use the transported triple \((B_{k+1},\Xi_{k+1},\mathcal F_{k+1})\) as the sketching at the next iterate.
Every few iterations the transported sketching is discarded and a fresh sketching operator is generated at the current point.
\end{remark}

\section{Riemannian Nystr\"om approximation with coordinate representations}
\label{sec:computations}

This section develops coordinate representations and computational formulas for the Riemannian Nystr\"om approximation introduced in the preceding sections.

\subsection{Coordinate representation}

Fix an $\ell$-dimensional subspace $B\subset \mathrm T_x\mathcal M$ with a $g_x$-orthonormal basis $\{b_{x,i}\}_{i=1}^\ell$, and consider another $\ell$-dimensional subspace $\Xi \subset \mathrm T_x\mathcal M$ spanned by vectors $\{\xi_{x,i}\}_{i=1}^\ell$ that are not necessarily $g_x$-orthogonal.
These spanning vectors determine a full-rank map $\mathcal F:B\to\mathrm T_x\mathcal M$ via
\[
  \mathcal F[b_{x,i}]=\xi_{x,i}, \qquad i=1,\dots,\ell.
\]
In this notation, the sketching operators~\eqref{eq:Px-coordfree}--\eqref{eq:Px*-coordfree} take the coordinate form
\begin{equation}
  \label{eq: Px}
\mathcal{P}_{x, B, \Xi}[v]=\sum_{i =1}^{\ell}\left\langle b_{x, i}, v\right\rangle_{x} \, \xi_{x, i},\quad \text{ for all } v \in \mathrm T_{x} \mathcal{M},
\end{equation}
and the adjoint identity
$
  \langle\mathcal{P}_{x, B, \Xi}[v], u\rangle_{x}=\langle\mathcal{P}_{x, B, \Xi}^{*}[u], v\rangle_{x}
$
for all $u,v\in \mathrm T_{x} \mathcal{M}$ immediately yields
\begin{equation}
  \label{eq: Px*}
\mathcal{P}_{x, B, \Xi}^{*}[u]=\sum_{i =1}^{\ell}\left\langle \xi_{x, i}, u\right\rangle_{x} \, b_{x, i},\quad  \text{for all } u \in \mathrm T_{x} \mathcal{M}.
\end{equation}
To pass to matrix notation, extend $\{b_{x,i}\}_{i=1}^\ell$ to a $g_x$-orthonormal basis $\{b_{x,i}\}_{i=1}^d$ of $\mathrm T_x\mathcal M$ and write
\begin{equation}
  \label{eq: pxi}
  \xi_{x, i}=\sum_{j=1}^d \omega_{ji}b_{x, j}, \quad i =1,\dots,\ell,
\end{equation}
where $\mathbf{\Omega}=(\omega_{ji})\in\mathbb R^{d\times\ell}$ is the coefficient matrix.
This representation suggests a convenient way to realize $\Xi$ and $\mathcal F$ under the Haar--Grassmann sketching condition.
After fixing $g_x$-orthonormal frames for $\mathrm T_x\mathcal M$ and $B$, factor
$\mathbf{\Omega}=\mathbf{Q}\,\mathbf{R}$, where $\mathbf{Q}\in\mathbb R^{d\times\ell}$ has orthonormal columns and is Haar-uniform on the Stiefel manifold $\mathrm{St}(d,\ell)$, and $\mathbf{R}\in\mathbb R^{\ell\times\ell}$ is positive definite and independent of $\mathbf{Q}$.
Indeed,
$
  \mathcal F[b_{x,i}]=\xi_{x,i}=\sum_{j=1}^d \mathbf{\Omega}_{ji}\,b_{x,j},\  i=1,\dots,\ell,
$
and hence $\Xi=\operatorname{range}(\mathcal F)=\operatorname{range}(\mathbf{\Omega})=\operatorname{range}(\mathbf{Q})$.
Since $\mathbf{Q}$ is Haar-uniform, this yields~\ref{itm:haar-grassmann-sketch-haar} of \cref{def:haar-grassmann-sketch}.
At the same time, $\mathbf{\Omega}=\mathbf{Q}\mathbf{R}$ corresponds to the polar decomposition $\mathcal F=\mathcal U\mathcal R$, with $\mathcal U:B\to \Xi$ represented by $\mathbf{Q}$ and $\mathcal R:B\to B$ represented by $\mathbf{R}$, thus~\ref{itm:haar-grassmann-sketch-polar} is built in.
Condition~\ref{itm:haar-grassmann-sketch-moment} then reduces to $\mathbb E[\mathbf{R}^2]=d \mathbf{I}_\ell$ and $\mathbb E[\|\mathbf{R}^{-1}\|_2^2]<\infty$.
The special case $\mathbf{R} = \sqrt{d}\, \mathbf{I}_\ell$ gives an isometric sketching.

\subsection{Inversion without regularization}
Many numerical algorithms on Riemannian manifolds require solving a linear system of the form
\begin{equation}
\label{eq:tx-linear-system}
\mathcal{H}_x[u]=b,\qquad u\in \mathrm{T}_x\mathcal M.
\end{equation}
To reduce the computational cost, we approximate the solution using the Nystr\"om pseudoinverse~\eqref{eq:Nystrom-Hessian-inverse}:
\[
u\ :=\ \mathcal{H}_{x,B,\Xi}^{\dagger}[b]
=\mathcal{P}_{x,B,\Xi}\,(\mathcal{P}_{x,B,\Xi}^{*}\mathcal{H}_x\mathcal{P}_{x,B,\Xi})^{\dagger}\,\mathcal{P}_{x,B,\Xi}^{*}[b].
\]
Equivalently, $u$ is the minimum-norm solution of the normal equations in the subspace $B$.
The following proposition provides a formula for evaluating this expression.

\begin{proposition}
  \label{prop:manifold-computational}
  Assume that $\mathcal{P}_{x, B, \Xi}^* \mathcal{H}_x \mathcal{P}_{x, B, \Xi}: B \to B$ has full rank.
  Then
\[  
\mathcal{H}_{x, B, \Xi}^{\dagger}[b]=\sum_{i =1}^{\ell}\left({\mathbf{Q}}^{\dagger} a\right)_{i} \xi_{x, i},
\]
where $\mathbf{Q} \in \mathbb{R}^{\ell \times \ell}$ and ${a} \in \mathbb{R}^{\ell}$ are given by
\begin{equation}
  \label{eq: Q and a}
\mathbf{Q}=\left[\begin{array}{ccc}
\left\langle \xi_{x, 1}, \mathcal{H}_x\left[\xi_{x, 1}\right]\right\rangle_{x} & \cdots & \left\langle \xi_{x, 1}, \mathcal{H}_x\left[\xi_{x, \ell}\right]\right\rangle_{x} \\
\vdots & \ddots & \vdots \\
\left\langle \xi_{x, \ell}, \mathcal{H}_x\left[\xi_{x, 1}\right]\right\rangle_{x} & \cdots & \left\langle \xi_{x, \ell}, \mathcal{H}_x\left[\xi_{x, \ell}\right]\right\rangle_{x}
\end{array}\right], \quad 
{a}=\left[\begin{array}{c}
\left\langle \xi_{x, 1}, b\right\rangle_{x} \\
\vdots \\
\left\langle \xi_{x, \ell}, b\right\rangle_{x}
\end{array}\right].
\end{equation}
\end{proposition}
\begin{proof}
Let $v \in B$ satisfy the linear equation
\begin{equation}
  \label{eq: linear system for v}
(\mathcal{P}_{x, B, \Xi}^* \mathcal{H}_x \mathcal{P}_{x, B, \Xi})[v] = \mathcal{P}_{x, B, \Xi}^* [b].
\end{equation}
Since $\mathcal{P}_{x, B, \Xi}^*[b] \in B = \mathrm{range}(\mathcal{P}_{x, B, \Xi}^* \mathcal{H}_x \mathcal{P}_{x, B, \Xi})$, $v = (\mathcal{P}_{x, B, \Xi}^* \mathcal{H}_x \mathcal{P}_{x, B, \Xi})^{\dagger} \mathcal{P}_{x, B, \Xi}^* [b]$ is a solution to~\eqref{eq: linear system for v}.
By~\eqref{eq: Px} and~\eqref{eq: Px*}, we compute
\begin{equation*} 
\left(\mathcal{P}_{x, B, \Xi}^{*} \mathcal{H}_x \mathcal{P}_{x, B, \Xi}\right)[v] =\sum_{i=1}^{\ell}\left\langle \xi_{x, i}, \sum_{j=1}^{\ell}\left\langle b_{x, j}, v\right\rangle_{x} \mathcal{H}_x\left[\xi_{x, j}\right]\right\rangle_{x} b_{x, i}.
\end{equation*}
Because $\{b_{x, i}\}_{i=1}^{\ell}$ is a $g_x$-orthonormal basis of $B$, equating the $i$-th component of~\eqref{eq: linear system for v} gives
\[
\left\langle \xi_{x, i}, \sum_{j=1}^{\ell}\left\langle b_{x, j}, v\right\rangle_{x} \mathcal{H}_x[\xi_{x, j}]\right\rangle_{x}  = \left\langle \xi_{x, i}, b\right\rangle_{x}.
\]
Setting $c_{i} = \left\langle b_{x, i}, v\right\rangle_{x}$, the above system is equivalent to $\mathbf{Q} c = a$, i.e., $c = \mathbf{Q}^{\dagger} a$.
Hence it holds that
\[
\mathcal{P}_{x, B, \Xi}\left(\mathcal{P}_{x, B, \Xi}^{*} \mathcal{H}_x \mathcal{P}_{x, B, \Xi}\right)^{\dagger} \mathcal{P}_{x, B, \Xi}^{*} [b] 
= \mathcal{P}_{x, B, \Xi}[v]
=\sum_{i=1}^{\ell} \left\langle b_{x, i}, v\right\rangle_{x} \xi_{x, i} 
=\sum_{i=1}^{\ell}  (\mathbf{Q}^{\dagger} a)_{i} \xi_{x, i} .
\]
\end{proof}
\begin{remark}
  \label{remark:manifold-computational-1}
  The proof shows that $u=\mathcal{H}_{x,B,\Xi}^{\dagger}[b]$ can be written as $u=\mathcal{P}_{x,B,\Xi}[v]$, where $v\in B$ solves~\eqref{eq: linear system for v}.
  Thus \cref{prop:manifold-computational} reduces the application of the Nystr\"om pseudoinverse (equivalently, computing the minimum-norm solution supported on $\operatorname{range}(\mathcal{H}_{x,B,\Xi})$) to two main operations: (i) forming the basis vectors $\{b_{x, i}\}_{i =1}^{\ell}$ and $\{\xi_{x, i}\}_{i =1}^{\ell}$; (ii) assembling the $\ell\times\ell$ matrix $\mathbf{Q}$ and vector $a$, then solving the $\ell$-dimensional system $\mathbf{Q} c = a$.
\end{remark}

\subsection{Inversion with regularization}
\label{subsec:regularized-riem-linear-systems}
To mitigate ill conditioning, one often solves the linear system~\eqref{eq:tx-linear-system} in a regularized form. We present an efficient routine based on the Riemannian Nystr\"om approximation.

Let $\mathcal J_{x,B,\Xi}:=\mathcal{P}_{x,B,\Xi}^{*}\,\mathcal{H}_x\,\mathcal{P}_{x,B,\Xi}:B\to B$.
Based on the Nystr\"om pseudoinverse~\eqref{eq:Nystrom-Hessian-inverse}, we define the regularized solution by adding a ridge term $\nu\mathrm{Id}_B$ with $\nu > 0$:
\begin{equation}
\label{eq:regularized-reduced-system}
u = \mathcal{P}_{x,B,\Xi} (\mathcal J_{x,B,\Xi} + \nu \mathrm{Id}_B)^{-1} \mathcal{P}_{x,B,\Xi}^{*} [b].
\end{equation}
The next proposition provides an explicit coordinate formula for this expression.

\begin{proposition}
  \label{prop:manifold-regularized-computational}
  Let $b\in \mathrm{T}_x\mathcal M$, and let $\mathbf{Q}\in\mathbb R^{\ell\times\ell}$ and $a\in\mathbb R^{\ell}$ be defined as in~\eqref{eq: Q and a}.
  Then
  \[
  u=\sum_{i=1}^{\ell}\bigl((\mathbf{Q}+\nu \mathbf{I})^{-1}a\bigr)_i\,\xi_{x,i}.
  \]
\end{proposition}

\begin{proof}
Expand $v\in B$ in the basis $\{b_{x,i}\}_{i=1}^{\ell}$ as $v=\sum_{i=1}^{\ell}c_i b_{x,i}$ with $c_i=\langle b_{x,i},v\rangle_x$.
By the same calculation as in the proof of \cref{prop:manifold-computational}, the $i$-th component of $\mathcal J_{x,B,\Xi}[v]$ in this basis equals $\sum_{j=1}^{\ell}Q_{ij}c_j$.
Since $\{b_{x,i}\}_{i=1}^{\ell}$ is $g_x$-orthonormal and $\mathrm{Id}_{B}[v]=v$, the system~\eqref{eq:regularized-reduced-system} reduces to
$
(\mathbf{Q}+\nu \mathbf{I})\,c=a.
$
Because $\mathbf{Q}\succeq 0$ and $\nu>0$, the matrix $\mathbf{Q}+\nu \mathbf{I}$ is positive definite, so $c=(\mathbf{Q}+\nu \mathbf{I})^{-1}a$.
Applying~\eqref{eq: Px} gives $u=\mathcal{P}_{x,B,\Xi}[v]=\sum_{i=1}^{\ell}c_i\,\xi_{x,i}$, which yields the stated formula.
\end{proof}

  We briefly compare two alternative regularization strategies.
  
  \paragraph{Ridge regularization of the lifted Nystr\"om operator}
  One may regularize the full Nystr\"om operator and solve
  $
  (\mathcal{H}_{x, B, \Xi} + \nu \,\mathrm{Id}_x)[u] = b.
  $
  Assuming that $\mathcal J_{x,B,\Xi}$ is invertible, the Sherman--Morrison--Woodbury formula yields
  \[
    u
    =\nu^{-1}b
    -\nu^{-2}\,\mathcal H_x\,\mathcal P_{x,B,\Xi}\,\bigl(\mathcal J_{x,B,\Xi}+\nu^{-1}\mathcal W_{x,B,\Xi}\bigr)^{-1}\,\mathcal P_{x,B,\Xi}^{*}\,\mathcal H_x[b],
  \]
  where
  $
    \mathcal W_{x,B,\Xi}:=\mathcal P_{x,B,\Xi}^{*}\,\mathcal H_x^{2}\,\mathcal P_{x,B,\Xi}.
  $
  The perturbation bound in \cref{prop:manifold-42} also applies to this variant. However, forming $\mathcal W_{x,B,\Xi}$ requires additional applications of $\mathcal H_x^2$, which may be substantially more expensive.
  
  \paragraph{Regularization inside the reduced inverse}
  Alternatively, one may replace only the pseudoinverse $(\mathcal{P}_{x,B,\Xi}^{*}\mathcal{H}_x\mathcal{P}_{x,B,\Xi})^{\dagger}$ in~\eqref{eq:Nystrom-Hessian} by the ridge inverse $(\mathcal J_{x,B,\Xi}+\nu\,\mathrm{Id}_B)^{-1}$.
  In coordinates, writing $u=\sum_{i=1}^{\ell}c_i\xi_{x,i}$, this leads to
  $
    \mathbf{Q}\,(\mathbf{Q}+\nu \mathbf{I})^{-1}\mathbf{Q}\,c=a.
  $
  The coefficient matrix $\mathbf{Q}(\mathbf{Q}+\nu \mathbf{I})^{-1}\mathbf{Q}$ is typically less well conditioned than $\mathbf{Q}+\nu \mathbf{I}$ in~\eqref{eq:regularized-reduced-system}. For this reason, the latter is usually preferable in practice, unless it is essential to preserve the outer $\mathcal H_x$ factors in the Nystr\"om construction.

\section{Riemannian Nystr\"om approximation in optimization}
\label{sec:riemannian-newton} 

This section investigates how the Riemannian Nystr\"om approximation can be used in optimization on manifolds. A Riemannian optimization problem takes the form
\[
\min_{x \in \mathcal{M}} \quad f(x),
\]
where $f \colon \mathcal{M} \to \mathbb{R}$ is a smooth objective function. We begin by introducing a Riemannian subspace method based on the projection operator from Section~\ref{subsec:Nystrom-manifold} and discussing its relationship to existing methods. We then develop a Newton-type method in which the gradient is projected onto a low-dimensional subspace and the Riemannian Hessian is replaced by its Riemannian Nystr\"om approximation, so that the Newton step is computed by solving a reduced system in the sketched tangent space. Finally, we combine this framework with cubic regularization to obtain a practical second-order method on manifolds.

Let $R$ be a retraction operator on the manifold $(\mathcal{M}, g)$.
A Riemannian subspace gradient step is defined by
\[
x^{+}=R_{x}\bigl(-\alpha\, \mathcal{P}_{x,B,\Xi}^{*}[\grad f(x)]\bigr),
\]
where $\alpha>0$ is the step size and $x^{+}$ denotes the next iterate.
When $\ell=d$, $B=\mathrm{T}_{x}\mathcal{M}$, and $\mathcal{P}_{x,B,\Xi}$ is the identity operator for every iterate $x$, this update reduces to the standard Riemannian gradient step
$
x^{+}=R_{x}\bigl(-\alpha\,\grad f(x)\bigr)
$.

  In coordinates, by letting $\xi_{x, i} = b_{x, i}$, the subspace gradient method reduces to Riemannian coordinate descent~\cite{han2024riemannian}.
A single-coordinate step at $x$ is obtained by selecting an index $i \in \{1,\dots,d\}$ and using the projection
$
\mathcal{P}_{x,\{i\}}[u] := \langle u,b_{x,i}\rangle_x b_{x,i}, u \in \mathrm{T}_x \mathcal{M}.
$
The corresponding search direction $\eta \in \mathrm{T}_x\mathcal{M}$ is given by
\[
\eta_j 
=
\begin{cases}
-\langle \operatorname{grad} f(x),b_{x,i}\rangle_x b_{x,i}, & i = j, \\
0, & i \neq j.
\end{cases}
\]
and a coordinate descent step takes the form
$
x^+ = R_{x}(\alpha \eta).
$
At each iteration, $i$ is chosen according to a prescribed rule (for example, cyclic or random). This is exactly the projected scheme with one-dimensional subspaces $B = \operatorname{span}\{b_{x,i}\}$. 

\subsection{Randomized Riemannian Nystr\"om Newton-type method}
\label{sec:ssn-abstract}
We now introduce a randomized Nystr\"om Newton-type method on Riemannian manifolds, in which the inverse of the full Riemannian Hessian is approximated by its Riemannian Nystr\"om approximation. The resulting scheme is suitable for high-dimensional geodesically convex problems, where forming and inverting the full Hessian is prohibitively expensive.

Let $f\colon \mathcal M\to\mathbb R$ be a smooth geodesically convex function. At a point $x\in\mathcal M$, denote the Riemannian Hessian by
\[
\mathcal{H}_x:=\Hess f(x)\colon \mathrm{T}_x\mathcal M\to \mathrm{T}_x\mathcal M.
\] 
The standard Riemannian Newton method computes a search direction $\eta \in \mathrm{T}_x\mathcal M$ by solving
\[
\mathcal{H}_x [\eta] = -\grad f(x).
\]
Replacing $\mathcal{H}_x$ by its low-rank Riemannian Nystr\"om approximation $\mathcal{H}_{x, B, \Xi}$, the direction is computed as the minimum-norm solution induced by the pseudoinverse:
\[
\eta = -\mathcal{H}_{x, B, \Xi}^{\dagger}[\grad f(x)].
\]
Setting $\mathcal{J}_{x, B, \Xi} := \mathcal{P}_{x, B, \Xi}^* \mathcal{H}_x \mathcal{P}_{x, B, \Xi}$ and $h_{x, B, \Xi} := \mathcal{P}_{x, B, \Xi}^* [\grad f(x)]$, \cref{remark:manifold-computational-1} shows that this direction can equivalently be written as
$\eta = \mathcal{P}_{x, B, \Xi} [v]$, where $v \in B$ solves the reduced system $\mathcal{J}_{x, B, \Xi} [v] = - h_{x, B, \Xi}$.
When $\mathcal{J}_{x, B, \Xi}$ is singular, a regularization term $\nu \mathrm{Id}_B$ with $\nu \geq 0$ is added following Section~\ref{subsec:regularized-riem-linear-systems}, leading to the regularized system
\begin{equation}
  \label{eq: Newton step with regularization}
(\mathcal{J}_{x, B, \Xi}+\nu \mathrm{Id}_B) [v] = - h_{x, B, \Xi}.
\end{equation}
In practice, $\nu$ is either set to zero or determined by a model-based rule.

\begin{remark}
  The direction $\eta = -\mathcal{P}_{x, B, \Xi}[v]$ lies in the subspace $\Xi$. 
  For a fixed subspace, it may contain no descent direction.
  This issue is mitigated by randomized sketching where the distribution of the subspace $\Xi$ has support over $\ell$-dimensional subspaces of $\mathrm{T}_x\mathcal M$. Consequently, $\Xi$ explores the tangent space in a probabilistic sense, and with positive probability it captures a nontrivial component of $\grad f(x)$, yielding a descent direction. 
\end{remark}

To obtain a scalable and globally convergent second-order method on manifolds, 
we combine the randomized Riemannian Nystr\"om Newton method with cubic regularization. This method may be viewed as a Riemannian Nystr\"om analogue of stochastic subspace cubic Newton methods in Euclidean spaces~\cite{hanzely2020stochastic}, with the key differences being the Riemannian Nystr\"om approximations for Hessians and the retraction step.
Based on~\eqref{eq: Newton step with regularization}, we consider the cubic subspace model $\phi \colon B\to\mathbb{R}$,
\begin{equation}
\label{eq:cubic-model}
\phi(v)\;=\langle h_{x, B, \Xi}, v\rangle_x + \tfrac{1}{2}\langle \mathcal{J}_{x, B, \Xi} [v], v\rangle_x + \tfrac{\sigma}{6}\,\|v\|_x^{3},\qquad v\in B,
\end{equation}
with a cubic regularization parameter $\sigma>0$.
The first-order optimality condition is equivalent to the Mor\'e--Sorensen-type system~\cite{more1983computing}
\[
\bigl(\mathcal{J}_{x, B, \Xi}+\nu \mathrm{Id}_B\bigr)[v]=-h_{x, B, \Xi},\qquad \nu=\tfrac{\sigma}{2}\|v\|_x,\qquad \nu\ge 0.
\]
Minimizing the unconstrained cubic model is equivalent to satisfying the trust-region KKT conditions. Hence the subproblem~\eqref{eq:cubic-model} can be solved via Riemannian trust-region methods~\cite{boumal2019global}.
The complete randomized Riemannian Nystr\"om cubic Newton method (RRNCN) is summarized in Algorithm~\ref{alg:sscn-riem}.

\begin{algorithm}[H]
  \caption{Randomized Riemannian Nystr\"om cubic Newton method on $(\mathcal{M},g)$}
  \label{alg:sscn-riem}
  \begin{algorithmic}[1]
  \Require manifold $(\mathcal{M},g)$, objective function $f$, second-order retraction operator $R$, initial point $x_0\in\mathcal{M}$, sketch size $\ell < d$, cubic regularization parameter $\sigma_k>0$.
  \State Choose $x_0\in\mathcal{M}$.
  \For{$k=0,1,2,\dots, K-1$}
  \State Compute $\grad f(x_k)\in \mathrm{T}_{x_k}\mathcal{M}$.
  \State Choose subspace $B_k, \Xi_k$, map $\mathcal{F}_k$ satisfying the Haar--Grassmann sketching condition.
  \State Build $h_k:=\mathcal{P}_{x_k, B_k, \Xi_k}^* [\grad f(x_k)]$.
  \State Solve $v_k \in B_k$ from the subproblem:\[
  v_k \in \arg\min_{v\in B_k} \phi_k(v) := \langle h_k, v\rangle_{x_k} + \tfrac{1}{2}\langle \mathcal{J}_{x_k, B_k, \Xi_k} [v], v\rangle_{x_k} + \tfrac{\sigma_k}{6}\|v\|_{x_k}^{3}.
  \]
  \State Set $\eta_k:=\mathcal{P}_{x_k, B_k, \Xi_k}(v_k) \in \mathrm{T}_{x_k}\mathcal M$.
  \State Update $x_{k+1}=R_{x_k}(\eta_k)$.
  \State (Optional) Reject the update if $f(R_{x_k}(\eta_k)) > f(x_k)$.
  \EndFor
  \end{algorithmic}
  \end{algorithm}

  \ifarxivversion

\subsection{Convergence Analysis}
\label{sec:convergence-analysis}
To analyze the convergence properties of the randomized Nystr\"om Riemannian cubic Newton method, we first establish the required regularity conditions on the objective function and the geometric properties of the manifold.
For convenience, we denote $\mathcal{P}_{x_k, B_k, \Xi_k}, \mathcal{P}_{x_k, B_k, \Xi_k}^*$, $h_{x_k, B_k, \Xi_k}$, $\mathcal{J}_{x_k, B_k, \Xi_k}$ by $\mathcal{P}_k$, $\mathcal{P}_k^*$, $h_k$, $\mathcal{J}_k$, respectively.
The convergence analysis is based on the following assumptions.
\begin{assumption}
\label{ass:smooth}
The following conditions hold:
\begin{enumerate}
\item[1.] The function $f$ is twice continuously differentiable and geodesically convex on $\mathcal{M}$.

\item[2.] For every $x\in\mathcal{M}$, the pullback $\hat f:=f\circ R_x$ has an $L_3$-Lipschitz Hessian on $\mathrm{T}_x\mathcal{M}$, i.e., $\|\nabla^{2}\hat f(u)-\nabla^{2}\hat f(0)\|_{\mathrm{op}}\le L_3\|u\|_x$ for all $u\in\mathrm{T}_x\mathcal{M}$, and $\Hess f(x) \preceq L_2 \mathrm{Id}_x$.

\item[3.] The retraction operator $R_x$ is second-order on $\mathrm{T}_x\mathcal{M}$ for all $x\in\mathcal{M}$. 

\item[4.] Assume that the sequence $\{\mathcal{P}_k\}$ is independent and satisfies the Haar--Grassmann sketching condition.
\end{enumerate}
\end{assumption} 

It is straightforward to verify that the Haar--Grassmann sketching condition implies that
$
\mathbb E\big[\Pi_{\Xi_k}[v]\big]=\frac{\ell}{d}\,v
$
for all $v\in \mathrm{T}_{x_k}\mathcal{M}$.
We now present the global complexity theorem for the randomized Riemannian Nystr\"om cubic Newton method.  For large $k$, the second term in the bound is dominant. This shows that the number of iterations to reach accuracy $\epsilon$ is $O(\frac{d}{\ell}\epsilon^{-1})$. The technical lemmas are provided in Appendix~\ref{app:global-complexity}. 
\begin{theorem}\label{thm:global-complexity}
  \textbf{(Global complexity)}
Suppose \cref{ass:smooth} holds.
Let $x^\star\in\arg\min f$ and $r>0$ denote the sublevel-radius bound
\[
r:=\sup\left\{\operatorname{dist}(x,x^\star)\mid f(x)\le f(x_0) \right\}<\infty.
\]
There exists a constant $M>0$ depending only on $L_2, L_3, r$ such that by setting $\sigma_k \geq M$ for all $k$, it holds that for all $k\ge 1$,
\[
\mathbb{E}\left[f(x_k)\right]-f^*\le \frac{f(x_0)-f^*}{\big(1+\tfrac{\ell}{4d}k\big)^{3}}+\frac{9\,L_2 r^{2}}{k}\frac{d}{\ell}+\frac{4.5\,M r^{3}}{k^{2}}\Big(\frac{d}{\ell}\Big)^{2}.
\]
\end{theorem}

The next theorem describes the local behavior of the method and establishes a linear convergence rate under strong geodesic convexity and a transported resampling assumption. The corresponding technical lemmas are collected in Appendix~\ref{app:local-convergence-rates}. Throughout the local analysis, we fix a local minimizer $x^\star$ of $f$.
\begin{assumption}
  (\textbf{strong geodesic convexity})
  \label{ass:strong-geo-convexity}
There exists a geodesically convex neighborhood $U$ of $x^\star$ such that $\Hess f(x)\succeq \mu \,\mathrm{Id}_x$ for all $x\in U$.
\end{assumption}

For each $x\in U$, let
  $
  \mathcal T_{x^\star\to x}:\mathrm{T}_{x^\star}\mathcal{M}\to \mathrm{T}_x\mathcal{M}
  $
  denote an isometric vector transport (e.g., the parallel transport along the unique minimizing geodesic from $x^\star$ to $x$), and write $\mathcal T_{x\to x^\star}:=(\mathcal T_{x^\star\to x})^{-1}$.
  Define the pullback of $\mathcal{P}_k$ to the reference tangent space $\mathrm{T}_{x^\star}\mathcal{M}$ by
  \[
  \widetilde{\mathcal P}_k
  :=
  \mathcal T_{x_k\to x^\star}\,\mathcal P_k\,\mathcal T_{x^\star\to x_k}.
  \]

\begin{assumption}
  \label{ass:local-transported-resampling}
  Assume that $\mathcal P_{x^\star,B^\star,\Xi^\star}$ satisfies the Haar--Grassmann sketching condition at $x^\star$.
  Conditional on $x_k$, the operator $\widetilde{\mathcal P}_k$ is independent of the past and has the same distribution as 
  $
  \mathcal P_{x^\star,B^\star,\Xi^\star}
  $ at $x^\star$.
  \end{assumption}

  \begin{theorem}
    \label{thm:explicit-rate-nystrom}
    (\textbf{Local convergence})
    Suppose \cref{ass:smooth}, \cref{ass:strong-geo-convexity}, and \cref{ass:local-transported-resampling} hold. Let $\rho_{\mathrm{HS}},\rho_{\mathrm{op}}$ and $C_0$ be as in \cref{prop:haar-grassmann-pinv-moments} (evaluated at $x^\star$). Set the constants
    \[
    C_p=1+2\rho_{\mathrm{HS}}+\frac{4C_0\rho_{\mathrm{op}}}{p}\,\mathrm{sr}_p(\mathcal{H}^\star),
    \quad
    \kappa=\frac{\lambda_1(\mathcal{H}^\star)}{\lambda_d(\mathcal{H}^\star)},
    \quad
    \overline{\zeta}=\frac{1}{2\bigl(1+2\kappa C_p\,\lambda_p(\mathcal{H}^\star)/\lambda_d(\mathcal{H}^\star)\bigr)}.
    \]
    For any $\delta\in(0,1)$ there exists a neighborhood $V$ of $x^\star$ such that if $x_0\in V$, then
    \[
    \mathbb E\bigl[f(x_k)-f(x^\star)\bigr]\le\bigl(1-(1-\delta)\overline{\zeta}\bigr)^k\bigl(f(x_0)-f(x^\star)\bigr), \text{ for all }k\ge 0.
    \]
  \end{theorem}

  \fi

\section{Numerical experiments} \label{sec:experiments}
In this section, we report numerical experiments to demonstrate the approximation properties and the effectiveness of the proposed optimization method. All experiments are performed on a MacBook Pro equipped with an Apple M2 chip with 8 CPU cores (4 performance cores and 4 efficiency cores) and 16\,GB of unified memory, running Python (Release 3.9.12) under macOS 26.3. The code that produced the results is available at https://github.com/nht2018/RiemannianNystrom.

\subsection{Principal geodesic analysis}
\label{app:pga}
Principal geodesic analysis (PGA) generalizes principal component analysis (PCA) to Riemannian manifolds by linearizing the data at a reference point and then performing PCA in the corresponding tangent space~\cite{fletcher2004principal,Pennec2006}. Specifically, given a dataset (training set) $\{y_i\}_{i=1}^N\subset\mathcal M$, one first computes a Fr\'{e}chet mean
$
  \mu\in\arg\min_{x\in\mathcal M}\ \frac{1}{N}\sum_{i=1}^N \operatorname{dist}(x,y_i)^2.
$
The samples are mapped to the tangent space $\mathrm{T}_\mu\mathcal M$ using the logarithm map
$
  v_i := \mathrm{Log}_\mu(y_i) \in \mathrm{T}_\mu\mathcal M, i=1,\dots,N.
$
PGA then extracts the leading eigenpairs of the empirical covariance operator 
\begin{equation}
  \mathcal{C}_\mu[u] := \frac{1}{N} \sum_{i=1}^N \langle v_i, u \rangle_\mu\, v_i, \quad u\in \mathrm{T}_\mu\mathcal M,
\end{equation}
as
$
  \mathcal C_\mu[\widehat u_k]=\widehat\lambda_k\,\widehat u_k, k=1,\dots,K
$. 
The resulting PGA scores
\begin{equation}
  \label{eq:pga-scores}
  \widehat a_{ik} := \langle \widehat u_k, v_i\rangle_\mu\in\mathbb R,\qquad \widehat a_i:=(\widehat a_{i1},\dots,\widehat a_{iK})^\top\in\mathbb R^K,
\end{equation}
provide an intrinsic, low-dimensional feature representation of manifold-valued data.

In our experiment, let $\mathcal{M}=\mathbb{S}_{++}^n$ be a SPD manifold equipped with the log--Euclidean metric and let $\{X_i\}_{i=1}^N\subset\mathcal{M}$ be samples with class labels $y_i\in\{1,\dots,G\}$. The Fr\'echet mean is $\mu=\exp\!\big(\tfrac1N\sum_{i=1}^N \log X_i\big)$. Each sample is mapped to the tangent space at $\mu$ via the logarithm map, yielding $v_i:=\log_\mu(X_i)=\log(X_i)-\log(\mu)\in \mathrm{T}_\mu\mathcal{M}$. 
The empirical covariance operator at $\mu$ is the self-adjoint, positive semidefinite operator $\mathcal{C}_\mu:\mathrm{T}_\mu\mathcal{M}\to \mathrm{T}_\mu\mathcal{M}$ defined by $\mathcal{C}_\mu[u]:=\tfrac1N\sum_{i=1}^N\langle v_i,u\rangle_\mu\,v_i$. For implementation, it is convenient to introduce the operator $\mathcal{V}:\mathbb{R}^N\to \mathrm{T}_\mu\mathcal{M}$, $\mathcal{V}e_i=v_i$, whose adjoint satisfies $(\mathcal{V}^\ast u)_i=\langle v_i,u\rangle_\mu$; then $\mathcal{C}_\mu=\tfrac1N\,\mathcal{V}\mathcal{V}^\ast$. 
In a fixed coordinate isometry $\mathrm{T}_\mu\mathcal{M}\simeq\mathbb{R}^d$ (e.g., $\mathrm{vech}$ under an orthonormal frame), this reduces to the factor form $C=\tfrac1N VV^\top$, enabling applications of $\mathcal{C}_\mu$ through operator actions $\mathcal{C}_\mu[u]=\tfrac1N\,\mathcal{V}(\mathcal{V}^\ast u)$ without explicitly forming a dense $d\times d$ matrix.

To form a Riemannian Nystr\"om approximation of $\mathcal{C}_\mu$, choose a subspace $B=\mathrm{span}\{b_1,\dots,b_\ell\}\subset \mathrm{T}_\mu\mathcal{M}$, and define the range operator $\mathcal{Y}:\mathbb{R}^\ell\to \mathrm{T}_\mu\mathcal{M}$ by $\mathcal{Y}e_j:=\mathcal{C}_\mu[b_j]$. The associated core matrix is $S\in\mathbb{R}^{\ell\times \ell}$ with entries $S_{ij}:=\langle b_i,\mathcal{C}_\mu[b_j]\rangle_\mu$, equivalently $S=B^\ast\mathcal{C}_\mu B$ in the coefficient representation. Using the factorization $\mathcal{C}_\mu=\tfrac1N\,\mathcal{V}\mathcal{V}^\ast$, each range vector admits the computable form $\mathcal{Y}e_j=\mathcal{C}_\mu[b_j]=\tfrac1N\,\mathcal{V}(\mathcal{V}^\ast b_j)$, so $\mathcal{Y}$ can be assembled from operator applications without forming $\mathcal{C}_\mu$. The Riemannian Nystr\"om approximation of the covariance operator is then defined by
$
\mathcal{C}_{\mu,B}:=\mathcal{Y}\,S^\dagger\,\mathcal{Y}^\ast,
$
and its action is evaluated as $\mathcal{C}_{\mu,B}[u]=\mathcal{Y}\,S^\dagger(\mathcal{Y}^\ast u)$ for $u\in \mathrm{T}_\mu\mathcal{M}$, storing only the range vectors $\{\mathcal{Y}e_j\}_{j=1}^\ell$ and the small core pseudoinverse $S^\dagger$. For either covariance operator $\mathcal{C}$ (i.e., $\mathcal{C}=\mathcal{C}_\mu$ or $\mathcal{C}=\mathcal{C}_{\mu,B}$), the top-$K$ eigenvectors are computed and the PGA scores are defined according to~\eqref{eq:pga-scores}.

The resulting approximate scores serve as features for downstream analyses, including Hotelling's $T^2$ testing and downstream machine learning models.
The Hotelling $T^2$ statistic is defined as follows.
Consider a binary partition. Let $\mathcal{G}_0=\{i: y_i=0\}$ and $\mathcal{G}_1=\{i: y_i=1\}$, and compute group means in score space as $\bar a_g = \frac{1}{n_g}\sum_{i\in\mathcal{G}_g} a_i \in \mathbb{R}^K$. Using the sample covariance within each group $S_g = \frac{1}{n_g-1}\sum_{i\in\mathcal{G}_g} (a_i-\bar a_g)(a_i-\bar a_g)^\top \in \mathbb{R}^{K\times K}$,
the pooled covariance is given by
$ 
S_p = \frac{(n_0-1)S_0 + (n_1-1)S_1}{n_0+n_1-2}.
$
The Hotelling $T^2$ statistic~\cite{Hotelling1931}
\[
T^2 = \frac{n_0 n_1}{n_0+n_1}\,(\bar a_0-\bar a_1)^\top S_p^{-1}(\bar a_0-\bar a_1) .
\]
gives a single scalar quantifying separation between the two groups.
Larger values of $T^2$ indicate that the two group mean score vectors are farther apart relative to the within-group variability.

This PGA pipeline is evaluated on the HDM05 dataset\footnote{We use the preprocessed SPD covariance-matrix version in the supplementary material of~\cite{brooks2019riemannian}, available at \url{https://www.dropbox.com/s/dfnlx2bnyh3kjwy/data.zip?dl=0}.}, which contains 2,086 SPD matrices of size $93\times 93$ (tangent-space dimension $d=4{,}371$) across 117 classes. The dataset is split into training and testing sets in an 80-20 ratio. To mitigate numerical issues arising from ill-conditioned matrices, eigenvalues are floored at $10^{-6}$. We fix the number of principal components to $k=20$ and compare Nystr\"om sketch sizes $\ell\in\{20,40,80\}$ against exact PGA. Table~\ref{tab:pga-hdm05} reports multiclass classification accuracy (multinomial logistic regression, SVM, and MLP) on the resulting score features, the associated memory cost, and the median pairwise Hotelling $T^2$ statistic on HDM05.

\begin{table}[H]
\caption{HDM05 PGA results at $k=20$. Multiclass classification reports accuracy from multinomial logistic regression, SVM, and MLP on PGA scores. Memory cost reports measured RSS increases for the covariance operator and Nystr\"om sketching approximations at $\ell\in\{20,40,80\}$, with percentages relative to the exact resident set size (RSS) increase. Median pairwise Hotelling $T^2$ is reported across all class pairs.}
\label{tab:pga-hdm05}
\centering
\small
\setlength{\tabcolsep}{4pt}
\begin{tabular}{l r@{.}l r@{.}l r@{.}l r@{.}l r@{.}l}
\hline
Method & \multicolumn{2}{c}{Acc (logreg)} & \multicolumn{2}{c}{Acc (SVM)} & \multicolumn{2}{c}{Acc (MLP)} & \multicolumn{2}{c}{Memory (MiB)} & \multicolumn{2}{c}{$T^2$ ($\times 10^3$)} \\
\hline
Exact & 0&683 & 0&770 & 0&688 & 145&77 (100\%) & 5&804 \\
Nystr\"om ($\ell=20$) & 0&679 & 0&741 & 0&640 & 6&27 (4.30\%) & 5&062 \\
Nystr\"om ($\ell=40$) & 0&686 & 0&772 & 0&652 & 10&50 (7.20\%) & 5&354 \\
Nystr\"om ($\ell=80$) & 0&691 & 0&763 & 0&676 & 14&05 (9.64\%) & 5&615 \\
\hline
\end{tabular}
\end{table}

The Nystr\"om sketching approximations preserve close downstream performance on this dataset. Accuracies of the machine learning models and median pairwise Hotelling's $T^2$ values remain close to the exact PGA baseline across all tested sketch sizes $\ell$. At the same time, they reduce memory usage, requiring only 4.30\%--9.64\% of the exact operator's RSS increase while maintaining comparable statistical quality.
\subsection{Riemannian Nystr\"om approximation in optimization}
In this section, we investigate the empirical performance of the proposed Riemannian Nystr\"om Newton method on two representative manifold optimization problems.
The Riemannian cubic Newton method~\cite{zhang2018cubic} serves as a baseline and is labeled ``cubic Newton'' in the figures.
In all experiments, the regularization parameter is updated according to a standard adaptive cubic regularization (ARC) rule~\cite{cartis2011adaptive}.

\paragraph{Geodesically convex functions on SPD manifolds}
We first consider the following geodesically convex optimization problem, which arises in geometrically regularized covariance estimation. Given a reference covariance matrix \(A \in\mathbb{S}_{++}^n\) and a data scatter matrix $B \in\mathbb{S}_{++}^n$, the goal is to estimate an SPD matrix \(X\) that balances the closeness to \(A\) with a fidelity term~\cite{Mahalanobis1936}. The optimization problem is formulated as
\[
\min_{X \in \mathbb{S}_{++}^{n}} f(X)=
w \left\|\log\!\left(X^{-1/2} A X^{-1/2}\right)\right\|_{\mathrm{F}}^{2}
+\lambda \tr\!\left(B X^{-1}\right)
+\rho \left(\tr\!\left(B X^{-1}\right)\right)^{2}.
\]
In this experiment, we set $n=40$, $w=\lambda=\rho=1$.
The matrices \(A\) and \(B\) are generated synthetically as random SPD matrices of the form
\[
A = Q_A \diag(\alpha_1,\dots,\alpha_n) Q_A^\top,
\qquad
B = Q_B \diag(\beta_1,\dots,\beta_n) Q_B^\top,
\]
where \(Q_A\) and \(Q_B\) are orthogonal factors obtained from QR decompositions of Gaussian random matrices, and the eigenvalues \(\alpha_i\) and \(\beta_i\) are sampled independently and uniformly from \([0.2,2.0]\). 
Figure~\ref{fig:spd} reports the performance of the RRNCN method with different sketch sizes on the SPD manifold.
In this experiment, the intermediate sketch size $\ell=80$ achieves the best runtime for reaching the optimal value, illustrating the trade-off between Riemannian Nystr\"om approximation quality and per-iteration computational cost.
As the sketch size increases, the algorithm typically reaches a critical point in fewer iterations.
On the other hand, a larger sketch size also incurs a higher computational cost per iteration.

\begin{figure}[htbp]
  \centering
  \begin{subfigure}[t]{0.48\textwidth}
  \centering
  \includegraphics[width=\textwidth]{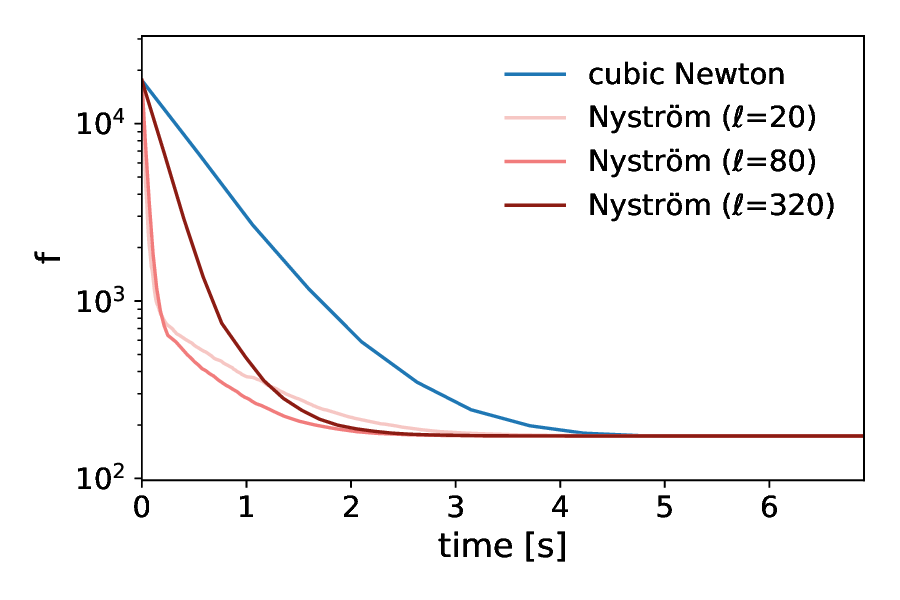}
  \end{subfigure}
  \begin{subfigure}[t]{0.48\textwidth}
  \centering
  \includegraphics[width=\textwidth]{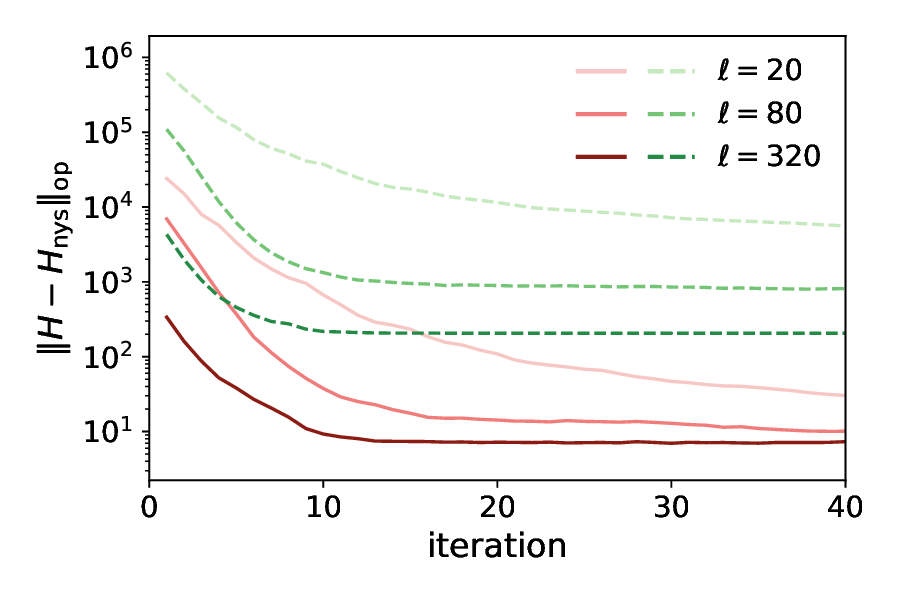}
  \end{subfigure}
  \caption{Performance of the RRNCN method on the geodesically convex optimization problem over the SPD manifold with $n=40$. Left: objective value gap versus wall-clock time. Right: the approximation error together with the theoretical bound from \cref{cor:manifold-22-gaussian} versus the iteration number, where the red solid curve denotes the approximation error, and the green dashed curve denotes the theoretical bound.}
  \label{fig:spd}
\end{figure}

\paragraph{Transported sketching on the Grassmann manifold}
We evaluate the transported sketching strategy proposed in Section~\ref{subsec:transported-sketching} on the Grassmann manifold~\cite{bendokat2024grassmann} $\mathrm{Gr}(n,p)$ with $n = 20000$, $p = 20$, and $\ell=20$. We consider the problem of computing top-$p$ eigenvalues of a given matrix, specifically, the optimization problem is
\[
\max_{X \in \mathrm{Gr}(n,p)} \operatorname{tr}\!\left(X^\top A X\right).
\]
The $A\in\mathbb{R}^{n\times n}$ is taken to be a diagonal matrix
$
\operatorname{diag}(a_1,\dots,a_n),
$
whose diagonal entries are uniform in $[1, 20]$. 
In this experiment, the sketching is refreshed every $T=2$ or $T=3$ iterations, and otherwise we use transported sketching instead of constructing a new sketching from scratch. Let $f(X) = - \operatorname{tr}(X^\top A X)$.
The results are shown in Figure~\ref{fig:grassmann}. The left panel shows that the numbers of iterations required by the Nystr\"om and transported-sketching variants with $T=2,3$ are nearly identical in practice, and the right panel demonstrates the potential runtime speed-up achieved by transported sketching.

\begin{figure}[htbp]
  \centering
  \begin{subfigure}[b]{0.48\textwidth}
  \centering
  \includegraphics[width=\textwidth]{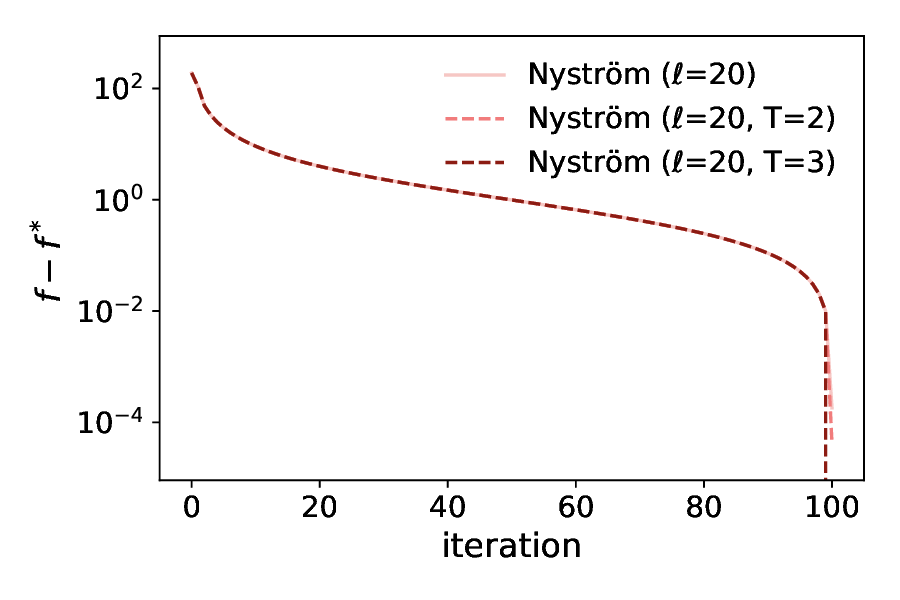}
  \end{subfigure}
  \begin{subfigure}[b]{0.48\textwidth}
  \centering
  \includegraphics[width=\textwidth]{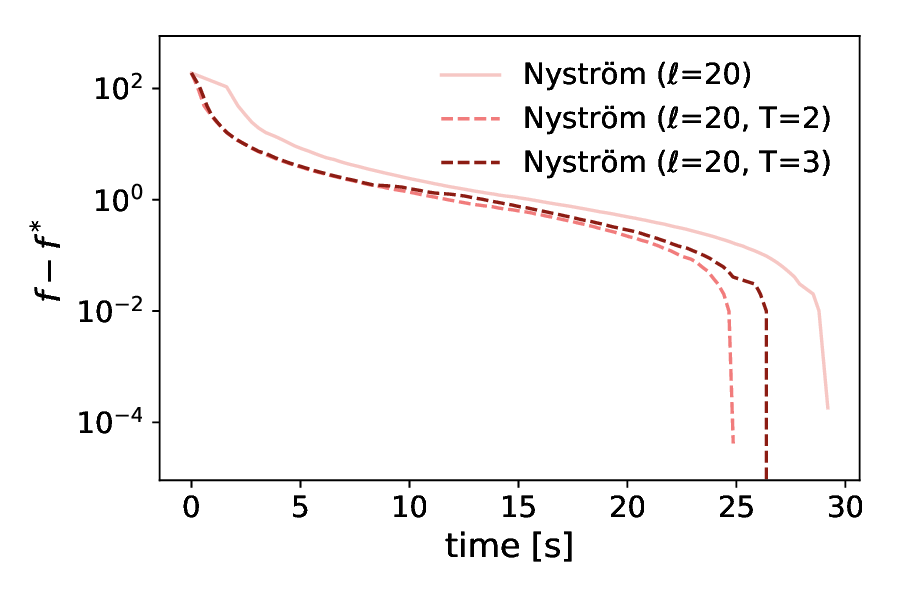}
  \end{subfigure}
  \caption{Performance of the RRNCN method on the Grassmann manifold with $n=20000$ and $p=20$. Left: objective value gap versus iteration. Right: objective value gap versus wall-clock time.}
  \label{fig:grassmann}
\end{figure}

\section{Conclusion and future work}
\label{sec:conclusion}

This paper develops an intrinsic Riemannian Nystr\"om approximation for self-adjoint positive semidefinite operators. The construction preserves elementary properties, including positive semidefiniteness, Loewner domination, and range characterizations. Under a Haar--Grassmann sketching condition, we establish spectral approximation bounds measured in the Riemannian operator norm, thereby extending classical Euclidean Nystr\"om theory to tangent-space operators on manifolds.

Building on this theoretical foundation, we study a randomized Nystr\"om Newton method for Riemannian optimization as an application of the Riemannian Nystr\"om approximation. Numerical experiments on principal geodesic analysis and optimization tasks on the SPD and Grassmann manifolds indicate that the Riemannian Nystr\"om approximation effectively captures the dominant spectral structure of the original operator while reducing computational cost and memory usage.

Several directions remain open. It would be valuable to develop adaptive rules for choosing sketch size and regularization parameters, and to select the best basis for constructing the Riemannian Nystr\"om approximation.

\appendix

\section{Technical details for global complexity}
\label{app:global-complexity}
This section contains the technical details for the global complexity.
To derive global complexity bounds, it is instrumental to analyze the behavior of the method along sketched retraction curves. The next lemma establishes a cubic upper bound that reflects the local geometry of the manifold and ensures the existence of a cubic Newton step.

\begin{lemma}\label{lem:retraction-cubic-upper}
 Suppose \cref{ass:smooth} holds. For any $\rho > 0$, there exists $M_\rho>0$ such that for all $v\in B$ with $\|v\|_x\le \rho$, it holds that
  \begin{equation}
  \label{eq:retraction-cubic-upper-bound}
  \begin{aligned}
  f\big(R_x(\mathcal{P}_{x, B, \Xi} [v])\big) &\le f(x)+\langle h_{x, B, \Xi},v\rangle_x+\tfrac{1}{2}\langle \mathcal{J}_{x, B, \Xi} [v],v\rangle_x+\tfrac{M_\rho}{6}\|v\|_x^{3} ,\\
  f\big(R_x(\mathcal{P}_{x, B, \Xi} [v])\big) &\ge f(x)+\langle h_{x, B, \Xi},v\rangle_x+\tfrac{1}{2}\langle \mathcal{J}_{x, B, \Xi} [v],v\rangle_x - \tfrac{M_\rho}{6}\|v\|_x^{3}.
  \end{aligned}
  \end{equation}
   The constant $M_\rho$ is chosen uniformly on $\mathcal{M}$ and is independent of $x$.
  \end{lemma}

  \begin{proof}
  Fix $x$ and consider the pullback $\hat f:=f\circ R_x$ defined on $\mathrm{T}_x\mathcal{M}$. Since $R_x$ is a second-order retraction, $\nabla\hat f(0)=\grad f(x)$ and $\nabla^{2}\hat f(0)=\Hess f(x)$ in the metric induced by $g_x$. By Proposition 3.1 in~\cite{zhang2018cubic}, for radius $ \rho > 0$, there exists a constant $\widehat M_{\rho}>0$ depending only on $L_2, L_3$ and $\rho$ such that
  \begin{subequations}
  \begin{align}
  \hat f(u)&\le \hat f(0)+\langle \grad f(x),u\rangle_x+\tfrac{1}{2}\langle \mathcal{H}_x [u],u\rangle_x+\tfrac{\widehat M_{\rho}}{6}\|u\|_x^{3}, \label{eq:cubic-upper-bound-1} \\
  \hat f(u)&\ge \hat f(0) +\langle \grad f(x),u\rangle_x+\tfrac{1}{2}\langle \mathcal{H}_x [u],u\rangle_x - \tfrac{\widehat M_{\rho}}{6}\|u\|_x^{3}, \label{eq:cubic-upper-bound-2}
  \end{align}
  \end{subequations}
  for all $u\in\mathrm{T}_x\mathcal{M}$ with $\|u\|_x\le \rho$.
  Taking $u = \mathcal{P}_{x, B, \Xi} [v]$ and using the adjoint identities
  \[
  \begin{aligned}
  \langle \grad f(x),\mathcal{P}_{x, B, \Xi} [v]\rangle_x &= \langle h_{x, B, \Xi},v\rangle_x, \\
  \langle \mathcal{H}_x [\mathcal{P}_{x, B, \Xi} [v]],\mathcal{P}_{x, B, \Xi} [v]\rangle_x &=\langle \mathcal{J}_{x, B, \Xi} [v],v\rangle_x,
  \end{aligned}
  \]
  together with $\| \mathcal{P}_{x, B, \Xi} [v] \|_x \leq \|\mathcal{R}\|_{\mathrm{op}}\|v\|_x$ (where $\mathcal{R}$ is the radial factor in the polar decomposition of $\mathcal{F}$), we obtain the desired inequality with $M_\rho:=\widehat M_{\rho}\,\|\mathcal{R}\|_{\mathrm{op}}^{3}$.
\end{proof}

\begin{lemma}
  Let $\mathcal{M}_0:=\left\{x \in \mathcal{M}: f(x) \leq f\left(x_0\right)\right\}$ and \[
  \bar \rho := 3 \max \{ \max_{x \in \mathcal{M}_0} \| \grad f(x) \|_x, \max_{x \in \mathcal{M}_0} \| \Hess f(x) \|_{\mathrm{op}} \}.
  \]
  By setting $\rho=\bar\rho$ and choosing $\sigma\ge \max\{M_\rho,1\}$, it holds that any minimizer $v$ of the subproblem~\eqref{eq:cubic-model} satisfies $\|v\|_x\le \rho$.
\end{lemma}

\begin{proof}
  According to Lemma 2.2 in~\cite{cartis2011adaptive}, the solution to the subproblem~\eqref{eq:cubic-model} satisfies 
  \[
  \|v\|_x \le \tfrac{3}{\sigma} \max \{ \sqrt{\sigma}\, \|h_{x, B, \Xi}\|_x, \| \mathcal{J}_{x, B, \Xi} \|_{\mathrm{op}} \} \le \tfrac{3}{\sigma} \max \{ \sqrt{\sigma}\, \|\grad f(x)\|_x, \| \mathcal{H}_x \|_{\mathrm{op}} \} \le \rho.
  \]
  \end{proof}
  \begin{remark}
    \label{rem:cubic-upper-bound}
  The lemma shows that for $\rho$ sufficiently large, the solution to the subproblem~\eqref{eq:cubic-model} is in the interior of the ball of radius $\rho$. Without loss of generality, we can let $M = \max_{\rho \in [0, \bar \rho]} M_\rho$ be a uniform constant independent of $x, \rho$. Then~\eqref{eq:retraction-cubic-upper-bound} becomes
  \[
  \begin{aligned}
  f\big(R_x(\mathcal{P}_{x, B, \Xi} [v])\big) &\le f(x)+\langle h_{x, B, \Xi},v\rangle_x+\tfrac{1}{2}\langle \mathcal{J}_{x, B, \Xi} [v],v\rangle_x+\tfrac{M}{6}\|v\|_x^{3} ,\\
  f\big(R_x(\mathcal{P}_{x, B, \Xi} [v])\big) &\ge f(x)+\langle h_{x, B, \Xi},v\rangle_x+\tfrac{1}{2}\langle \mathcal{J}_{x, B, \Xi} [v],v\rangle_x - \tfrac{M}{6}\|v\|_x^{3}.
  \end{aligned}
  \]
  \end{remark}

  The following lemma establishes fundamental isotropy properties of random projections acting on gradients and Hessians.

  \begin{lemma}\label{lem:isotropy-identities}
  Under \cref{ass:smooth}, let $\Pi_{\Xi_x}$ denote the $g_x$-orthogonal projector onto the random sketch subspace $\Xi$ at $x$. Then for all $u \in \mathrm{T}_x\mathcal{M}$, it holds that
  \[
  \mathbb{E}\left[\langle \grad f(x),\Pi_{\Xi_x} [u]\rangle_x\right]=\tfrac{\ell}{d}\langle \grad f(x),u\rangle_x,
  \]
  \[
  \mathbb{E}\left[\langle \mathcal{H}_x[\Pi_{\Xi_x}[u]],\Pi_{\Xi_x}[u]\rangle_x\right]\le\tfrac{\ell}{d}\,L_2\,\|u\|_x^{2},
  \]
  \[
  \mathbb{E}\left[\|\Pi_{\Xi_x} [u]\|_x^{3}\right]\le \tfrac{\ell}{d}\|u\|_x^{3}.
  \]
  \end{lemma}
  
  \begin{proof}
  The first identity follows from linearity of expectation and $\mathbb{E}[\Pi_{\Xi_x}]=\tfrac{\ell}{d}\,\mathrm{Id}_x$. For the second, apply $\mathcal{H}_x \preceq L_2 \mathrm{Id}_x$ and the first identity to obtain
  \[
  \mathbb{E}\left[\langle \mathcal{H}_x[\Pi_{\Xi_x}[u]],\Pi_{\Xi_x}[u]\rangle_x\right]\le \mathbb{E}\left[L_2\|\Pi_{\Xi_x}[u]\|_x^{2}\right]=L_2\,\tfrac{\ell}{d}\|u\|_x^{2}.
  \]
  The third inequality uses that $\Pi_{\Xi_x}$ is an orthogonal projector, hence idempotent and contractive: $\Pi_{\Xi_x}[u]=\Pi_{\Xi_x}[\Pi_{\Xi_x}[u]]$ and $\|\Pi_{\Xi_x}[u]\|_x\le \|u\|_x$. Therefore
  \[
  \|\Pi_{\Xi_x}[u]\|_x^{3}=\|\Pi_{\Xi_x}[u]\|_x\,\|\Pi_{\Xi_x}[u]\|_x^{2}\le \|u\|_x\,\|\Pi_{\Xi_x}[u]\|_x^{2}.
  \]
  Taking expectations and using $\mathbb{E}[\|\Pi_{\Xi_x}[u]\|_x^{2}]=\langle u,\mathbb{E}[\Pi_{\Xi_x}]u\rangle_x=\tfrac{\ell}{d}\|u\|_x^{2}$ yields\\
  $
  \mathbb{E}[\|\Pi_{\Xi_x}[u]\|_x^{3}]\le \tfrac{\ell}{d}\|u\|_x^{3}.
  $
  \end{proof}

  With the cubic upper bound and isotropy properties in hand, it remains to analyze the expected progress of a single iteration. The following lemma relates the expected function value at the next iterate to the current iterate and an arbitrary comparison point.
  
  \begin{lemma}\label{lem:one-step-progress}
  Suppose \cref{ass:smooth} holds. At iterate $x\in\mathcal{M}$, set $\sigma \geq M$ and let $\hat{v} \in\arg\min_{v\in B_x}\big\{\langle h_{x, B, \Xi},v\rangle_x+\tfrac{1}{2}\langle \mathcal{J}_{x, B, \Xi} [v],v\rangle_x+\tfrac{\sigma}{6}\|v\|_x^{3}\big\}$ and set $x^{+}:=R_x(\mathcal{P}_{x, B, \Xi} [\hat{v}])$. Then for every $y\in\mathcal{M}$,
  \[
  \mathbb{E}\left[f(x^{+})\mid x\right]\le\Big(1-\tfrac{\ell}{d}\Big)f(x)+\tfrac{\ell}{d}f(y)+\tfrac{L_2\ell}{2d}\operatorname{dist}(x,y)^{2}+\tfrac{\sigma\ell}{6d}\,\operatorname{dist}(x,y)^{3}.
  \]
  \end{lemma}

  \begin{proof}
  
  Let $w := \Exp_x^{-1}(y) \in \mathrm{T}_x \mathcal{M}$.
  By the geodesic convexity of $f$, $\langle \grad f(x),w\rangle_x\le f(y)-f(x)$. Since $\tfrac{M}{6}\|w\|_x^{3}\ge 0$, it follows that
  \begin{equation}
  \label{eq: proof of lemma one-step-progress-1}
  \langle \grad f(x),w\rangle_x \leq f(y) - f(x) + \tfrac{M}{6}\|w\|_x^{3}.
  \end{equation}
  Let $u:=\Pi_{\Xi_x}[w]\in \Xi$. Since $\mathcal{P}_{x,B,\Xi}|_{B}=\mathcal{F}:B\to\Xi$ has rank $\ell=\dim(\Xi)$, it is surjective onto $\Xi$, so we may choose $v\in B$ such that $\mathcal{P}_{x,B,\Xi}[v]=u$.
  The optimality of $\hat v$ implies that
  \[
  \begin{aligned}
  f(x^{+})
  &\le f\big(R_x(\Pi_{\Xi_x}[w])\big)\\
  &\le f(x)
    +\langle \grad f(x),\Pi_{\Xi_x}[w]\rangle_x
    +\tfrac{1}{2}\langle \mathcal{H}_x[\Pi_{\Xi_x}[w]],\Pi_{\Xi_x}[w]\rangle_x
    +\tfrac{\sigma}{6}\|\Pi_{\Xi_x}[w]\|_x^{3}.
  \end{aligned}
  \]
  Taking conditional expectation and applying \cref{lem:isotropy-identities} gives
  \[
  \begin{aligned}
  \mathbb{E} [ f(x^{+}) | x]
    &\leq f(x)
      +\tfrac{\ell}{d}\langle \grad f(x), w\rangle_x
      +\tfrac{L_2 \ell}{2d}\|w\|_x^{2}
      +\tfrac{\sigma \ell}{6d}\|w\|_x^{3} \\
    &\overset{\eqref{eq: proof of lemma one-step-progress-1}}{\leq} f(x)+\tfrac{\ell}{d}\big(f(y)-f(x) + \tfrac{M}{6}\|w\|_x^{3}\big)  +\tfrac{L_2 \ell}{2d}\|w\|_x^{2}
    +\tfrac{\sigma \ell}{6d}\|w\|_x^{3} \\
    &\leq f(x)+\tfrac{\ell}{d}\big(f(y)-f(x)\big)+\tfrac{L_2 \ell}{2d}\|w\|_x^{2}+\tfrac{\sigma \ell}{6d}\|w\|_x^{3}.
  \end{aligned}
  \]
  Finally, substituting $\|w\|_x=\operatorname{dist}(x,y)$ completes the proof.
  \end{proof}
Building on the one-step progress bound, \cref{thm:global-complexity} establishes global convergence rates, with explicit guarantees from any starting point.

\paragraph{Proof of \Cref{thm:global-complexity}}
\begin{proof}
  For $k\ge 0$, set $\delta_k:=\mathbb{E}\left[f(x_k)\right]-f^{\star}$ and define the weights
\[
a_k:=k^{2},\quad A_0:=\tfrac{4}{3}\Big(\tfrac{d}{\ell}\Big)^{3},\quad A_k:=A_0+\sum_{i=1}^{k}a_i\ \text{ for }k\ge 1,\quad \alpha_k:=\tfrac{d}{\ell}\tfrac{a_{k+1}}{A_{k+1}}.
\]
Let $y_k:=\operatorname{Exp}_{x_k}\big(\alpha_k\,\operatorname{Exp}_{x_k}^{-1}(x^\star)\big)$. 
    The geodesic convexity of $f$ implies that
    \[
    f(y_k)\le (1-\alpha_k)f(x_k)+\alpha_k f^*+\tfrac{L_2}{2}\alpha_k^{2}\operatorname{dist}(x_k,x^\star)^{2},
    \]
    and $\operatorname{dist}(x_k,y_k)=\alpha_k\,\operatorname{dist}(x_k,x^\star)\le \alpha_k r$. Applying \cref{lem:one-step-progress} with $y=y_k$, substituting the bound on $f(y_k)$, and using $\sigma_k=M$, we obtain
    \[
      \mathbb{E}\left[f(x_{k+1}) \mid x_k\right]\le \Big(1-\tfrac{\ell}{d}\alpha_k\Big)f(x_k)+\tfrac{\ell}{d}\alpha_k f^*+\tfrac{L_2 \ell}{d} \alpha_k^{2}r^{2}+\tfrac{M \ell}{6d}\alpha_k^{3}r^{3},
    \]
    where the coefficient $\tfrac{L_2\ell}{d}$ arises from combining the $\tfrac{L_2}{2}\alpha_k^{2}r^{2}$ term in the bound on $f(y_k)$ with the $\tfrac{L_2\ell}{2d}\alpha_k^{2}r^{2}$ term from \cref{lem:one-step-progress}.
    Multiplying both sides by $A_{k+1}$ and using $1-\tfrac{\ell}{d}\alpha_k=\tfrac{A_k}{A_{k+1}}$ gives
    $
    A_{k+1}\delta_{k+1}\le A_k\delta_k+\tfrac{d}{\ell}L_2 r^{2}\frac{a_{k+1}^{2}}{A_{k+1}}+\Big(\tfrac{d}{\ell}\Big)^{2}\tfrac{M r^{3}}{6}\frac{a_{k+1}^{3}}{A_{k+1}^{2}}
    $.
    Summing from $i=0$ to $k-1$ and using the elementary estimates $\sum_{i=1}^{k}\tfrac{a_i^{2}}{A_i}\le 3k^{2}$ and $\sum_{i=1}^{k}\tfrac{a_i^{3}}{A_i^{2}}\le 9k$ yields
    $
    A_k\delta_k\le A_0\delta_0+\tfrac{d}{\ell}L_2 r^{2}\cdot 3k^{2}+\Big(\tfrac{d}{\ell}\Big)^{2}\tfrac{M r^{3}}{6}\cdot 9k.
    $
    Since $A_k\ge A_0+\tfrac{k^{3}}{3}$ and $\tfrac{A_0}{A_0+k^{3}/3}\le \big(1+\tfrac{\ell}{4d}k\big)^{-3}$, division by $A_k$ yields the result.
    \end{proof}

\section{Technical details for local convergence rates}
\label{app:local-convergence-rates}
This section contains the technical details for the local convergence rates.
Throughout this section, let $x^\star$ be the local minimizer of $f$, and let $U$ be a geodesically convex neighborhood of $x^\star$ on which the local analysis is carried out. For each $x\in U$, let
\[
\mathcal T_{x^\star\to x}:\mathrm{T}_{x^\star}\mathcal{M}\to \mathrm{T}_x\mathcal{M}
\]
be the isometric vector transport from \cref{ass:local-transported-resampling}, and write $\mathcal T_{x\to x^\star}:=(\mathcal T_{x^\star\to x})^{-1}$.

For every $x\in U$, define the pullback gradient and pullback Hessian on the fixed reference tangent space $\mathrm{T}_{x^\star}\mathcal{M}$ by
\[
\widetilde g(x)
:=
\mathcal T_{x\to x^\star}\,\grad f(x),
\qquad
\widetilde{\mathcal H}(x)
:=
\mathcal T_{x\to x^\star}\,\Hess f(x)\,\mathcal T_{x^\star\to x}.
\]
At iteration $k$, let
\[
\widetilde{\mathcal P}_k
:=
\mathcal T_{x_k\to x^\star}\,\mathcal P_k\,\mathcal T_{x^\star\to x_k},
\qquad
\widetilde B_k:=\mathcal T_{x_k\to x^\star}(B_k)\subset \mathrm{T}_{x^\star}\mathcal{M},
\]
\[
\widetilde g_k:=\widetilde g(x_k),
\quad
\widetilde{\mathcal H}_k:=\widetilde{\mathcal H}(x_k),
\quad
\widetilde h_k:=\widetilde{\mathcal P}_k^*[\widetilde g_k]\in \widetilde B_k,
\quad
\widetilde{\mathcal J}_k:=\widetilde{\mathcal P}_k^*\,\widetilde{\mathcal H}_k\,\widetilde{\mathcal P}_k:\widetilde B_k\to \widetilde B_k.
\]
Since $\mathcal T_{x^\star\to x}$ is an isometry, adjoints, norms, and the Loewner order are preserved under pullback.

\begin{definition}
  \textbf{(Self-concordant function)}
  A geodesically convex function $f$ is $M_f$-self-concordant on an open and geodesically convex neighborhood $U$ if, for every affinely parameterized geodesic $\gamma\subset U$, the scalar function $\varphi(t)=f(\gamma(t))$ satisfies
  \[
  \left|\varphi^{\prime\prime\prime}(t)\right|
  \le 2M_f\left(\varphi^{\prime\prime}(t)\right)^{3/2},
  \qquad
  \forall t \text{ such that } \gamma(t)\in U.
  \]
  Equivalently,
  \[
  \big|\nabla^{3} f(x)[v,v,v]\big|
  \le 2 M_f \big\langle \Hess f(x)[v], v\big\rangle_x^{3/2},
  \qquad
  \forall x\in U,\ \forall v\in \mathrm{T}_x\mathcal{M}
  \]
  whenever $\Exp_x(tv)\in U$ for all $t\in[0,1]$. Here $\nabla^3 f$ is the third covariant derivative of $f$ with respect to the Levi-Civita connection.
\end{definition}

\begin{corollary}
\cref{ass:smooth} and~\cref{ass:strong-geo-convexity} imply that $f$ is $M_f$-self-concordant with
$
M_f=\frac{1}{2}L_3\,\mu^{-3/2}
$.
\end{corollary}

The following proposition records the self-concordant estimates used below. It is the geodesic counterpart of the classical Euclidean self-concordant property~\cite[{Theorem 5.1.7, Theorem 5.1.12}]{nesterov2018lectures}.
\begin{proposition}
\label{prop:self-concordant}
Let $f$ be geodesically convex and $M_f$-self-concordant on an open geodesically convex neighborhood $U$.
\begin{enumerate}
\item[1.] For any $x,y\in U$, let
\[
v:=\Exp_x^{-1}(y)\in \mathrm{T}_x\mathcal{M},
\qquad
\delta_x(y):=\big\langle \Hess f(x)[v],v\big\rangle_x^{\frac{1}{2}}.
\]
If
$
M_f\,\delta_x(y)<1,
$
then
\[
\bigl(1-M_f\delta_x(y)\bigr)^2 \Hess f(x)
\preceq
\mathcal T_{y\to x}\,\Hess f(y)\,\mathcal T_{x\to y}
\preceq
\bigl(1-M_f\delta_x(y)\bigr)^{-2}\Hess f(x).
\]

\item[2.] For any $y\in U$, let
\[
\Lambda(y):=
\big\langle
\grad f(y),\,
[\Hess f(y)]^{-1}[\grad f(y)]
\big\rangle_y^{\frac{1}{2}}.
\]
If
$
M_f\,\Lambda(y)<1,
$
then
\[
f(y)-f(x^\star)\le \frac{1}{M_f^2}\,\omega_*\bigl(M_f\Lambda(y)\bigr),
\]
where
\[
\omega_*(t):=-t-\ln(1-t),
\quad t\in[0,1).
\]
In addition, for every $\gamma>0$ and every
$
t\in\Bigl[0,\frac{\gamma}{1+\gamma}\Bigr],
$
we have
$
\omega_*(t)\le \frac{1+\gamma}{2}\,t^2.
$
\end{enumerate}
\end{proposition}

To analyze the one-step decrease, it is convenient to separate the ambient gradient from its sketched counterpart.
\begin{lemma}
\label{lem:manifold-descent-lemma}
Let $\sigma_k\ge M$ in Algorithm~\ref{alg:sscn-riem}, where $M$ is defined in \cref{rem:cubic-upper-bound}. Then the exact cubic step satisfies
\[
f(x_k)-f(x_{k+1})
\ge
\frac{1}{2}\,
\Big\langle
h_k,\
\bigl(\mathcal J_k+\sqrt{\sigma_k/2}\,\|h_k\|_{x_k}^{\frac{1}{2}}\,\mathrm{Id}_{B_k}\bigr)^{-1}[h_k]
\Big\rangle_{x_k}.
\]
\end{lemma}

\begin{proof}
Fix $x\in \mathcal M$ and consider the cubic model
\[
\phi(v)
=
\langle h_{x,B,\Xi},v\rangle_x
+\frac{1}{2}\langle \mathcal J_{x,B,\Xi}[v],v\rangle_x
+\frac{\sigma}{6}\|v\|_x^3,
\qquad
v\in B.
\]
Let $\hat v\in \arg\min_{v\in B}\phi(v)$, and set
\[
\hat\eta_x:=\mathcal P_{x,B,\Xi}[\hat v],
\qquad
x^+:=R_x(\hat\eta_x).
\]
The first-order optimality condition for $\phi$ yields
\begin{equation}
-\;h_{x,B,\Xi}
=
\bigl(\mathcal J_{x,B,\Xi}+\tfrac{\sigma}{2}\|\hat v\|_x\,\mathrm{Id}_B\bigr)[\hat v].
\label{eq:opt-local}
\end{equation}
Taking the inner product of~\eqref{eq:opt-local} with $\hat v$ gives
\begin{equation}
\langle h_{x,B,\Xi},\hat v\rangle_x
+
\langle \mathcal J_{x,B,\Xi}[\hat v],\hat v\rangle_x
+
\tfrac{\sigma}{2}\|\hat v\|_x^3
=
0.
\label{eq:opt-id-local}
\end{equation}

Because $R$ is a second-order retraction and the Hessian is Lipschitz in the sketched directions, the cubic upper bound along the retraction curve yields
\begin{equation}
f(x^+)
\le
f(x)
+
\langle \grad f(x),\hat\eta_x\rangle_x
+
\frac{1}{2}\langle \Hess f(x)[\hat\eta_x],\hat\eta_x\rangle_x
+
\frac{M}{6}\|\hat v\|_x^3.
\label{eq:cubic-upper-local}
\end{equation}
It follows from~\eqref{eq:cubic-upper-local},~\eqref{eq:opt-id-local} and the facts
\[
\langle \grad f(x),\hat\eta_x\rangle_x
=
\langle \mathcal P_{x,B,\Xi}^*\grad f(x),\hat v\rangle_x
=
\langle h_{x,B,\Xi},\hat v\rangle_x,
\]
and
$
\langle \Hess f(x)[\hat\eta_x],\hat\eta_x\rangle_x
=
\langle \mathcal J_{x,B,\Xi}[\hat v],\hat v\rangle_x
$
that
\begin{align}
f(x)-f(x^+)
&\ge
-\langle h_{x,B,\Xi},\hat v\rangle_x
-\frac{1}{2}\langle \mathcal J_{x,B,\Xi}[\hat v],\hat v\rangle_x
-\frac{M}{6}\|\hat v\|_x^3 \nonumber\\
&=
\frac{1}{2}\langle \mathcal J_{x,B,\Xi}[\hat v],\hat v\rangle_x
+
\Bigl(\frac{\sigma}{2}-\frac{M}{6}\Bigr)\|\hat v\|_x^3 \nonumber\\
&\ge
\frac{1}{2}
\Big\langle
\bigl(\mathcal J_{x,B,\Xi}+\tfrac{\sigma}{2}\|\hat v\|_x\,\mathrm{Id}_B\bigr)[\hat v],\hat v
\Big\rangle_x.
\label{eq:pred-red-local}
\end{align}
Combining~\eqref{eq:pred-red-local} with~\eqref{eq:opt-local} yields
\begin{equation}
f(x)-f(x^+)
\ge
\frac{1}{2}
\Big\langle
h_{x,B,\Xi},
\bigl(\mathcal J_{x,B,\Xi}+\tfrac{\sigma}{2}\|\hat v\|_x\,\mathrm{Id}_B\bigr)^{-1}[h_{x,B,\Xi}]
\Big\rangle_x.
\label{eq:key-local}
\end{equation}

Next, by~\eqref{eq:opt-local},
it holds that
\[
\|h_{x,B,\Xi}\|_x
=
\big\|
\bigl(\mathcal J_{x,B,\Xi}+\tfrac{\sigma}{2}\|\hat v\|_x\,\mathrm{Id}_B\bigr)[\hat v]
\big\|_x
\ge
\frac{\sigma}{2}\|\hat v\|_x^2,
\]
and hence
\[
\frac{\sigma}{2}\|\hat v\|_x
\le
\sqrt{\frac{\sigma}{2}}\,\|h_{x,B,\Xi}\|_x^{\frac{1}{2}}.
\]
Therefore,
\[
\mathcal J_{x,B,\Xi}+\frac{\sigma}{2}\|\hat v\|_x\,\mathrm{Id}_B
\preceq
\mathcal J_{x,B,\Xi}+\sqrt{\frac{\sigma}{2}}\,\|h_{x,B,\Xi}\|_x^{\frac{1}{2}}\,\mathrm{Id}_B.
\]
Since $A\mapsto A^{-1}$ is operator-monotone decreasing on positive definite operators,~\eqref{eq:key-local} implies
\[
f(x)-f(x^+)
\ge
\frac{1}{2}
\Big\langle
h_{x,B,\Xi},
\bigl(\mathcal J_{x,B,\Xi}+\sqrt{\tfrac{\sigma}{2}}\,\|h_{x,B,\Xi}\|_x^{\frac{1}{2}}\,\mathrm{Id}_B\bigr)^{-1}[h_{x,B,\Xi}]
\Big\rangle_x.
\]
Applying this with $x=x_k$, $\sigma=\sigma_k$, $h_{x,B,\Xi}=h_k$, and $x^+=x_{k+1}$ proves the claim.
\end{proof}

\begin{theorem}
\label{thm:riem-local-exp}
Suppose \cref{ass:smooth}, \cref{ass:strong-geo-convexity}, and \cref{ass:local-transported-resampling} hold. Let $\mathcal{H}^\star:=\Hess f(x^\star)$ and define
\[
\zeta
:=
\lambda_{\min}\!\Big(
(\mathcal{H}^\star)^{\frac{1}{2}}\,
\mathbb E\big[
\mathcal P_{x^\star,B^\star,\Xi^\star}
\bigl(\mathcal P_{x^\star,B^\star,\Xi^\star}^{*}\mathcal{H}^\star\mathcal P_{x^\star,B^\star,\Xi^\star}\bigr)^{-1}
\mathcal P_{x^\star,B^\star,\Xi^\star}^{*}
\big]\,
(\mathcal{H}^\star)^{\frac{1}{2}}
\Big).
\]
Then for any $\delta\in(0,1)$ there exists a neighborhood $V\subset U$ such that if $x_0\in V$, then
\[
\mathbb E\bigl[f(x_k)-f(x^\star)\bigr]
\le
\bigl(1-(1-\delta)\zeta\bigr)^k\,\bigl(f(x_0)-f(x^\star)\bigr).
\]
\end{theorem}

\begin{proof}
By isometric pullback, \cref{lem:manifold-descent-lemma} can be rewritten on the reference tangent space $\mathrm{T}_{x^\star}\mathcal{M}$ as
\begin{equation}
f(x_k)-f(x_{k+1})
\ge
\frac{1}{2}
\Big\langle
\widetilde h_k,\
\bigl(\widetilde{\mathcal J}_k+\tilde{\mu}_k\mathrm{Id}_{\widetilde B_k}\bigr)^{-1}[\widetilde h_k]
\Big\rangle_{x^\star},
\label{eq:descent-pulledback}
\end{equation}
where
$
\tilde{\mu}_k
:=
\sqrt{\sigma_k/2}\,\|\widetilde h_k\|_{x^\star}^{\frac{1}{2}}.
$
For any $\nu>0$, define
\[
S_{\nu}
:=
(\mathcal{H}^\star)^{\frac{1}{2}}\,
\mathbb E\Big[
\mathcal P_{x^\star,B^\star,\Xi^\star}
\bigl(
\mathcal P_{x^\star,B^\star,\Xi^\star}^{*}\mathcal{H}^\star\mathcal P_{x^\star,B^\star,\Xi^\star}
+\nu\,\mathrm{Id}_{B^\star}
\bigr)^{-1}
\mathcal P_{x^\star,B^\star,\Xi^\star}^{*}
\Big]\,
(\mathcal{H}^\star)^{\frac{1}{2}}.
\]
As $\nu\to 0$, the operators $S_{\nu}$ increase monotonically in the Loewner order to
\[
S_0
=
(\mathcal{H}^\star)^{\frac{1}{2}}\,
\mathbb E\Big[
\mathcal P_{x^\star,B^\star,\Xi^\star}
\bigl(
\mathcal P_{x^\star,B^\star,\Xi^\star}^{*}\mathcal{H}^\star\mathcal P_{x^\star,B^\star,\Xi^\star}
\bigr)^{-1}
\mathcal P_{x^\star,B^\star,\Xi^\star}^{*}
\Big]\,
(\mathcal{H}^\star)^{\frac{1}{2}},
\]
and therefore $\lambda_{\min}(S_{\nu})$ converges to $\lambda_{\min}(S_0)=\zeta$ as $\nu \to 0$.
Fix $\delta\in(0,1)$. Choose $\tau,\phi,\gamma>0$ such that
\begin{equation}
\frac{1-\tau}{(1+\phi)^2(1+\gamma)}\ge 1-\delta.
\label{eq:local-parameter-choice}
\end{equation}
Next, choose $\nu>0$ small enough such that
\begin{equation}
\lambda_{\min}(S_{\nu})\ge (1-\tau)\zeta.
\label{eq:Smu-close-to-zeta}
\end{equation}
By \cref{prop:self-concordant}.1, after shrinking $V\subset U$ if necessary, for all $x\in V$,
\begin{equation}
(1+\phi)^{-1}\mathcal{H}^\star
\preceq
\widetilde{\mathcal H}(x)
\preceq
(1+\phi)\mathcal{H}^\star.
\label{eq:Hessian-comparison-pulledback}
\end{equation}
In particular, for $x=x_k$, we have
\begin{equation}
\widetilde{\mathcal P}_k^*\,\widetilde{\mathcal H}_k\,\widetilde{\mathcal P}_k
\preceq
(1+\phi)\,\widetilde{\mathcal P}_k^*\mathcal{H}^\star\,\widetilde{\mathcal P}_k.
\label{eq:reduced-Hessian-comparison}
\end{equation}
Moreover, since $\widetilde g(x^\star)=0$ and $x\mapsto \widetilde g(x)$ is continuous, we have $\widetilde h_k\to 0$ as $x_k\to x^\star$. Hence, after shrinking $V$ once more if needed, we may assume that whenever $x_k\in V$,
\begin{equation}
\tilde{\mu}_k
=
\sqrt{\sigma_k/2}\,\|\widetilde h_k\|_{x^\star}^{\frac{1}{2}}
\le
(1+\phi)\nu.
\label{eq:augmentation-small}
\end{equation}
Combining~\eqref{eq:reduced-Hessian-comparison} and~\eqref{eq:augmentation-small}, we obtain
\[
\widetilde{\mathcal J}_k+\tilde{\mu}_k\mathrm{Id}_{\widetilde B_k}
\preceq
(1+\phi)\bigl(\widetilde{\mathcal P}_k^*\mathcal{H}^\star\,\widetilde{\mathcal P}_k+\nu\,\mathrm{Id}_{\widetilde B_k}\bigr),
\]
and therefore
\begin{equation}
\bigl(\widetilde{\mathcal J}_k+\tilde{\mu}_k\mathrm{Id}_{\widetilde B_k}\bigr)^{-1}
\succeq
\frac{1}{1+\phi}
\bigl(\widetilde{\mathcal P}_k^*\mathcal{H}^\star\,\widetilde{\mathcal P}_k+\nu\,\mathrm{Id}_{\widetilde B_k}\bigr)^{-1}.
\label{eq:inverse-comparison}
\end{equation}
Substituting~\eqref{eq:inverse-comparison} into~\eqref{eq:descent-pulledback} and using $\widetilde h_k=\widetilde{\mathcal P}_k^*[\widetilde g_k]$ yields
\[
f(x_k)-f(x_{k+1})
\ge
\frac{1}{2(1+\phi)}
\Big\langle
\widetilde g_k,\
\widetilde{\mathcal P}_k
\bigl(\widetilde{\mathcal P}_k^*\mathcal{H}^\star\,\widetilde{\mathcal P}_k+\nu\,\mathrm{Id}_{\widetilde B_k}\bigr)^{-1}
\widetilde{\mathcal P}_k^*[\widetilde g_k]
\Big\rangle_{x^\star}.
\]
Taking conditional expectation and using \cref{ass:local-transported-resampling}, we obtain
\begin{align}
&\quad \mathbb E\!\left[f(x_k)-f(x_{k+1})\,\middle|\,x_k\right] \\
&\ge
\frac{1}{2(1+\phi)}
\Big\langle
\widetilde g_k,\
\mathbb E\Big[
\mathcal P_{x^\star,B^\star,\Xi^\star}
\bigl(
\mathcal P_{x^\star,B^\star,\Xi^\star}^{*}\mathcal{H}^\star\mathcal P_{x^\star,B^\star,\Xi^\star}
+\nu\,\mathrm{Id}_{B^\star}
\bigr)^{-1}
\mathcal P_{x^\star,B^\star,\Xi^\star}^{*}
\Big][\widetilde g_k]
\Big\rangle_{x^\star} \nonumber\\
&=
\frac{1}{2(1+\phi)}
\Big\langle
(\mathcal{H}^\star)^{-1/2}\widetilde g_k,\
S_{\nu}\,(\mathcal{H}^\star)^{-1/2}\widetilde g_k
\Big\rangle_{x^\star} \nonumber\\
&\ge
\frac{\lambda_{\min}(S_{\nu})}{2(1+\phi)}
\big\langle
\widetilde g_k,\,
(\mathcal{H}^\star)^{-1}[\widetilde g_k]
\big\rangle_{x^\star}.
\label{eq:expected-descent-pulledback}
\end{align}
Define
$
\Lambda_k
:=
\big\langle
\widetilde g_k,\,
\widetilde{\mathcal H}_k^{-1}[\widetilde g_k]
\big\rangle_{x^\star}^{\frac{1}{2}}.
$
Since $\mathcal T_{x_k\to x^\star}$ is an isometry, it holds that
\[
\Lambda_k^2
=
\big\langle
\grad f(x_k),\,
[\Hess f(x_k)]^{-1}[\grad f(x_k)]
\big\rangle_{x_k}.
\]
By \cref{prop:self-concordant}.2, after shrinking $V$ if necessary, we may assume
$
M_f\Lambda_k\le \frac{\gamma}{1+\gamma}\quad \text{for all } x_k\in V.
$
Hence it follows that
\[
f(x_k)-f(x^\star)
\le
\frac{1}{M_f^2}\omega_*(M_f\Lambda_k)
\le
\frac{1+\gamma}{2}\Lambda_k^2
=
\frac{1+\gamma}{2}
\big\langle
\widetilde g_k,\,
\widetilde{\mathcal H}_k^{-1}[\widetilde g_k]
\big\rangle_{x^\star}.
\]
Using~\eqref{eq:Hessian-comparison-pulledback}, we have
\[
\widetilde{\mathcal H}_k^{-1}\preceq (1+\phi)(\mathcal{H}^\star)^{-1},
\]
and therefore
\begin{equation}
f(x_k)-f(x^\star)
\le
\frac{(1+\gamma)(1+\phi)}{2}
\big\langle
\widetilde g_k,\,
(\mathcal{H}^\star)^{-1}[\widetilde g_k]
\big\rangle_{x^\star}.
\label{eq:gap-upper-bound-pulledback}
\end{equation}
Combining~\eqref{eq:expected-descent-pulledback},~\eqref{eq:gap-upper-bound-pulledback}, and~\eqref{eq:Smu-close-to-zeta}, we get
\[
\mathbb E\!\left[f(x_k)-f(x_{k+1})\,\middle|\,x_k\right]
\ge
\frac{(1-\tau)\zeta}{(1+\phi)^2(1+\gamma)}\,\bigl(f(x_k)-f(x^\star)\bigr).
\]
Equivalently,
\[
\mathbb E\!\left[f(x_{k+1})-f(x^\star)\,\middle|\,x_k\right]
\le
\Bigl(1-\frac{(1-\tau)\zeta}{(1+\phi)^2(1+\gamma)}\Bigr)\,\bigl(f(x_k)-f(x^\star)\bigr).
\]
By~\eqref{eq:local-parameter-choice},
\[
\frac{(1-\tau)\zeta}{(1+\phi)^2(1+\gamma)}\ge (1-\delta)\zeta,
\]
and thus
\[
\mathbb E\!\left[f(x_{k+1})-f(x^\star)\,\middle|\,x_k\right]
\le
\bigl(1-(1-\delta)\zeta\bigr)\,\bigl(f(x_k)-f(x^\star)\bigr).
\]
Finally, by shrinking $V$ if necessary, the one-step decrease implies that the iterates remain in $V$, and the stated linear rate follows by iterating the inequality.
\end{proof}

While the previous theorem establishes linear convergence, the rate $\zeta$ depends nontrivially on the Hessian spectrum and sketch quality. \Cref{thm:explicit-rate-nystrom} provides explicit bounds using Nystr\"om approximation theory, yielding concrete guarantees in terms of eigenvalue decay and sketch size.

\paragraph{Proof of \Cref{thm:explicit-rate-nystrom}}

\begin{proof}
    Following the proof of \cref{thm:riem-local-exp}, let
    \begin{equation}
    \label{eq:zeta-definition}
    \zeta\ :=\ \lambda_{\min}\!\Big((\mathcal{H}^\star)^{\frac{1}{2}}\,\mathbb E\big[\mathcal{P}_{x^\star, B^\star, \Xi^\star}\big(\mathcal{P}_{x^\star, B^\star, \Xi^\star}^{*}\mathcal{H}^\star \mathcal{P}_{x^\star, B^\star, \Xi^\star}\big)^{-1}\mathcal{P}_{x^\star, B^\star, \Xi^\star}^{*}\big]\,(\mathcal{H}^\star)^{\frac{1}{2}}\Big).
    \end{equation}
    For any $\nu>0$, define
    \[
    S_{\nu}\ :=\ (\mathcal{H}^\star)^{\frac{1}{2}}\,\mathbb E\Big[\mathcal{P}_{x^\star, B^\star, \Xi^\star}\big(\mathcal{P}_{x^\star, B^\star, \Xi^\star}^{*}\mathcal{H}^\star \mathcal{P}_{x^\star, B^\star, \Xi^\star}+\nu \,\mathrm{Id}_{B^\star}\big)^{-1}\mathcal{P}_{x^\star, B^\star, \Xi^\star}^{*}\Big]\,(\mathcal{H}^\star)^{\frac{1}{2}}.
    \]
    Since $A\mapsto A^{-1}$ is operator monotone decreasing on positive definite operators, for every $\nu>0$, it holds that
    \[
    \big(\mathcal{P}_{x^\star, B^\star, \Xi^\star}^{*}\mathcal{H}^\star \mathcal{P}_{x^\star, B^\star, \Xi^\star}+\nu\,\mathrm{Id}_{B^\star}\big)^{-1}
    \ \preceq\ 
    \big(\mathcal{P}_{x^\star, B^\star, \Xi^\star}^{*}\mathcal{H}^\star \mathcal{P}_{x^\star, B^\star, \Xi^\star}\big)^{-1}.
    \]
    Therefore $S_{\nu}\preceq (\mathcal{H}^\star)^{\frac{1}{2}}\,\mathbb E\big[\mathcal{P}_{x^\star, B^\star, \Xi^\star}(\mathcal{P}_{x^\star, B^\star, \Xi^\star}^{*}\mathcal{H}^\star \mathcal{P}_{x^\star, B^\star, \Xi^\star})^{-1}\mathcal{P}_{x^\star, B^\star, \Xi^\star}^{*}\big]\,(\mathcal{H}^\star)^{\frac{1}{2}}$, and hence
    \begin{equation}
    \label{eq:zeta-lower-by-Smu}
    \zeta\ \ge\ \lambda_{\min}(S_{\nu})\quad \text{for all } \nu>0.
    \end{equation}
    
    Write $\mathcal P=\mathcal P_{x^\star,B^\star,\Xi^\star}$ for brevity. For the isometric case $\mathcal R=\sqrt{d}\,\mathrm{Id}_B$, one verifies that $\mathcal P(\mathcal P^{*}\mathcal{H}^\star\mathcal P+\nu\,\mathrm{Id}_{B^\star})^{-1}\mathcal P^{*}=\Pi_{\Xi^\star}\big(\Pi_{\Xi^\star}\mathcal{H}^\star\Pi_{\Xi^\star}+\tfrac{\nu}{d}\,\mathrm{Id}_{\Xi^\star}\big)^{-1}\Pi_{\Xi^\star}$. Applying \cref{prop:manifold-42} with ridge parameter $\nu/d$ in place of $\nu$ and $\ell=2p-1$, together with Jensen's inequality for the convex map $A\mapsto \|A\|_{\mathrm{op}}$, yields
    \begin{align*}
    &\Big\|(\mathcal{H}^\star+\tfrac{\nu}{d}\,\mathrm{Id}_{x^\star})^{-1}
    -\mathbb E\big[\mathcal P(\mathcal P^{*}\mathcal{H}^\star \mathcal P+\nu\,\mathrm{Id}_{B^\star})^{-1}\mathcal P^{*}\big]\Big\|_{\mathrm{op}}\\
    &\le\ \mathbb E\Big\|(\mathcal{H}_{x^\star,B^\star,\Xi^\star}+\tfrac{\nu}{d}\,\mathrm{Id}_{x^\star})^{-1}
    -(\mathcal{H}^\star+\tfrac{\nu}{d}\,\mathrm{Id}_{x^\star})^{-1}\Big\|_{\mathrm{op}}
    \ \le\ C_p\,\frac{d\,\lambda_p(\mathcal{H}^\star)}{\nu\,(\lambda_d(\mathcal{H}^\star)+\nu/d)}.
    \end{align*}
    For general $\mathcal R$ the same argument applies with $\nu$ replaced by $\nu/\|\mathcal R\|_{\mathrm{op}}^{2}$ in the ridge parameter. Since the constants in the final bound depend on $\mathcal R$ only through $\|\mathcal R\|_{\mathrm{op}}$, we absorb these factors into $C_p$ for the remainder of the proof.
    Let $\varepsilon_{\nu}:=C_p\,\frac{d\,\lambda_p(\mathcal{H}^\star)}{\nu\,(\lambda_d(\mathcal{H}^\star)+\nu/d)}$. Since both operators are self-adjoint, the bound $\|X-Y\|_{\mathrm{op}}\le \varepsilon_{\nu}$ implies $Y\succeq X-\varepsilon_{\nu}\,\mathrm{Id}_{x^\star}$, and thus
    \[
    \mathbb E\big[\mathcal{P}_{x^\star, B^\star, \Xi^\star}(\mathcal{P}_{x^\star, B^\star, \Xi^\star}^{*}\mathcal{H}^\star \mathcal{P}_{x^\star, B^\star, \Xi^\star}+\nu\,\mathrm{Id}_{B^\star})^{-1}\mathcal{P}_{x^\star, B^\star, \Xi^\star}^{*}\big]
    \ \succeq\ (\mathcal{H}^\star+\tfrac{\nu}{d}\,\mathrm{Id}_{x^\star})^{-1}-\varepsilon_{\nu}\,\mathrm{Id}_{x^\star}.
    \]
    Multiplying both sides by $(\mathcal{H}^\star)^{\frac{1}{2}}$ yields
    \[
    S_{\nu}\ \succeq\ (\mathcal{H}^\star)^{\frac{1}{2}}(\mathcal{H}^\star+\tfrac{\nu}{d}\,\mathrm{Id}_{x^\star})^{-1}(\mathcal{H}^\star)^{\frac{1}{2}}-\varepsilon_{\nu}\,\mathcal{H}^\star
    \ =\ \mathcal{H}^\star(\mathcal{H}^\star+\tfrac{\nu}{d}\,\mathrm{Id}_{x^\star})^{-1}-\varepsilon_{\nu}\,\mathcal{H}^\star.
    \]
    Taking the minimum eigenvalue and using that the eigenvalues of $\mathcal{H}^\star(\mathcal{H}^\star+\tfrac{\nu}{d}\,\mathrm{Id}_{x^\star})^{-1}$ are $\lambda_j(\mathcal{H}^\star)/(\lambda_j(\mathcal{H}^\star)+\nu/d)$, we obtain
    \begin{equation}
    \label{eq:lminSmu}
    \lambda_{\min}(S_{\nu})\ \ge\ \frac{\lambda_d(\mathcal{H}^\star)}{\lambda_d(\mathcal{H}^\star)+\nu/d}-\varepsilon_{\nu}\,\lambda_1(\mathcal{H}^\star).
    \end{equation}
    Let $\nu=\theta\,\lambda_p(\mathcal{H}^\star)$ with $\theta>0$. Then~\eqref{eq:lminSmu} becomes
    \[
    \lambda_{\min}(S_{\nu})\ \ge\ \frac{\lambda_d(\mathcal{H}^\star)-\frac{d\,\lambda_1(\mathcal{H}^\star)}{\theta}\,C_p}{\lambda_d(\mathcal{H}^\star)+\theta\,\lambda_p(\mathcal{H}^\star)/d}.
    \]
    Choose $\theta:=2d\kappa C_p$ with $\kappa=\lambda_1(\mathcal{H}^\star)/\lambda_d(\mathcal{H}^\star)$. The numerator is then $\lambda_d(\mathcal{H}^\star)/2$ and the denominator is $\lambda_d(\mathcal{H}^\star)+2\kappa C_p\,\lambda_p(\mathcal{H}^\star)$, whence
    \[
    \lambda_{\min}(S_{\nu})\ \ge\ \frac{\lambda_d(\mathcal{H}^\star)/2}{\lambda_d(\mathcal{H}^\star)+2\kappa C_p\,\lambda_p(\mathcal{H}^\star)}
    \ =\ \frac{1}{2\bigl(1+2\kappa C_p\,\lambda_p(\mathcal{H}^\star)/\lambda_d(\mathcal{H}^\star)\bigr)}\ =:\ \overline{\zeta}.
    \]
    Combining this with~\eqref{eq:zeta-lower-by-Smu} yields $\zeta\ge \overline{\zeta}$.
    
    Finally, substitute $\zeta\ge \overline{\zeta}$ into the one-step contraction of \cref{thm:riem-local-exp}. For any $\delta\in(0,1)$ and all iterates in a sufficiently small neighborhood of $x^\star$,
    \[
    \mathbb E\big[f(x_{k+1})-f(x^\star)\mid x_k\big]\ \le\ \bigl(1-(1-\delta)\overline{\zeta}\bigr)\bigl(f(x_k)-f(x^\star)\bigr).
    \]
    Iterating the inequality gives the claimed linear rate and the complexity bound.
    \end{proof}

\bibliographystyle{siamplain}
\bibliography{ref}

\end{document}